\documentclass[final,12pt,twoside]{article}

\usepackage{a4}
\usepackage[english]{babel}                 
\usepackage{amsmath}                       
\usepackage{amssymb}                       
\usepackage{amsfonts}                      
\usepackage{enumerate}                     
\usepackage{amsthm}                        
\usepackage{epsf}
\usepackage[T1]{fontenc}
\usepackage[ansinew]{inputenc}

\newcommand{\RR}{\mathbb R}
\newcommand{\NN}{\mathbb N}

\newcommand{\cW}{\mathcal{W}}
\newcommand{\cM}{\mathcal{M}}
\newcommand{\eps}{\varepsilon}
\newcommand{\lam}{\lambda}
        
\newcommand{\abs}[1]{\left\vert#1\right\vert}
\newcommand{\norm}[1]{\left\vert\left\vert#1\right\vert\right\vert}
\newcommand{\dist}[2] {\operatorname{dist}\left(#1;#2\right)}
\newcommand{\mysetr}[2] {\left\{~#1~\left|~#2~\right.\right\}}
\newcommand{\mysetl}[2] {\left\{\left.~#1~\right|~#2~\right\}}

\newcommand{\loc}{\operatorname{loc}}

\newcommand{\identity}{\operatorname{id}}
\newcommand{\mbf}[1]{\boldsymbol{#1}}
\newcommand{\To}{\longrightarrow}

\renewcommand{\labelenumi}{(\roman{enumi})}
\newenvironment{myenum}{\vspace*{-2ex}\begin{enumerate}}{\end{enumerate}\vspace*{-2ex}}

\newtheorem{theorem}{Theorem}[section]

\newtheorem{thm}[theorem]{Theorem}

\newtheorem{cor}[theorem]{Corollary}
\newtheorem{lem}[theorem]{Lemma}
\newtheorem{prop}[theorem]{Proposition}

\theoremstyle{definition}
\newtheorem{defn}[theorem]{Definition}

\theoremstyle{remark}
\newtheorem{rem}[theorem]{Remark}
\renewenvironment{proof}[1][\proofname]{\par
  \normalfont
  \trivlist
  \item[\hskip\labelsep\itshape
    \bfseries{#1.}]\ignorespaces
}{%
  \qed\endtrivlist
}

\author{
Stefan Kr\"omer and Markus Lilli\\~\\ 
Institut für Mathematik\\
Universit\"at Augsburg\\
86135 Augsburg, Germany\\
stefan.kroemer@math.uni-augsburg.de,\\
lilli@math.uni-augsburg.de\\
}
\date{%
June 5, 2007
}
\title{On properness and related properties of quasilinear systems on unbounded domains\footnote{MSC 35G30, 35J55, 35A35, 49J45}}
\begin{document}
\maketitle
\begin{abstract}
The purpose of this paper is to provide tools for analyzing the compactness properties of sequences
in Sobolev spaces, in particular if the sequence
gets mapped onto a compact set by some nonlinear operator.
Here, our focus lies on a very general class of nonlinear operators arising in quasilinear systems of partial differential equations of second order, in divergence form.
Our approach, based on a suitable decomposition lemma, admits the discussion of problems with some inherent loss of compactness, for example due to a domain with infinite measure or a lower order term with critical growth. As an application, we obtain a characterization of properness which is considerably easier to verify
than the definition. The methods presented can also be used to check Palais--Smale conditions for variational problems.
\end{abstract}
\newpage
\tableofcontents
%
\section{Introduction}\label{sec:intro}
This work is motivated by the articles of Stuart and
Rabier \cite{RaStu01a} and Gebran and Stuart \cite{GeStu05a}. In the former, elliptic equations on the whole space are studied, and the approach is generalized to quasilinear elliptic systems on exterior domains in the latter.
Both papers focus on elliptic quasilinear differential operators of second order mapping $W^{2,p}$ into $L^p$, where $p>N$. By compact embedding, this choice of spaces entails that perturbations which only contain derivatives up to first order are "locally" compact, i.e., compact if restricted to a bounded subdomain of $\Omega$. 
As a consequence, properness for such an operator restricted to a bounded subset of $W^{2,p}(\Omega)$ on a bounded regular domain $\Omega$ can be obtained using well known a-priori estimates for linear systems with continuous coefficients. If, on the other hand, $\Omega$ is unbounded, one has to deal with the possibility that mass might
escape into the outer regions of $\Omega$ and vanish in the limit, a lack of compactness otherwise not present. 
The results in \cite{RaStu01a} and \cite{GeStu05a} provide conditions to rule out the aforementioned 
behavior. However, to apply those results (for example as described in \cite{RaStu01b}, to carry out a global continuation argument along a real parameter using a suitable degree), one has to face a serious obstacle inherent in the choice of spaces: Usually, a-priori estimates in the corresponding space are required (to prevent a blow-up of a continuum of solutions prior to reaching any parameter value of interest, for instance). Even if the leading part is of of divergence form and the equation considered admits a simple a-priori estimate in its weak setting (in $W^{1,q}$, e.g., if the leading operator is of $q$-Laplace type), this estimate cannot always be lifted to $W^{2,p}$. In particular, this is a problem in the case of quasilinear elliptic systems, where regularity theory is available only in special situations and known to fail in general (for an overview see \cite{Gia83B}, e.g.). 
Hence it seems expedient to derive a characterization of properness applicable in the "natural" weak setting of the equation, which is the content of Section~\ref{sec:app1}. An important new difficulty which arises is that perturbations containing first order derivatives are no longer locally compact. 
  
Technically, our approach differs from that of \cite{RaStu01a}. Instead, we extend ideas employed in \cite{FoMuePe98a} (see also \cite{GraMe06a}) for problems on bounded domains. 
A key observation is the fact that on a bounded domain $\Omega$, any given bounded sequence $u_n$ in $W^{1,p}(\Omega)$ can be decomposed into a sum of two sequences (say, $u_n=v_n+w_n$) in such a way that $\abs{\nabla v_n}^p+\abs{v_n}^p$ is equiintegrable and $w_n$ converges to zero in measure ("decomposition lemma", cf.~Lemma 1.2 in
\cite{FoMuePe98a}). This
means that the two qualitatively different types of noncompactness in $W^{1,p}(\Omega)$ can be separated:
$v_n$ does not concentrate and $w_n$ does not oscillate whereas $u_n$ might do both. (Here, recall that if a sequence converges in measure and it is equiintegrable as above then it is strongly convergent due to Vitali's Theorem.) 
To use the decomposition lemma to check properness or a Palais--Smale condition, a second ingredient
is needed: As already observed in \cite{FoMuePe98a}, nonlinear functionals satisfying a suitable local Lipschitz condition behave asymptotically additive in the limit $n\to \infty$ with respect to such a decomposition $u_n=v_n+w_n$ ("orthogonality principle", cf.~Lemma 2 in \cite{GraMe06a}). 
We generalize these arguments in the following ways: First of all, recall that if $\Omega$ has infinite measure, other types of divergent bounded sequences are possible. The two typical examples for such sequences $u_n$ in $W_0^{1,p}(\Omega)$ are "traveling bulks of mass" (i.e, $\int_{B_1(y_n)\cap\Omega} \abs{u_n}^p
\not\to 0$ for a sequence of points $y_n\in\Omega$ with $\abs{y_n}\to\infty$) and vanishing (in the sense of P.-L.~Lions \cite{Li84a}, for instance a sequence such that $u_n\to 0$ in $W^{1,\infty}$ and $u_n\not\to 0$
in $W^{1,p}$). 
Of course, the sequence of gradients might do the same.
One major aim of this paper is to obtain a suitable extension of the decomposition lemma which allows us to deal with this kind of behavior as well, by decomposing into more than two sequences (actually, we end up with five, cf.~Lemma~\ref{lem:dec}), while the orthogonality principle remains valid
(cf.~Theorem~\ref{lem:W1p-prop}). Similar as in \cite{FoMuePe98a}, the decomposition is obtained by truncating the original sequence in a suitable way, thus splitting the truncated part and its remainder. However, we use three different ways of truncating, adapted to the presence of an unbounded domain: 
The first type of truncation cuts off unbounded parts of $\Omega$ via multiplication with suitable smooth functions which vanish outside some ball.
Second, we truncate gradients above large levels. 
Note that this is not a trivial operation because a na\"{i}ve approach would destroy gradient structure.
Our definition is based on Theorem~\ref{thm:hcutoffOmega}, which employs arguments involving maximal operators and the extension of Lipschitz functions.
It is similar to the one used in \cite{FoMuePe98a} and \cite{DoHuMue00a}, respectively, apart from the fact that we truncate in a way which preserves Dirichlet boundary conditions while avoiding the assumption that the domain is bounded.
The third method of truncation is again based on the aforementioned truncation of gradients, but now we truncate at small levels, to remove parts of a function which are uniformly small in $W^{1,\infty}$ but at the same time "spread out" in $\Omega$ in a way which prevents that the their norm in $W^{1,p}$ is small.
In place of the orthogonality principle for real--valued functionals, we derive an analogue 
which is valid for a large class of nonlinear operators $F:X\to Y$ between two Banach spaces. The main class of examples are perturbed quasilinear differential operators of second order in divergence form mapping $W_0^{1,p}\cap W_0^{1,q}$ into its dual space. The key assumption we employ in that context is that $F$ is uniformly continuous on bounded subsets of $X$.
If we assume that $F(u_n)$ converges in $Y$, it turns out that the images of each of the components of the decomposition of $u_n$ converge (up to a subsequence). This in fact entails a characterization of properness:
To show properness of $F$ (on closed bounded sets), one has to show that every bounded sequence $u_n$ such that $F(u_n)$ converges has a convergent subsequence. In view of the results above, we now may assume in addition
that $u_n$ is a sequence of one of the types encountered in the decomposition lemma, each of which is carrying at most one of the types of noncompactness mentioned above. As a consequence, obtaining the existence of a convergent subsequence becomes considerably easier.

In contrast to the concentration-compactness lemmas \cite{Li84a}, \cite{Li85a} and the main result 
of \cite{FoMuePe98a}, we do not rely on limiting notions describing the lack of compactness, such as suitable measures capturing concentration effects. In particular, we avoid related assumptions such as continuity in the independent variable and homogeneity of the terms with critical growth, as well as the associated loss of information.
Moreover, perturbation terms containing derivatives can be treated. 
We also mention that a specialized type of decomposition lemma is well known
as a method to verify a Palais-Smale condition for various semilinear elliptic variational problems, its prototype being \cite{Stru84a}. An indroduction to this topic and some applications are
given in \cite{Stru00B}, for further versions and applications the reader is referred to
\cite{Stru00B}, \cite{Li88a} (Theorem III.4), \cite{CaMa02a} (Lemma 2.9), \cite{Wi03a} and \cite{AdStru00a} (e.g.). Results concerning the Palais-Smale condition for variational problems of more general form can be found in \cite{DraHua99a}, \cite{DraStaZo03a} and \cite{Squa06B}. 
Apart from the aforementioned papers \cite{RaStu01a} and \cite{GeStu05a}, properness of nonlinear operators
arising in context of elliptic equations and systems without variational structure (of general form) has been discussed in \cite{VoVo02a} and \cite{VoVo03a} with the focus on unbounded domains, and in
\cite{Zvy90a} and \cite{ZvyDmi92a} for bounded domains. 

This paper is structured as follows: Section~\ref{sec:abstractproper} contains an abstract framework which allows us to derive suitable abstract versions of the orthogonality principle (Theorem~\ref{thm:op}) and the associated characterization of properness (Theorem~\ref{thm:abstractprop}), assuming that a decomposition lemma is valid (axiom \eqref{phi3}). It works with an abstract notion of truncation, a common denominator of all three methods of truncation mentioned above. These are defined and discussed in Section 3, 
culminating in an associated decomposition lemma (Lemma~\ref{lem:dec}). In the remaining sections, the abstract results are applied. In particular, in Section~\ref{sec:app1} we obtain a characterization of properness for quasilinear operators of second order in divergence form (not necessarily elliptic) in the weak setting (Theorem~\ref{thm:W1p-prop}), based on our version of the decomposition lemma
and an associated (asymptotical) decomposition of the nonlinear operator (Theorem~\ref{lem:W1p-prop}). 
Our methods can also be used to verify the Palais--Smale condition in variational context, which of course is closely related to properness of the Fréchet derivative of the energy. In Section~\ref{sec:app1b}, we provide another tool for that purpose which supplements the results of Theorem~\ref{lem:W1p-prop}, namely an orthogonality principle for energies or integral constraints, respectively (Theorem~\ref{lem:energyop}).
In the final section, we revisit the setting discussed in \cite{RaStu01a} and \cite{GeStu05a}. In that context, our method yields a direct proof of the equivalence of properness (on closed bounded sets) and "properness at 0" (Corollary~\ref{cor:2Fprop}, cf.~Theorem 6.5 and Theorem 7.9 in \cite{RaStu01a} and Theorem 5.7 in \cite{GeStu05a}) which avoids the use of limit problems (and the associated assumptions concerning the asymptotic behavior of the coefficient functions and the domain), thus answering a question raised in \cite{RaStu01a}. 


%

%



\subsection{Notation and preliminaries}
As usual, $W^{k,p}(\Omega;V)$ is the Sobolev space of functions $u:\Omega\to V$ with distributional derivatives up to order $k$ in $L^p$, where $\Omega\subset \RR^N$ is a domain 
(i.e., open and connected) and $V$ is some finite dimensional euclidean vector space. 
The closed subspace $W_0^{k,p}(\Omega;V)$ consists of the closure of $C_0^\infty(\Omega)$,
the smooth functions with compact support, in $W^{k,p}$.
Norms of infinite dimensional spaces are denoted by $\norm{\cdot}_X$, where the corresponding space $X$ is given in the index, whereas finite dimensional norms are denoted by $\abs{\cdot}$, as is the real modulus. 
Moreover, if $A\subset \RR^N$ is a measurable set, $\abs{A}$ is its Lebesgue measure.
The letter $I$ always stands for the identity map on a set which should be clear from the context.
For super-level sets of a function $f:D\to \RR$, we sometimes use the abbreviated notation
$\{ f\geq h \}:=\{x\in D\mid f(x)\geq h\}$, where of course the inequality sign can be exchanged to 
define a (strict) (sub/super-) level set instead.
We also recall the following property of Sobolev functions, which will be used without further reference:
\begin{lem}[Lemma 7.7 in \cite{GiTru83B}, e.g.]
Let $\Omega\subset \RR^N$ be an arbitrary domain, $p\in [1,\infty]$ and $v\in W^{1,p}(\Omega)$. If $D\subset \Omega$ is measurable and $v(x)=0$ for a.e. $x\in D$, then also $\nabla v(x)=0$ for a.e.~$x\in D$.
\end{lem}

\section{An abstract characterization of properness}\label{sec:abstractproper}
Let $X$, $Y$ be a normed vector spaces. 
We consider nonlinear operators of the following type:
\begin{alignat*}{1}
	&\begin{array}{l}
		F:D\to Y~\text{is a continuous function,}\\
		\text{defined on a closed additive subgroup $D$ of $X$.}
	\end{array}\label{F0}\tag{F:0}
\end{alignat*}
\begin{rem}\label{rem:0inD}
The case of an affine subspace $D$ such that $0\notin D$ 
can be recovered as follows: 
For an arbitrary but fixed $x_0\in D$ consider $\tilde{F}:\tilde{D}\to Y$, 
$\tilde{F}(x):=F(x+x_0)$, instead of $F$, where $0\in \tilde{D}:=\{x-x_0\mid x\in D\}$. 
\end{rem}
We study properness only on bounded subsets of $X$.
\begin{defn}[Properness]
	The function $F$ is called \emph{proper} (on closed bounded subsets of $D$) if 
	\begin{align*}
		& \text{every {bounded} sequence $(u_n)\subset D$ such that $F(u_n)$ converges in $Y$}\\
		& \text{has a subsequence which converges in $X$}.
	\end{align*}
\end{defn}
Below, truncation techniques play a major role. All different types of truncation employed in the applications fit into the following abstract framework.  
\begin{defn}[Truncation operators]\label{def:cutoff}~\\
Let $X$ be a normed vector space, $D\subset X$ and let 
$\phi_n:D\to D$, $n\in \NN$, be a sequence of maps. 
We call $\phi_n$ a \emph{family of truncation operators} on $D\subset X$ if it 
satisfies \eqref{phi1} and \eqref{phi3} below. 
That is,
the maps $\phi_n$ are \emph{equibounded} in the sense that
\begin{align*}
& \begin{aligned}
	{}&\norm{\phi_n(u)}_X\leq C\norm{u}_X~\text{for every $u\in D$ and $n\in \NN$}\\
	{}&\text{with a constant $C\geq 0$ independent of $u$ and $n$,}
\end{aligned}\label{phi1}\tag{$\phi$:1}
\intertext{and}
& \begin{aligned}
	{}&\text{every bounded sequence $(u_n)\subset D$}\\
	{}&\text{has a subsequence $(u_{k(n)})$ such that \eqref{uknequiint} holds.}
	\end{aligned}\label{phi3}\tag{$\phi$:2}
\end{align*}	
Here, the latter means that
\begin{align}
& \begin{aligned}
	{}&\left(\phi_n-\phi_{j(n)}\right)(u_{k(n)})\underset{n\to \infty}{\To} 0
	~\text{for every sequence $(j(n))\subset \NN$}\\
	{}&\text{such that $j(n)\underset{n\to\infty}{\To}\infty$ and $j(n)<n$ for every $n\in \NN$.}
\end{aligned}\label{uknequiint}
\end{align}
\end{defn}
\begin{rem}\label{remcutoffdef}~
\begin{myenum}
\item Note that $\phi_n$ does not have to be linear or continuous. As a matter of fact, 
in case of our second and third example below (truncation of gradients) the truncation operators 
are nonlinear and continuity is not clear, cf.~Remark~\ref{rem:truncgradvec}.
\item Instead of \eqref{phi1}, it would actually suffice to have that $\bigcup_{n\in\NN}\phi_n(W)$
is bounded in $X$, for every bounded set $W\subset D$. 
We still use \eqref{phi1} because it is notationally convenient, and in all of our examples below,
the family $\phi_n$ is linearly equibounded as required in \eqref{phi1}, anyway.
\item Axiom \eqref{phi3} is satisfied by each of the three types of truncation operators
introduced in Section~\ref{secCOOpExamples}.
The property of the subsequence in \eqref{phi3} 
can be characterized as follows:
\begin{align}
& \begin{aligned}
	{}&\left(\phi_n-\phi_{j(n)}\right)(u_{n})\underset{n\to\infty}{\To} 0
	~\text{for every sequence $(j(n))\subset \NN$}\\
	{}&\text{such that $j(n)\underset{n\to\infty}{\To}\infty$ 
	and $j(n)<n$ for every $n\in \NN$}
\end{aligned}\label{coequiintA}
\intertext{if and only if}
& \begin{aligned}
	{}&\text{for every $\eps>0$ there exists $j_0=j_0(\eps)\in \NN$ such that}\\
	{}&
	\norm{\left(\phi_n-\phi_{j}\right)(u_{n})}_X< \eps~\text{for every $n,j\in \NN$ with
	$n>j\geq j_0$}.
\end{aligned}\label{coequiintB}
\end{align}
Roughly speaking, the "tail" $\left(\phi_n-\phi_{j}\right)(u_{n})$ of $\phi_n(u_n)$ starting "at height $j$" becomes small as $j\to\infty$, uniformly in $n$. For instance, in case of truncation of gradients, \eqref{coequiintA} means that that the sequence $\phi_n(u_n)$ does not concentrate in $W^{1,p}$ (in the sense of Definition~\ref{def:concentrate}), cf.~Proposition~\ref{prop:coheight2}.
\item If $\phi_n$, $\tilde{\phi}_n$ are families of truncation operators on $D\subset X$ and $
\tilde{D}\subset \tilde{X}$, respectively, such that $\phi_n=\tilde{\phi}_n$ on $\hat{D}:=D\cap \tilde{D}$,
then the restrictions $\hat{\phi}_n:=\phi_n|_{\hat D}$ are a family of truncation operators on
$\hat{D}\subset \hat{X}:=X\cap \tilde{X}$ (with norm $\norm{\cdot}_{\hat{X}}=\norm{\cdot}_{X}+
\norm{\cdot}_{\tilde{X}})$. Moreover, if $Z$ is a closed subspace of $X$ which is invariant
under $\phi_n$ for every $n$, then 
$\hat{\phi}_n:=\phi_n|_{\hat D}$ is a family of truncation operators on $\hat D:=D\cap Z\subset Z$.
\item For the arguments employed in the 
present section, \eqref{phi1} suffices as assumption on $\phi_n$. 
However, Theorem~\ref{thm:op}
and Theorem~\ref{thm:abstractprop} 
only apply to sequences $u_n$ with a subsequence $u_{k(n)}$ satisfying \eqref{uknequiint}.
Hence we also require \eqref{phi3}.
\end{myenum}
\end{rem}
For technical reasons, we also want to truncate elements of the image of $F$.
Roughly speaking, we need that this outer truncation has a similar effect
as truncating the argument of $F$. The precise requirement is the following:
\begin{defn}[Compatibility]\label{def:compat}~\\
Suppose that $F:D\to Y$ is a function on $D\subset X$, $\phi_n$ is a family of truncation operators on $X$
and $\psi_n:R\to Y$ is a family of maps defined on a set $R\subset Y$
which contains the range of $F$.
We say that $\phi_n$ is \emph{compatible} to $\psi_n$ 
with respect to $F$ if 
for all bounded subsets $W\subset D$,
\begin{align}
	&\sup_{v,w\in W} \norm{\psi_m\left[F(v+(I-\phi_n)(w))\right]-\psi_m\left[F(v)\right]}_Y
	\underset{n\to\infty}{\To} 0~~\text{and}\label{compat1}\\
	&\sup_{v,w\in W} \norm{(I-\psi_n)\left[F(v+\phi_m(w))\right]-(I-\psi_n)\left[F(v)\right]}_Y
	\underset{n\to\infty}{\To} 0,\label{compat2}
\end{align}
where $m\in \NN$ is arbitrary but fixed. 
\end{defn}
Of course, merely being a sequence of functions is not yet enough structure
for the family $\psi_n:R\to Y$.
In addition, we assume that
the family $\psi_n$ is equicontinuous, uniformly on bounded subsets of $R$, i.e.,
\begin{alignat*}{1}
{}&\begin{array}{l}
    \sup_{n\in\NN}\sup_{w_1,w_2\in W,~\norm{w_1-w_2}_Y<\delta}
    \norm{\psi_n(w_1)-\psi_n(w_2)}_Y\underset{\delta\searrow 0}{\To}0,\\
    \text{for every bounded subset $W$ of $R$.}
  \end{array}
  \label{copsi1}\tag{$\psi$:1}
\intertext{In Theorem~\ref{thm:abstractprop} below, we also require that}
  {}&\begin{array}{l}
    \text{the sequence $(\psi_n(y))_{n\in\NN}$ converges in $Y$, for every fixed $y\in R$.}
  \end{array}
  \label{copsi2}\tag{$\psi$:2}  
\end{alignat*} 
%
We now list the remaining assumptions on $F$.
For a given family $\phi_n$ of truncation operators on $D\subset X$, suppose that
 \begin{alignat*}{1}
  &\begin{array}{l}
  	\text{there exists a set $R\subset Y$ which contains the closure of the range of $F$}\\
  	\text{and a family of maps $\psi_n:R\to Y$ ($n\in \NN$) such that \eqref{copsi1} holds}\\
  	\text{and such that $\phi_n$ and $\psi_n$ are compatible with respect to $F$.}\\
	\end{array}\label{F1}\tag{F:1}
\intertext{Last but not least, we assume uniform continuity of $F$ on bounded subsets of $D$, 
at least up to a perturbation as follows:}
  &\begin{array}{l}
		\text{There is a function $F_1:D\to R$, uniformly continuous on bounded}\\
		\text{subsets of $D\subset X$, such that for every bounded subset $W\subset D$,}\\
		(I-\psi_n)[F(w)]-(I-\psi_n)[F_1(w)]\underset{n\to\infty}{\To} 0~~\text{in $Y$, uniformly in $w\in W$.}
	\end{array}\label{F2}\tag{F:2}
\end{alignat*} 
\begin{rem}\label{rem:F12}~
As to the role of $R$, note that it is convenient to allow sets larger than the range of $F$
or its closure. Otherwise, the possible choices of $F_1$ in \eqref{F2} would be restricted too much. 
In many cases, one may use $R:=Y$ or a suitable closed subspace. 
\end{rem}
%
The main results of this section are subsumed in the following theorems. The first one is an abstract generalization of Lemma 2 in \cite{GraMe06a} ("orthogonality priciple").
\begin{thm}[Asymptotical additivity of $F$]
\label{thm:op}
Let $\phi_n:D\to D$ be an equibounded family in the sense of \eqref{phi1}.
Moreover, assume that $F$ satisfies \eqref{F0}--\eqref{F2} 
and let $(u_n)\subset D$ be a bounded sequence. Then
\begin{align}\label{Forthop}
	F(u_{k(n)})-F(\phi_{n}(u_{k(n)}))+F(0)-F((I-\phi_{n})(u_{k(n)}))\underset{n\to \infty}{\To} 0,
\end{align}
for every subsequence $u_{k(n)}$ which satisfies \eqref{uknequiint}.
\end{thm}
To check properness of $F$, we only consider sequences $(u_n)$ such that $F(u_n)$ converges in $Y$. 
In this case, even more can be said.
\begin{thm}[Abstract characterization of properness]\label{thm:abstractprop}
Let $\phi_n:D\to D$ be an equibounded family in the sense of \eqref{phi1}
and assume that \eqref{F0}--\eqref{F2} hold with a family $\psi_n$ 
which also satisfies \eqref{copsi2}. Moreover,
let $(u_{n})\subset D$ be a bounded sequence
such that the limit $G:=\lim_{n\to\infty}F(u_n)$ exists in $Y$.
Then we have that
\begin{align*}
	F(\phi_n(u_{k(n)}))\to G+H~~\text{and}~~F\left((I-\phi_n)(u_{k(n)})\right)\to F(0)-H~~\text{in $Y$}
\end{align*}
as $n\to\infty$, for any subsequence $(u_{k(n)})$ of $(u_n)$ which satisfies \eqref{uknequiint}. 
Here, 
\begin{align*} 
	H:=\lim_{n\to \infty} \left[(I-\psi_n)\left(F(0)\right)-(I-\psi_n)(G)\right]\in Y.
\end{align*}	
\end{thm}
\begin{rem}
Theorem~\ref{thm:abstractprop} tells us that for the purpose of showing properness of $F$ 
(on closed bounded sets), it is enough to study bounded sequences with special properties, namely 
sequences of the type $\phi_n(u_{k(n)})$ such that \eqref{uknequiint} holds and sequences
of "tails" $(I-\phi_n)(u_{k(n)})$.
Also note that $H=0$ if $\psi_n(y)\to y$ for every fixed $y\in R$.
\end{rem}
\subsubsection*{Proofs of Theorem~\ref{thm:op} and Theorem~\ref{thm:abstractprop}}
We first collect a few basic consequences of~\eqref{compat1} and~\eqref{compat2}.
\begin{prop}\label{prop:compat1}
Suppose that the assumptions of Theorem~\ref{thm:op} are satisfied and 
let $W$ be a bounded subset of $D\subset X$.
Then there exists a sequence $(m(n))\subset \NN$ ($m(n)\geq n$), such that for every sequence $(h(n))\subset \NN$ with $h(n)\geq m(n)$ for every $n\in\NN$,
\begin{align} 
	&\sup_{w\in W}\norm{\psi_{n}\left[F(w)\right]-\psi_{n}\left[F(\phi_{h(n)}(w))\right]}_Y
	\underset{n\to\infty}{\To} 0\quad \text{and}\label{psi1}\\
	&\sup_{w\in W}\norm{\psi_{n}\left[F((I-\phi_{h(n)})(w))\right]-\psi_{n}\left[F(0)\right]}_Y
	\underset{n\to\infty}{\To} 0.\label{psi2}
\end{align}
\end{prop}
\begin{proof}
Let $U:=W \cup \{0\}\cup \bigcup_{m\in\NN}\phi_m(W)$, which is a bounded set in $X$
since $\phi_n$ is equibounded.
By virtue of~\eqref{compat1}, we can choose a strictly increasing sequence $(m(n))\subset \NN$ 
such that for every sequence $h(n)\geq m(n)$,
\begin{align} \label{psi3}
\begin{aligned}
	\sup_{v,w\in U}\norm{\psi_{n}\left[F(v+(I-\phi_{h(n)})(w))\right]
	-\psi_{n}\left[F(v)\right]}_Y\leq \frac{1}{n}.
\end{aligned}	
\end{align}
With $v=\phi_{h(n)}(w)$, this entails \eqref{psi1}, whereas with $v=0$, we get~\eqref{psi2}.
\end{proof}
\begin{prop} \label{prop:compat2}
Suppose that the assumptions of Theorem~\ref{thm:op} are satisfied.
Then for every bounded sequence $(u_n)\subset X$ which satisfies \eqref{coequiintA},
\begin{align}\label{lemcomp21}
	&(I-\psi_{n})\left[F\left(0\right)\right]
	-(I-\psi_{n})\left[F\left(\phi_{h(n)}(u_{h(n)})\right)\right]	
	\underset{n\to \infty}{\To}0\quad\text{and}\\
	&(I-\psi_{n})\left[F\left(u_{h(n)}\right)\right]
	-(I-\psi_{n})\left[F\left((I-\phi_{h(n)})(u_{h(n)})\right)\right]
	\underset{n\to \infty}{\To}0,\label{lemcomp22}
\end{align}
for every sequence $h(n)\geq n$.
\end{prop}
\begin{proof}
As a consequence of~\eqref{compat2}, we can choose a sequence $\tilde{m}(n)<n$  
with $\tilde{m}(n)\to \infty$ (slow enough) such that 
\begin{equation}
\begin{aligned}
	(I-\psi_{n})\left[F(v+\phi_{j(n)}(w))\right]-(I-\psi_{n})\left[F(v)\right] 
	\underset{n\to\infty}{\To} 0~~\text{in $Y$,}&\\ 
	\text{uniformly in}~v,w\in W:=\{0\}\cup\bigcup_{m,n\in \NN}\left\{u_n,\phi_m(u_n)\right\}&,
\end{aligned}\label{lemcomp23}
\end{equation}
for every sequence $j(n)\leq \tilde{m}(n)$. Here, note that $W$ is bounded since $(u_n)$ is bounded and the $\phi_n$ are equibounded. Moreover,
\begin{align*}
	(I-\phi_{j(n)})(u_{h(n)})-(I-\phi_{h(n)})(u_{h(n)})=(\phi_{h(n)}-\phi_{j(n)})(u_{h(n)})
	\underset{n\to \infty}{\To} 0
\end{align*}
as long as $h(n)\geq n>j(n)$ for all $n$ and $j(n)\to \infty$, due to \eqref{coequiintA}. Consequently,
\begin{equation} \label{lemcomp24}
\begin{aligned} 
	&(I-\psi_{n})\left[F\left((I-\phi_{j(n)})(u_{h(n)})\right)\right] \\&\quad\qquad
	-(I-\psi_{n})\left[F\left((I-\phi_{h(n)})(u_{h(n)})\right)\right]
	\underset{n\to \infty}{\To}0
\end{aligned}
\end{equation}
and
\begin{equation} \label{lemcomp24b}
\begin{aligned} 
	&(I-\psi_{n})\left[F\left(\phi_{j(n)}(u_{h(n)})\right)\right] \\&\quad\qquad
	-(I-\psi_{n})\left[F\left(\phi_{h(n)}(u_{h(n)})\right)\right]
	\underset{n\to \infty}{\To}0.
\end{aligned}
\end{equation}
If $F=F_1$, \eqref{lemcomp24} and \eqref{lemcomp24b} are due to the uniform continuity of $F_1$ on bounded sets
(also recall that the $\psi_n$ are uniformly equicontinuous on bounded subsets of $Y$). This argument also
yields the general case since the remainder is negligible in the limit by \eqref{F2}.
Assertion \eqref{lemcomp21} now is an immediate consequence of \eqref{lemcomp23} (with $v=0$, $w=u_{h(n)}$) and \eqref{lemcomp24b}. Assertion \eqref{lemcomp22} follows from \eqref{lemcomp23} (with $v=(I-\phi_{j(n)})(u_{h(n)})$, $w=u_{h(n)}$) and \eqref{lemcomp24}.
\end{proof}
\begin{proof}[Proof of Theorem~\ref{thm:op}]
For simplicity, assume (w.l.o.g.) that $u_{k(n)}=u_n$, whence
\eqref{uknequiint} turns into \eqref{coequiintA}.
Let $m(n)$ be the subsequence of $n$ obtained in Proposition~\ref{prop:compat1} and choose an arbitrary sequence $h(n)\geq m(n)$. We first claim that
\begin{align}\label{Forthop2}
	F(u_{h(n)})-F(\phi_{h(n)}(u_{h(n)}))+F(0)-F((I-\phi_{h(n)})(u_{h(n)}))\underset{n\to\infty}{\To} 0.
\end{align}
Since the left hand side of~\eqref{Forthop2} is the sum of the terms listed below, it is enough to show that
\begin{align*}
	&\psi_{n}\left[F\left(u_{h(n)}\right)\right]
	-\psi_{n}\left[F\left(\phi_{h(n)}(u_{h(n)})\right)\right]
	\underset{n\to\infty}{\To}0,\\
	&\psi_{n}\left[F\left(0\right)\right]
	-\psi_{n}\left[F\left((I-\phi_{h(n)})(u_{h(n))})\right)\right]
	\underset{n\to\infty}{\To}0,\\
	&(I-\psi_{n})\left[F\left(0\right)\right]
	-(I-\psi_{n})\left[F\left(\phi_{h(n)}(u_{h(n))})\right)\right]
	\underset{n\to\infty}{\To}0~\text{and}\\
	&(I-\psi_{n})\left[F\left(u_{h(n))}\right)\right]
	-(I-\psi_{n})\left[F\left((I-\phi_{h(n)})(u_{h(n))})\right)\right]
	\underset{n\to\infty}{\To}0.
\end{align*}
Here, the first two lines follow from Proposition~\ref{prop:compat1}, whereas the last two are a consequence of Proposition~\ref{prop:compat2}. The same argument yields that every subsequence of $n$ has a subsequence $h(n)$ 
such that~\eqref{Forthop2} holds, which implies convergence of the whole sequence as asserted in~\eqref{Forthop}.
\end{proof}
For the proof of Theorem~\ref{thm:abstractprop}, we need one additional ingredient.
\begin{prop}\label{prop:Fconvcutoff}
Suppose that the assumptions of Theorem~\ref{thm:abstractprop} are satisfied.
Then
\begin{align*}
	\text{$F\left(\phi_n(u_{k(n)})\right)\to G+H$ in $Y$ as $n\to \infty$,}
\end{align*}
for every subsequence $u_{k(n)}$ of $u_n$ 
which satisfies \eqref{uknequiint}.
\end{prop}
\begin{proof}
Recall that
$H=\lim_{n\to \infty} \left[(I-\psi_n)\left(F(0)\right)-(I-\psi_n)(G)\right]\in Y$.
For simplicity, assume (w.l.o.g.) that $u_{k(n)}=u_n$, whence \eqref{coequiintA}
replaces \eqref{uknequiint}. 
It is enough to show that every subsequence of $n$ (not relabeled) has another subsequence $h(n)$ such that
\begin{align}\label{lemFcu1}
	F\left(u_{h(n)}\right)-F\left(\phi_{h(n)}(u_{h(n)})\right)+H
	\underset{n\to \infty}{\To}0
\end{align}
in $Y$.
We claim that \eqref{lemFcu1} is valid whenever $h(n)\geq m(n)$ for every $n$, where $m(n)$ is the subsequence of $n$ obtained in Proposition~\ref{prop:compat1}. 
Since we can decompose the left hand side of \eqref{lemFcu1} accordingly,
\eqref{lemFcu1} follows once we show that
\begin{align*}
	&\psi_n\left[F\left(u_{h(n)}\right)\right]-\psi_n\left[F\left(\phi_{h(n)}(u_{h(n)})\right)\right]
	\underset{n\to\infty}{\To}0,\\
	&(I-\psi_n)\left[F(0)\right]-(I-\psi_n)\left[F\left(\phi_{h(n)}(u_{h(n)})\right)\right]
	\underset{n\to\infty}{\To}0,\\
	&(I-\psi_n)\left[F\left(u_{h(n)}\right)\right]-(I-\psi_n)(G)
	\underset{n\to\infty}{\To}0~~\text{and}\\
	&(I-\psi_n)(G)-(I-\psi_n)\left[F(0)\right]+H
	\underset{n\to\infty}{\To}0.
\end{align*}
The first two lines hold due to Proposition~\ref{prop:compat1} and Proposition~\ref{prop:compat2}, respectively. 
The term on the left hand side of the third line converges to zero since $F(u_{h(n)})\to G$ and the $\psi_n$ are equicontinuous at $G$, and 
the term in the last line does the same due to our choice of $H$. 
\end{proof}
\begin{proof}[Proof of Theorem~\ref{thm:abstractprop}]
Combining Theorem~\ref{thm:op} and Proposition~\ref{prop:Fconvcutoff}, the assertion
follows immediately.
\end{proof}
%
\section{Examples for families of truncation operators\label{secCOOpExamples}}
We first introduce some terminology used troughout the rest of this paper. 
The notion of equiintegrability is commonly used for functions in $L^1$, as are the terms 
"vanishing" and "tight" (cf.~\cite{Li84a}), but we find it convenient to extend them to $W^{k,p}$ in a canonical way. The reader should be warned that the precise definitions in the literature might differ in the framework of unbounded domains. In particular, "equiintegrable" is sometimes used in a sense equivalent to what we term "does not concentrate".
\begin{defn}
\label{def:concentrate}
Let $\Omega\subset \RR^N$ be a domain (possibly unbounded) and 
let $V$ be a finite dimensional euclidean vector space with norm $\abs{\cdot}$.
Furthermore, let $u_n$ be a sequence in
$W^{k,p}(\Omega;V)$, where $k\in \NN_0$ and $p\in [1,\infty)$, and let $\alpha$ be a multiindex
with length $\abs{\alpha}\leq k$. 
We say that \emph{$(u_n)$ does not concentrate 
in $W^{k,p}(\Omega;V)$} if
\begin{align}\label{defconcomega}
	\sup_{E\subset \Omega,~ \abs{E}\leq \delta}~
	\int_E \abs{D^\alpha u_n}^p\,dx
	\underset{\delta\searrow 0}{\To} 0~~&\text{uniformly in $n\in \NN$, for every $\abs{\alpha}\leq k$}.
\intertext{%
We say that \emph{$(u_n)$ is tight in $W^{k,p}(\Omega;V)$} if
}
  \label{deftight}
	\int_{\Omega\setminus B_R(0)} \abs{D^\alpha u_n}^p\,dx
	\underset{R\to \infty}{\To} 0~~&\text{uniformly in $n\in \NN$, for every $\abs{\alpha}\leq k$}.
\intertext{%
If both \eqref{defconcomega} and \eqref{deftight} are satisfied we say that \emph{$u_n$ is equiintegrable in \emph{$W^{k,p}(\Omega;V)$}}. To describe a possible lack of tightness
in greater detail, we employ following two terms:
We say that \emph{$(u_n)$ does not spread out in $W^{k,p}(\Omega;V)$} if
}
  \label{defspread}
	\int_\Omega \min\left\{\delta,\abs{D^\alpha u_n}^p\right\}\, dx
	\underset{\delta\searrow 0}{\To} 0~~&\text{uniformly in $n\in \NN$, for every $\abs{\alpha}\leq k$}.
\intertext{%
Finally, we say that \emph{$(u_n)$ is vanishing in $W^{k,p}(\Omega;V)$} if
}
  \label{defvanish}
	\sup_{y\in \RR^N}\int_{B_1(y)\cap \Omega} \abs{D^\alpha u_n}^p\, dx
	\underset{n\to\infty}{\To} 0~~&\text{for every $\abs{\alpha}\leq k$}.
\end{align}
\end{defn}
\begin{rem}\label{rem:spread} ~ 
\begin{myenum}
\item If $q<p$, any sequence which is bounded in $W^{k,p}(\Omega;V)$ 
does not concentrate in $W^{k,q}(\Omega;V)$, as a
consequence of Hölder's inequality.
\item Observe that \eqref{deftight} implies \eqref{defspread}, i.e.,
every tight sequence does not spread out.
A typical example for a sequence "purely" spreading out in $L^p(\RR^N)$ is 
$v_n:=n^{-N/p}\chi_{E_n}$ with any sequence of measurable sets $E_n$ satisfying $\abs{E_n}=C n^N$ for a constant $C>0$ ($E_n:=B_n(0)$, e.g.). Here, $\chi_E$ denotes the indicator function of the set $E$ given in the index. Carefully note that
if a bounded sequence in $L^p(\RR^N)$ does not spread out and is vanishing at the same time, this still does not imply strong convergence to zero, unless additional assumptions are made (cf.~Lemma~\ref{lem:cospread3}).
For instance, the sequence $\chi_{E_n}$ of indicator functions of
\begin{align*}
	E_n:=\left\{(x_1,\ldots,x_N)\, \left| \,\abs{x_1}\leq n^{N-1}~\text{and}~\abs{x_j}\leq 
	n^{-1}~\text{for}~j=2,\ldots,N \right.\right\}\subset \RR^N
\end{align*}
provides a counterexample. 
\item If $q\leq p$, any sequence $(u_n)$ in $W^{k,q}(\Omega;V)$ which does not spread out in $W^{k,q}$ 
also does not spread out in $W^{k,p}$. In particular, note that \eqref{defspread} makes sense
even if some or all of the members $u_n$ do not have finite norm in $W^{k,p}$.
\end{myenum}
\end{rem}

\subsection{Auxiliary results}
%
%
We first record an interesting relation between spreading and vanishing.
\begin{lem}\label{lem:cospread3}
Let $\Omega\subset \RR^N$ be an arbitrary domain
and let $p\in [1,\infty)$.
If $(u_n)$ is a bounded sequence in $W_0^{1,p}(\Omega;\RR^M)$ which 
does not spread out in $L^{p}(\Omega;\RR^M)$ and which is vanishing in $L^{p}(\Omega;\RR^M)$,
then $u_n\to 0$ strongly in $L^p(\Omega;\RR^M)$.
\end{lem}
Postponing the proof for a moment, we now state three results forming the basis for a decomposition lemma 
associated to the truncations introduced below, which essentially comes down to verifying \eqref{phi3} in each case. Actually, they are decomposition lemmas for sequences in $L^p$. For instance,
in Lemma~\ref{lem:Lpequiintloc} the corresponding decomposition is
$v_{k(n)}=\eta_n\circ v_{k(n)}+(\identity_\RR-\eta_n)\circ v_{k(n)}$.
Lemma~\ref{lem:Lpequiintloc} is a close relative of Chacon's biting lemma (for the latter see \cite{Pe97B}, e.g.).
In its essence, it is well known; in particular, it is implicitly proved in \cite{FoMuePe98a} in case of a bounded domain, with an argument involving Young measures. 
By contrast, the proof given below is elementary. 
%
\begin{lem} \label{lem:Lpequiintloc}
Let $\Omega\subset \RR^N$ be an arbitrary domain (possibly unbounded) and $p\in [1,\infty)$.
Then every bounded sequence $(v_n)\subset L^p(\Omega)$ has a subsequence $(v_{k(n)})$ such that
the sequence $(\eta_n\circ v_{k(n)})$ does not concentrate in $L^p(\Omega)$, i.e.,
\begin{align}\label{lemLpeqil0}
	\sup_{E\subset \Omega,~\abs{E}\leq \delta}~
	\int_E \abs{\eta_n[v_{k(n)}(x)]}^p\,dx 
	\underset{\delta\to 0}{\To} 0,~\text{uniformly in $n\in \NN$}.
\end{align}
Here, for every $\lam>0$,
\begin{equation}\label{defetalam}
\begin{aligned}
	\eta_\lam: \RR\to \RR,~~\eta_\lam(t):=\left\{ \begin{array}{rl}
		\lam\quad &\text{if}~t\in (\lam,\infty),\\
		t\quad &\text{if}~t\in [-\lam,\lam],\\
		-\lam\quad &\text{if}~t\in (-\infty,-\lam).
	\end{array}\right.
\end{aligned}
\end{equation}
\end{lem}
\begin{lem} \label{lem:Lpequiintoutside}
Let $\Omega\subset \RR^N$ be an arbitrary domain and $p\in [1,\infty)$.
Then every bounded sequence $(v_n)\subset L^p(\Omega)$ has a subsequence $(v_{k(n)})$ such that
$(\chi_n\cdot v_{k(n)})$ is tight in $L^p(\Omega)$, i.e.,
\begin{align}\label{lemLpeqio0}
	\int_{\Omega\setminus B_{R}(0)} \abs{\chi_n(x)v_{k(n)}(x)}^p\,dx 
	\underset{R\to\infty}{\To} 0,~\text{uniformly in $n\in\NN$}.
\end{align}
Here, 
\begin{equation}\label{defchin}
\begin{aligned}
	\chi_n: \Omega\to \RR,~\chi_n(x):=\chi_{B_n(0)}:=\left\{ \begin{array}{rl}
		1\quad &\text{if $x\in B_n(0)$},\\
		0\quad &\text{if $x\notin B_n(0)$}.
	\end{array}\right.
\end{aligned}
\end{equation}
\end{lem}
\begin{lem} \label{lem:Lpequispread}
Let $\Omega\subset \RR^N$ be an unbounded domain and $p\in [1,\infty)$.
Then every bounded sequence $(v_n)\subset L^p(\Omega)$ has a subsequence $(v_{k(n)})$ such that
\begin{align}\label{lemLpeqisp0}
	\sup_{E\subset\Omega,~\abs{E}\leq n} \int_{E} \min\{\delta,|v_{k(n)}(x)|^p\}\,dx 
	\underset{\delta\to 0}{\To} 0,~\text{uniformly in $n\in\NN$}.
\end{align}
In particular, the sequence $w_n:=\chi_{E_n}\cdot v_{k(n)}$ does not spread out in $L^p(\Omega)$
if $E_n$ is an arbitrary sequence of measurable sets satisfying $\abs{E_n}\leq n$,
where $\chi_{E_n}$ denotes the indicator function of the set $E_n$.
\end{lem}
\subsubsection*{Proofs of Lemma~\ref{lem:cospread3}--Lemma~\ref{lem:Lpequispread}}

For the proof of Lemma~\ref{lem:cospread3}, we need 
Poincaré's inequality in the following form:
\begin{lem}\label{lem:poincare}
Let $D\subset \RR^N$ be 
a smooth 
bounded domain, 
$B_1:=B_1(0)$ the unit ball in $\RR^N$ and $p\in [1,\infty)$. Then for every $v\in W_0^{1,p}(D)$,
we have that
\begin{align}\label{poincare}
	\norm{v}_{L^p(D)}\leq S^{-1} \abs{B_1}^{-\frac{1}{N}}\abs{\{v\neq 0\}}^{\frac{1}{N}}
	\norm{\nabla v}_{L^p(D)},
\end{align}
where $S=S(N,p)>0$ is the optimal Poincaré constant on the unit ball, i.e.,
\begin{align*}
	S:=\inf\left\{ 
	\norm{\nabla w}_{L^p(B_1)}~\left|~w\in W_0^{1,p}(B_1)
	~\text{and}~\norm{w}_{L^p(B_1)}=1\right.\right\}
\end{align*}
\end{lem}
\begin{proof}
By rescaling the variable and Schwarz rearrangement (or spherical symmetric rearrangement, as it is called in \cite{BroZie88a}), the assertion can be obtained as a consequence of Poincaré's inequality on $B_1$. 
We omit the details.
\end{proof}
\begin{proof}[Proof of Lemma~\ref{lem:cospread3}]
We only consider the scalar case $M=1$, the general case can be obtained by arguing 
component-wise. Define a sequence
\begin{align*}
	\delta_n:=\left(
	\sup_{y\in \RR^N}\int_{B_1(y)} \abs{u_n(x)}^p \,dx
	\right)^{\frac{1}{p+1}},
\end{align*}
where $u_n$ is extended with zero outside of $\Omega$.
Moreover, for $x\in \RR^N$ let
\begin{align*}
	v_n(x):=\left\{\begin{array}{rll}
	u_n(x)-\delta_n &~& \text{if}~u_n(x)>\delta_n,\\
	u_n(x)+\delta_n &~& \text{if}~u_n(x)<-\delta_n,\\
	0 && \text{elsewhere},
	\end{array}\right.
\end{align*}
which defines a function in $W_0^{1,p}(\Omega)\subset W^{1,p}(\RR^N)$.
Note that for every measurable set $E\subset\Omega$,
\begin{align*}
	\norm{v_n}_{L^p(E)}\leq \norm{u_n}_{L^p(E)}
	~~\text{and}~~\norm{\nabla v_n}_{L^p(E)}\leq \norm{\nabla u_n}_{L^p(E)}.
\end{align*}
Since $u_n$ does not spread out in $L^p(\Omega)$, $u_n-v_n\to 0$ strongly in $L^p$. Hence it is enough to show that $v_n\to 0$ in $L^p$.
Our choice of $\delta_n$ entails that
\begin{align*}
	(\delta_n)^p\abs{\{x\in B_1(y)\,:\, \abs{u_n(x)}>\delta_n\}}
	\leq \int_{B_1(y)} \abs{u_n(x)}^p\,dx\leq (\delta_n)^{p+1},
\end{align*}
for every $y\in \RR^N$. Since $u_n$ is vanishing in $L^p(\Omega)$, we infer that
\begin{align}\label{pcos31}
  \abs{\{x\in B_1(y)\mid v_n(x)\neq 0\}}\leq \delta_n\underset{n\to\infty}{\To}0.
\end{align}
Now choose a covering of $\RR^N$ by a countable family of unit balls 
$B^{(i)}=B_1(y_i)$ ($i\in \NN$), 
locally finite in the sense that
\begin{align}\label{pcos32a}
	\text{each $x\in \RR^N$ is contained in at most $J$ different balls $B^{(i)}$},
\end{align}
where $J=J(N)\in\NN$ is a constant.
By a suitable corresponding smooth partition of unity, 
$v_n$ can be decomposed into a locally finite sum
\begin{align*}
	v_n=\sum_{i\in \NN} w_n^{(i)},~~
	\text{where $w_n^{(i)}\in W_0^{1,p}(B^{(i)})$.} 
\end{align*}
Moreover, 
\begin{align}\label{pcos32}
	\{w_n^{(i)}\neq 0\}\subset B^{(i)}\cap\{v_n\neq 0\} ~~\text{and}~~
	\|\nabla w_n^{(i)}\|_{L^p(B^{(i)})}\leq C_1 \norm{v_n}_{W^{1,p}(B^{(i)})},
\end{align}
where $C_1=C_1(N)$ is a constant. By Poincaré's
inequality in the form \eqref{poincare}, 
\begin{align*}
	\|w_n^{(i)}\|_{L^{p}(B^{(i)})}^p\leq S^{-p} \abs{B_1(0)}^{-\frac{p}{N}}
	|\{w_n^{(i)}\neq 0\}|^{\frac{p}{N}}
	\|\nabla w_n^{(i)}\|_{L^{p}(B^{(i)})}^p.
\end{align*}
Due to \eqref{pcos32a}, \eqref{pcos32} and \eqref{pcos31}, adding up yields
\begin{align*}
	\norm{v_n}_{L^{p}(\Omega)}^p
	&\leq \int_{\Omega} J^p \big(\sum_{i\in\NN} \big|w_n^{(i)}\big|^p\Big)\,dx\\
	&\leq J^p S^{-p} \abs{B_1(0)}^{-\frac{p}{N}}
	\sum_{i\in\NN} \abs{B^{(i)}\cap \{v_n\neq 0\}}^{\frac{p}{N}} 
	\|\nabla w_n^{(i)}\|_{L^{p}(B^{(i)})}^p\\
	&\leq J^p S^{-p} \abs{B_1(0)}^{-\frac{p}{N}} (C_1)^p 
	(\delta_n)^{\frac{p}{N}} J \norm{v_n}_{W^{1,p}(\Omega)}^p
	\underset{n\to\infty}{\To}0,
\end{align*}
which entails the assertion.
\end{proof}
\begin{proof}[Proof of Lemma~\ref{lem:Lpequiintloc}]
We inductively define a chain $(k_{1,n}(j))_j$ of subsequences of $j$:
Let $k_{1,0}(j):=j$ for $j\in \NN$. For fixed $n\in \NN$, choose $k_{1,n}(j)$ 
as a subsequence of $k_{1,n-1}(j)$ in such a way that
$\int_{\Omega}\abs{\eta_n[v_{k_{1,n}(j)}]}^p\,dx$
converges in $\RR$ as $j\to\infty$.
The diagonal subsequence inherits this property, i.e., for every fixed $n$,
\begin{align*}
	S_n^j:=\int_{\Omega}\abs{\eta_n[v_{k_{1}(j)}]}^p\,dx~~\text{converges as $j\to\infty$,
	where $k_1(j):=k_{1,j}(j)$.}
\end{align*}
Let
\begin{align*}
	S_n^\infty:=\lim_{j\to \infty} S_n^j\quad\text{and}\quad S_\infty:=\lim_{j\to \infty} S_n^\infty.
\end{align*}
Here, the limit $S_\infty$ exists since $S_n^\infty$ is increasing, and $S_\infty\leq\sup_{m\in \NN}\norm{v_m}_{L^p(\Omega)}<\infty$.
Furthermore, there exists a subsequence $k_2(n)$ of $n$ in such a way that 
\begin{align*}
 	\abs{S_n^j-S_n^\infty}\leq \frac{1}{n}\quad\text{whenever}~j\geq k_2(n).
\end{align*}
We claim that~\eqref{lemLpeqil0} holds with $k(n):=k_1(k_2(n))$. For the proof, let $\eps>0$.
First choose a number $n_0=n_0(\eps)\in \NN$ such that
\begin{align}\label{lemLpeqil2}
	\abs{S_n^{k_2(n)}-S_{n_0}^{k_2(n)}}<\frac{\eps}{2}\quad\text{for every $n> n_0$}, 
\end{align}
which is possible since
\begin{align*}
	& \abs{S_n^{k_2(n)}-S_{n_0}^{k_2(n)}}\\
	& \leq  \abs{S_n^{k_2(n)}-S_n^\infty}
	+\abs{S_n^\infty-S_\infty}
	+\abs{S_\infty-S_{n_0}^\infty}
	+\abs{S_{n_0}^\infty-S_{n_0}^{k_2(n)}} \\
	& \leq \frac{1}{n}+\abs{S_n^\infty-S_\infty}
	+\abs{S_\infty-S_{n_0}^\infty}+\frac{1}{n_0},
\end{align*}
and the last line becomes small if $n_0$ is large enough and $n>n_0$.
It is enough to show that
\begin{align}\label{lemLpeqil3}
	&\int_E \abs{\eta_n[v_{k(n)}(x)]}^p\,dx<\eps~\text{whenever}~\abs{E}<\delta_0:=\frac{\eps}{2n_0^p},
\end{align}
for every $n\in \NN$ and $E\subset \Omega$ (measurable). 
We distinguish the two cases $n\leq n_0$ and $n>n_0$: In the former case, 
\eqref{lemLpeqil3} holds since
\begin{align*}
	&\int_E \abs{\eta_n[v_{k(n)}(x)]}^p\,dx\leq \abs{E} n^p
		<\frac{\eps}{2}<\eps\quad \text{if $n\leq n_0$ and $\abs{E}<\delta_0$},
\end{align*}
whereas in the latter case, we have that
\begin{align*}
	&\int_E \abs{\eta_n[v_{k(n)}(x)]}^p\,dx \\
	&= \int_{E} \left(\abs{\eta_n[v_{k(n)}(x)]}^p
	-\abs{\eta_{n_0}[v_{k(n)}(x)]}^p\right)\,dx
	+\int_{E} \abs{\eta_{n_0}[v_{k(n)}(x)]}^p\,dx\\
	&\leq \int_{\Omega} \left(\abs{\eta_n[v_{k(n)}(x)]}^p
	-\abs{\eta_{n_0}[v_{k(n)}(x)]}^p\right)\,dx+\abs{E} n_0^p	\\
	&=S_n^{k_2(n)}-S_{n_0}^{k_2(n)}+\abs{E} n_0^p\\
	&<\frac{\eps}{2}+\frac{\eps}{2}=\eps \quad\text{if $n>n_0$ and $\abs{E}<\delta_0$},
\end{align*}
due to~\eqref{lemLpeqil2} and the definition of $\delta_0$.
\end{proof}
\begin{proof}[Proof of Lemma~\ref{lem:Lpequiintoutside}]
We proceed analogously to the proof of Lemma~\ref{lem:Lpequiintloc}. 
Just define $E=E(\delta):=\Omega\setminus B_{\frac{1}{\delta}}(0)$, use
$\chi_n\cdot v_m$ instead of $\eta_n\circ v_m$ ($n,m\in \NN$) and replace
the estimate 
\begin{align*}
	\int_E \abs{\eta_{m}[v_{k(n)}(x)]}^p\,dx\leq \abs{E} m^p
\end{align*}
employed twice in the proof of \eqref{lemLpeqil3} (where $m=n$ or $m=n_0$, respectively) with	
\begin{align} \label{lemLpeqio1}
	\int_E \abs{\chi_{m}(x)v_{k(n)}(x)}^p\,dx\leq
	\int_{E\cap B_{m}(0)} \abs{v_{k(n)}(x)}^p\,dx.
\end{align}
Here, note that the right hand side of~\eqref{lemLpeqio1}
is zero (and thus smaller than $\eps/2$) if 
$m\leq n_0$ and $\delta<\delta_0:=\frac{1}{n_0}$.
\end{proof}
\begin{proof}[Proof of Lemma~\ref{lem:Lpequispread}]
For $m\in\NN$ define
\begin{align*}
	g_m:(0,\infty)\to [0,\infty),~~g_m(t):=\abs{\left\{x\in \Omega\,:\,|v_{m}(x)|^p\geq t \right\}},
\end{align*}
which is a decreasing, upper semicontinuous function. Moreover,
$(g_m)_{m\in\NN}$ is a bounded sequence in $L^1((0,\infty))$ since
\begin{align*}
	\int_0^\infty g_m(t)\,dt=\int_{\Omega} |v_{m}(x)|^p\,dx,
\end{align*}
due to Cavalieri's principle. The latter also entails that
\begin{align*}
	\sup_{\abs{E}\leq n}\int_E \min\{\delta, |v_{m}(x)|^p\}\,dx
	\leq \int_0^\delta \min\{n,g_m(t)\}\,dt
	= \int_0^\delta \abs{\eta_n(g_m(t))}\,dt.
\end{align*}
where $\eta_n$ is defined in \eqref{defetalam}. 
Lemma~\ref{lem:Lpequiintloc} now provides a suitable choice of $k(n)$.
\end{proof}

\subsection{Cutoff of outer regions in unbounded domains}\label{ssec:coouterregions}
Let $\Omega\subset \RR^N$ be an arbitrary domain, $M\in \NN$ and $p\in [1,\infty)$.
\begin{defn}\label{def:trunc1}
For $n\in \NN$ and every function $u:\Omega\to \RR^M$ define
\begin{align} \label{defphin1}
	\phi^{(1)}_n(u)(x):=\nu(\abs{x}-n)u(x),
\end{align}
where $\nu:\RR\to \RR$ is a fixed function of class $C^\infty$ which is decreasing on $\RR$ and satisfies
\begin{align*}
	& \nu(r)=1~\text{if}~r\leq 0,~\nu(r)=0~\text{if}~r\geq 1,\quad\text{and}\quad
	\abs{\frac{d^j}{dr^j}\nu}\leq C_j~\text{on}~(0,1)
\end{align*}
for suitable constants $C_j>0$ (for example, $C_1=2$ and $C_2=16$ can be achieved).
\end{defn}
The maps $\phi^{(1)}_n$ are a family of truncation operators on $W^{k,p}(\Omega;\RR^M)$:
\begin{prop}\label{prop:codomain}
Let $\Omega\subset \RR^N$ be a domain, $p\in [1,\infty)$ and $k\in \NN\cup \{0\}$.
Then the family of linear operators $\phi^{(1)}_n$ defined above satisfies \eqref{phi1} and \eqref{phi3},
for $D=X=W^{k,p}(\Omega;\RR^M)$.
Moreover, \eqref{uknequiint} implies that
$\phi^{(1)}_n(u_{k(n)})$ is tight in $W^{m,r}(\Omega;\RR^M)$ for every pair $m\in \{0,\ldots,k\}$, $r\in [p,\infty)$ such that $X$ is continuously (but not necessarily compactly) embedded in $W^{m,r}(\Omega;\RR^M)$. 
Conversely, if $\phi^{(1)}_n(u_{k(n)})$ is tight in $W^{k,p}(\Omega;\RR^M)$ then
\eqref{uknequiint} is satisfied.
\end{prop}%
\begin{rem}
In particular, the family $\phi^{(1)}_n$ also is a family of truncation operators
on $X=W^{2,p}(\Omega;\RR^M)\cap W_0^{1,p}(\Omega;\RR^M)$ and on
$X=W_0^{1,p}(\Omega;\RR^M)\cap W_0^{1,q}(\Omega;\RR^M)$, respectively, cf.~Remark~\ref{remcutoffdef} (iv).
Here, $p,q\in [1,\infty)$ are arbitrary.
\end{rem}
\begin{proof}[Proof of Proposition~\ref{prop:codomain}]
For brevity, we write $\phi_n$ instead of $\phi_n^{(1)}$ below.
By definition, it is clear that the $\phi_n$ are an equibounded family of linear operators on $X$. The limiting properties
are also satisfied. 
Next, we prove~\eqref{phi3}.
Since the derivatives of $\nu$ are bounded, we have that
\begin{align*}
	\norm{(\phi_{n}-\phi_{j}) u_{k(n)}}_{W^{k,p}(\Omega;\RR^M)}~\leq~ C \norm{u_{k(n)}}_{W^{k,p}\left(\Omega\cap B_{n+1}(0)
	\setminus B_{j}(0);\RR^M\right)},
\end{align*}
whenever $j<n$, where $C>0$ is a constant independent of $j$ and $n$. Thus it is enough to show that 
\begin{align}
	\norm{u_{k(n)}}_{W^{k,p}\left(\Omega\cap B_{n+1}(0)\setminus B_{j}(0);\RR^M\right)}
	~\underset{j\to \infty}{\To}~ 0~~\text{uniformly in $n\in \NN$,} \label{pubdei1}
\end{align}
for a suitable subsequence $u_{k(n)}$ of $u_n$.
The subsequence is chosen with the help of Lemma~\ref{lem:Lpequiintoutside}, subsequently applied to $v^{(\alpha)}_n:=\abs{D^\alpha u_n}$ for every multiindex of length $\abs{\alpha}\in \{0,\ldots, k\}$.
Then~\eqref{pubdei1} is an immediate consequence (use $\delta:=1/j$),
and the last sentence of the assertion follows by reasoning as above. 
It remains to show that \eqref{uknequiint} implies tightness of $\phi^{(1)}_n(u_{k(n)})$ in $W^{m,r}(\Omega;\RR^M)$. For the proof, observe that if $X$ is continuously embedded in $W^{m,r}$,
we may estimate
\begin{equation*}
\begin{aligned}
  &\norm{\phi_{n}(u_{k(n)})}_{W^{m,r}(\Omega
	\setminus B_{j+1}(0);\RR^M)}\\
	&\qquad \leq \norm{(\phi_{n}-\phi_{j}) (u_{k(n)})}_{W^{m,r}(\Omega;\RR^M)}\leq
	C\norm{(\phi_{n}-\phi_{j}) (u_{k(n)})}_{W^{k,p}(\Omega;\RR^M)}
\end{aligned}
\end{equation*}
for every $j<n$, where $C>0$ is the embedding constant. The case $j\geq n$ can be ignored,
since then $\norm{\phi_{n}(u_{k(n)})}_{W^{m,r}(\Omega\setminus B_{j+1}(0);\RR^M)}$ is zero anyway.
\end{proof}

\subsection{Truncation of gradients above large levels}%
Let $\Omega\subset \RR^N$ be a Lipschitz domain in the sense of Definition~\ref{def:Lipschitzdom}
(possibly unbounded) and $X:=W_0^{1,p}(\Omega;\RR^M)\cap W_0^{1,q}(\Omega;\RR^M)$. We now want to truncate a function $u\in X$ by replacing it with a Lipschitz function $\phi^{(2)}_n(u)$ with a global Lipschitz constant bounded by a constant multiple of $n$. Ultimately, we thus cut off concentration of a sequence at sets of measure zero. 
The truncation of vector--valued functions rests on its scalar equivalent, Theorem~\ref{thm:hcutoffOmega} in the appendix. It is  defined component-wise:
\begin{defn}[Truncation of gradients of vector functions above large levels]\label{def:truncgradmax} 
Let $\Omega\subset \RR^N$ be a Lipschitz domain and $1\leq q\leq p<\infty$. For $n\in \NN$ and for $u=(u_1,\ldots,u_M)\in W_0^{1,p}(\Omega;\RR^M)\cap W_0^{1,q}(\Omega;\RR^M)$, the truncation $\phi^{(2)}_n(u)$ of $u$ at level $n$, 
$\phi^{(2)}_n(u)=\big([\phi^{(2)}_n(u)]_1,\ldots,[\phi^{(2)}_n(u)]_M\big)$,
is defined by
\begin{align*}
	[\phi^{(2)}_n(u)]_i:=\phi_n(u_i)\quad\text{for $i=1,\ldots,M$},
\end{align*}
where $\phi_n$ on the right hand side is the truncation of scalar functions obtained in Theorem~\ref{thm:hcutoffOmega} (with $\lam:=n$). Accordingly, we define	
\begin{align*}
	 \hat{R}^n=\hat{R}^n(u):=\bigcap_{i=1}^{M} \hat{R}^n(u_i).
\end{align*}
\end{defn}
\begin{rem}\label{rem:truncgradvec}
The truncation of a scalar function obtained in Theorem~\ref{thm:hcutoffOmega} is not uniquely determined. We choose one arbitrarily for every function $u$ and every value of the truncation parameter $\lam$. For our purposes, this ambiguity does not matter since \eqref{phcutoffomegaubound}--\eqref{hcoOlim}, the only properties we exploit, are satisfied irrespective of the choice.
\end{rem}
All properties of the scalar truncation are inherited:
\begin{prop}\label{prop:coheight0}
Let $\Omega\subset \RR^N$ be a Lipschitz domain, $1\leq q\leq p<\infty$
and $u\in W_0^{1,p}(\Omega;\RR^M)\cap W_0^{1,q}(\Omega;\RR^M)$. Then we have that
\begin{align}
	{}&
	\abs{\nabla \phi^{(2)}_n(u)(x)}+\abs{\phi^{(2)}_n(u)(x)}~\leq~ C_0 n\quad\text{for a.e.}~x\in\Omega,
	\label{pcoh01} 
	\\{}&
	u(x)=\phi^{(2)}_n(u)(x)\quad \text{for a.e.}~x\in \hat{R}^n(u),
	\label{pcoh02} 
	\\{}&	
	\abs{\Omega \setminus \hat{R}^n(u)}
	~\leq~	(C_3)^r\frac{1}{n^r} \int_{\left\{\abs{u}+\abs{\nabla u}>\frac{n}{2}\right\}}
	\abs{u}^r+\abs{\nabla u}^r\,dx~~\text{for $r\in [1,\infty)$}, 
  \label{pcoh04} 
\end{align}
and 
\begin{equation}
	{}\begin{aligned}
	&\abs{\left\{x\in \Omega\mid\abs{\phi^{(2)}_n(u)(x)}+\abs{\nabla \phi^{(2)}_n(u)(x)}>\delta\right\}}\\
	&~\leq C_1 \abs{\left\{x\in \RR^N\mid C_0
	\min\{n,\cM(\abs{u}+\abs{\nabla u})(x)\}>\delta\right\}}
	~ \text{for every $\delta\geq 0$},
	\end{aligned} \label{pcoh06} 
\end{equation}
where $\cM$ is the maximal operator of Definiton~\ref{def:maxop}. 
Here, $C_0, C_1, C_3\geq 1$ are constants which only depend on $N$, $M$
and $\Omega$.
\end{prop}
\begin{proof}
These properties are immediate consequences of the corresponding ones listed in
Theorem~\ref{thm:hcutoffOmega} and Corollary~\ref{cor:hcoOmega}.
Corresponding constants coincide up to a constant factor which only depends on
$M$.
\end{proof}
For $p>1$, the maps $\phi^{(2)}_n$ form a family of truncation operators on 
$W_0^{1,p}\cap W_0^{1,p}$ in the sense of Definition~\ref{def:cutoff}:
\begin{prop}\label{prop:coheight}
Let $\Omega\subset \RR^N$ be a Lipschitz domain, let $1\leq q\leq p<\infty$ and assume that $p>1$. 
Then the family of (nonlinear) maps $\phi_n^{(2)}$ on $X:=W_0^{1,p}(\Omega;\RR^M)\cap W_0^{1,q}(\Omega;\RR^M)$
introduced in Definition~\ref{def:truncgradmax} above satisfies~\eqref{phi1} and \eqref{phi3}.
Moreover,
\begin{align}
	\phi^{(2)}_n(u)\underset{n\to\infty}{\To} u~~\text{in $W_0^{1,p}\cap W_0^{1,q}$, for every fixed $u$.}
	\label{pcohlim}
\end{align}
\end{prop} 
The concrete meaning of \eqref{uknequiint} (or \eqref{coequiintB}) in the present setting is as follows:
\begin{prop}\label{prop:coheight2}
Let $1\leq q \leq p<\infty$ and let $\Omega\subset \RR^N$ be a Lipschitz domain.
If $(u_n)\subset X:=W_0^{1,p}(\Omega;\RR^M)\cap W_0^{1,q}(\Omega;\RR^M)$ is a bounded sequence which satisfies \eqref{coequiintB} 
then $\phi^{(2)}_n(u_n)$ does not concentrate in $W^{1,p}(\Omega;\RR^M)$, in the 
sense of Definition~\ref{def:concentrate}. If $p<N$ then $\phi^{(2)}_n(u_n)$ also does not concentrate in  $L^{p^*}(\Omega;\RR^M)$, where $p^*:=\frac{pN}{N-p}$.
\end{prop}
\begin{rem}\label{rem:decompW11}
The restriction $p>1$ in Proposition~\ref{prop:coheight} is not just technical. 
More precisely, neither \eqref{phi3} (for $\phi^{(2)}_n$) nor 
the decomposition lemma (Lemma~\ref{lem:dec} below) hold if $q=p=1$.
For instance, consider 
\begin{align*}
	u_n:(-1,1)\to \RR,~~ u_n:=\left\{\begin{array}{rll}
	-t-1  & \quad  & \text{if $t\in (-1,-\frac{1}{n})$,}\\
	(n-1)t && \text{if $t\in [-\frac{1}{n},\frac{1}{n}]$,}\\
	-t+1 && \text{if $t\in (\frac{1}{n},1)$,}
	\end{array}
	\right.
\end{align*}
which is a bounded sequence in $X=W_0^{1,1}((-1,1))$.
Now assume for a moment that \eqref{phi3} is valid for $u_n$, i.e., 
a subsequence $u_{k(n)}$ satisfies \eqref{uknequiint}. 
Since the measure of $\{\phi^{(2)}_n(u_{k(n)})\neq u_{k(n)}\}$ converges to zero
due to \eqref{pcoh02} and \eqref{pcoh04}, 
there are two sequences $t_n^-$,
$t_n^+$ in $(-1,1)$ such that for every $n\in \NN$,
\begin{align*}
	 \text{$t_n^-<-\frac{1}{k(n)}$, $t_n^+>\frac{1}{k(n)}$ and  
	 $\phi^{(2)}_n(u_{k(n)})(t_n^{\pm})=u_{k(n)}(t_n^{\pm})$, 
	 and $t_n^{\pm}\underset{n\to\infty}{\To} 0$.}
\end{align*}
Hence 
\begin{align*}
 \int_{[t_n^-,t_n^+]} [\phi^{(2)}_n(u_{k(n)})]'(s)\,ds= 
 (-t_n^+ +1)-(-t_n^- -1)\to 2\neq 0,
\end{align*}
which means that the first derivative $[\phi^{(2)}_n(u_{k(n)})]'$ does concentrate
(i.e., \eqref{defconcomega} is violated), contradicting 
Proposition~\ref{prop:coheight2}.
\end{rem}
%
\begin{proof}[Proof of Proposition~\ref{prop:coheight}]
For brevity, we write $\phi_n$ instead of $\phi_n^{(2)}$ below.
It is suffices to discuss the case $p=q$: 
As to \eqref{phi3}, observe that
due to \eqref{pcoh04}, the measure of $\{\phi_m(u_n)\neq 0\}$ is bounded uniformly in $n,m\geq 1$
whence the norm in $W^{1,p}$ dominates the norm in $W^{1,q}$ for functions of this form in case $p>q$.
Concerning the other assertions, the general case can be recovered by 
repeating the corresponding argument below with $q$ replacing $p$.\\
Due to \eqref{pcoh01} and \eqref{pcoh02}, $\abs{\nabla \phi_n(u)}+\abs{\phi_n(u)}\leq C_0n$ a.e.~in $\Omega$ and $\phi_n(u)=u$ a.e.~on $\hat{R}^n(u)$.  Consequently,~\eqref{pcoh04} implies that
\begin{equation}\label{propcoh1}
\begin{aligned}
	&\norm{\phi_n(u)-u}_{W^{1,p}(\Omega;\RR^M)}\\
	&\leq \left(2C_0^pn^p\abs{\Omega\setminus \hat{R}^n(u)}\right)^{\frac{1}{p}}
	+ \left(\int_{\Omega\setminus \hat{R}^n(u)} \left(\abs{u}^p+\abs{\nabla u}^p\right)\,dx\right)^{\frac{1}{p}}\\
	&\leq 2(2 (C_0 C_3)^p+1)^{\frac{1}{p}}
	\Big(\int_{\left\{\abs{\nabla u}+\abs{u}>\frac{n}{2}\right\}\cup
	(\Omega\setminus \hat{R}^n(u))}\left(\abs{\nabla u}^p+\abs{u}^p\right)\,dx\Big)^{\frac{1}{p}}.
\end{aligned}
\end{equation}
In particular, the maps $\phi_n$ are equibounded in the sense of~\eqref{phi1} and \eqref{pcohlim} holds.
Last but not least, we verify \eqref{phi3}. We have to show that every bounded sequence $(u_n)\subset W_0^{1,p}(\Omega;\RR^M)$ 
has a subsequence $(u_{k(n)})$ such that
\begin{align} \label{propcohequiint0}
		{}&\int_\Omega \left(\abs{\nabla(\phi_{n}-\phi_{j(n)})(u_{k(n)})}^p 
		+\abs{(\phi_{n}-\phi_{j(n)})(u_{k(n)})}^p \right)\,dx
		\underset{n\to\infty}{\To} 0.
\end{align}
for every sequence $(j(n))\subset \NN$ with $j(n)\to \infty$ as $n\to \infty$ and 
$j(n)<n$ for every $n\in\NN$. 
As a matter of fact, \eqref{pcoh06} allows
us to estimate the left hand side of \eqref{propcohequiint0} in such a way that Lemma~\ref{lem:Lpequiintloc}
can be applied. To achieve this, we proceed as follows:
\begin{align*}
		{}&\frac{1}{2^p}\int_\Omega \left(\abs{\nabla(\phi_{n}-\phi_{j(n)})(u_{k(n)})}^p 
		+\abs{(\phi_{n}-\phi_{j(n)})(u_{k(n)})}^p \right)\,dx \\
		& \leq \frac{1}{2^p}\int_{\Omega\setminus (\hat{R}^{j(n)}\cap \hat{R}^n)} 
		\left(\abs{\nabla(\phi_{n}-\phi_{j(n)})(u_{k(n)})}^p 
		+\abs{(\phi_{n}-\phi_{j(n)})(u_{k(n)})}^p \right)\,dx \\
		& \qquad\qquad\qquad \begin{array}{l}
			\text{since $\phi_{n}(u_{k(n)})=u_{k(n)}=\phi_{j(n)}(u_{k(n)})$ 
			on $\hat{R}^{j(n)}\cap \hat{R}^{n}$ by~\eqref{pcoh02}}
		\end{array}
		\\
		& \leq \int_{\Omega\setminus (\hat{R}^{j(n)}\cap \hat{R}^n)} \left(U_{n,n}\right)^p\,dx 
		+\int_{\Omega\setminus (\hat{R}^{j(n)}\cap \hat{R}^n)} \left(U_{j(n),n}\right)^p\,dx\\
		& \qquad\qquad\qquad\begin{array}{l}
			\text{where}~
			U_{m,n}:=\abs{\nabla \phi_{m}(u_{k(n)})}+\abs{\phi_{m}(u_{k(n)})}\in L^p(\Omega)
		\end{array}
		\\
		& \leq \sup_{\abs{E}\leq \gamma_n}
		\int_{E} \left[C_0\min\left\{n,v_{k(n)}(x)\right\}\right]^p\,dx
		+\sup_{\abs{E}\leq \gamma_n}
		\int_{E} \left[C_0\min\left\{j(n),v_{k(n)}(x)\right\}\right]^p\,dx\\
		& \qquad\qquad\qquad
		\begin{array}{l}
			\text{due to~\eqref{pcoh06} and Cavalieri's priciple, where}\\ 
			\text{$E$ can be any measurable subset of $\RR^N$,}\\
			\gamma_n:=\big|\Omega \setminus (\hat{R}^{j(n)}\cap \hat{R}^n)\big|,~~\text{and}\\
			v_m(x):=\cM(\abs{\nabla u_{m}}+\abs{u_{m}})\big(C_1^{-\frac{1}{N}}x\big)~~
			\text{(cf.~Definition~\ref{def:maxop})}\\
		\end{array}
		\\
		& \leq 2 C_0^p
		\sup_{\abs{E}\leq \gamma_n} 
		\int_{E} \left[\min\left\{n,v_{k(n)}(x)\right\}\right]^p\,dx.
\end{align*}
Since by (\ref{pcoh04}),
\begin{align*}
	\gamma_n=|(\Omega \setminus \hat{R}^{j(n)})\cup(\Omega \setminus \hat{R}^n)|\leq 
	(C_3)^p\left(\frac{1}{j(n)^p}+\frac{1}{n^p}\right) \sup_{m\in \NN}\norm{u_m}_{W^{1,p}(\Omega;\RR^M)}^p 
	\underset{n\to\infty}{\To} 0,
\end{align*}
Lemma~\ref{lem:Lpequiintloc} applied to $v_n\in L^p(\RR^N)$
yields the desired subsequence $k(n)$. Here, $v_n$ is a bounded sequence in $L^p(\RR^N)$
for $p>1$ due to Lemma~\ref{lem:MaxOpcont}
\end{proof}
\begin{proof}[Proof of Proposition~\ref{prop:coheight2}]
For brevity, we write $\phi_n$ instead of $\phi_n^{(2)}$ below.
We first show that $\phi_n(u_n)$ does not concentrate in $W^{1,p}(\Omega;\RR^M)$.
Let $\eps>0$. Due to \eqref{coequiintB}, there exists $j=j(\eps)\in \NN$ such that
\begin{align*}
	\norm{(\phi_n-\phi_{j})(u_n)}_{W^{1,p}(\Omega;\RR^M)}^p<\frac{\eps}{2}{2^{-p}}\quad\text{for every $n>j$}.
\end{align*}  
We claim that as a consequence, 
\begin{align*}
	\norm{\phi_n(u_n)}_{W^{1,p}(E;\RR^M)}^p<\eps~\text{if $\abs{E}<\delta_0$, uniformly in $n\in \NN$,}
\end{align*}
where $\delta_0:=\frac{\eps}{2}(4C_0j)^{-p}$. Here, $C_0>0$ is the constant introduced in
\eqref{pcoh01}.
In case $n>j$ we have
\begin{align*}
	\norm{\phi_n(u_n)}_{W^{1,p}(E;\RR^M)}^p
	&	\leq 2^p\norm{(\phi_n-\phi_{j})(u_n)}_{W^{1,p}(E;\RR^M)}^p
	+2^p\norm{\phi_j(u_n)}_{W^{1,p}(E;\RR^M)}^p\\
	&	\leq 2^p\norm{(\phi_n-\phi_{j})(u_n)}_{W^{1,p}(\Omega;\RR^M)}^p
	+(4C_0)^p j^p \abs{E} \\
	&	< \frac{\eps}{2}+\frac{\eps}{2}=\eps
\end{align*}
provided that $\abs{E}<\delta_0$. On the other hand, if $n\leq j$,
\begin{align*}
	&\norm{\phi_n(u_n)}_{W^{1,p}(E)}^p
	\leq (2C_0n)^p \abs{E}\leq (2C_0j)^p \abs{E}< \frac{\eps}{2^{p+1}}<\eps
\end{align*}
whenever $\abs{E}<\delta_0$. Finally, in case $p<N$, first recall that 
$W^{1,p}(\Omega;\RR^M)$ is continuously embedded in 
$L^{p^*}(\Omega;\RR^M)$. Hence we also have that
\begin{align*}
	\norm{(\phi_n-\phi_{j})(u_n)}_{{L^{p^*}}(\Omega;\RR^M)}^{p^*}<\frac{\eps}{2}{2^{-p^*}}\quad\text{for every $n>j$},
\end{align*}  
for a suitable $j=j(\eps)$. Arguing analogously as above with $\delta_0:=\frac{\eps}{2}(2C_0j)^{-p^*}$, we infer that
$\phi_n(u_n)$ does not concentrate in  $L^{p^*}(\Omega;\RR^M)$.
\end{proof}
%
\subsection{Truncation of gradients below small levels}%
As before let $\Omega\subset \RR^N$ be an Lipschitz domain in the sense of 
Definition~\ref{def:Lipschitzdom}. Although domains with finite measure are allowed from a technical point of view, the way of truncating described below only has practical meaning if the measure of $\Omega$ is infinite.
We want to truncate a function $u\in W_0^{1,p}(\Omega;\RR^M)\cap W_0^{1,q}(\Omega;\RR^M)$ by replacing it with a function $\phi^{(3)}_n(u)$ such that $\phi^{(3)}_n(u)$ is zero outside a set of finite measure and such that $\|u-\phi^{(3)}_n(u)\|_{W^{1,\infty}}$ is bounded by a constant times $\frac{1}{n}$.
In the applications, we use this way of truncating as a method to identify and extract components of a sequence which are "purely spreading", i.e., which converge to zero in $W^{1,\infty}$ while staying away from zero in the norm of $W^{1,p}\cap W^{1,q}$. The definition of $\phi^{(3)}_n$ is based on Theorem~\ref{thm:hcutoffOmega}, which also holds for arbitrarily small truncation levels.
\begin{defn}[Truncation of gradients vector functions below small levels]\label{def:truncvanish}~\\
Let $\Omega\subset \RR^N$ be a Lipschitz domain, $1\leq q\leq p<\infty$, $n\in\NN$ and $u=(u_1,\ldots,u_M)\in W_0^{1,p}(\Omega;\RR^M)$. The truncation
$\phi^{(3)}_n(u)=([\phi^{(3)}_n(u)]_1,\ldots,[\phi_n(u)]_M)$ of $u$ is defined by
\begin{align*}
	[\phi^{(3)}_n(u)]_i:=u_i-\phi_{1/n}(u_i)\quad\text{for $i=1,\ldots,M$},
\end{align*}
where $\phi_{1/n}$ on the right hand side denotes the truncation of scalar functions introduced in Theorem~\ref{thm:hcutoffOmega} (with $\lam:=1/n$). Accordingly, we define
\begin{align*}
	 \hat{R}^{1/n}=\hat{R}^{1/n}(u):=\bigcap_{i=1}^{M} \hat{R}^{1/n}(u_i).
\end{align*}
\end{defn}
From Theorem~\ref{thm:hcutoffOmega} and Corollary~\ref{cor:hcoOmega} we inherit the following properties:
\begin{prop}\label{prop:cospread0}
Let $\Omega\subset \RR^N$ be a Lipschitz domain, $1\leq q\leq p<\infty$, $n\in\NN$ and 
$u\in W_0^{1,p}(\Omega;\RR^M)\cap W_0^{1,q}(\Omega;\RR^M)$. 
Then we have that
\begin{align}
	{}&\begin{aligned}
	\big|\nabla [u-\phi^{(3)}_n(u)](x)\big|+\big|[u-\phi^{(3)}_n(u)](x)\big|\leq C_0 \frac{1}{n}
	\quad\text{for a.e.}~x\in\Omega,
	\end{aligned}
	\label{pcospr01}\\
	{}&\begin{aligned}
	\phi^{(3)}_n(u)(x)=0\quad \text{for a.e.}~x\in \hat{R}^{1/n}(u),
	\end{aligned}
	\label{pcospr02}\\
	{}&\begin{aligned}
   \abs{\Omega \setminus \hat{R}^{1/n}(u)}
   ~\leq~	& (C_3)^r n^r 
   \int_{\left\{\frac{1}{2n^{\beta}}\geq \abs{u}+\abs{\nabla u}> \frac{1}{2n}\right\}}
   \abs{u}^r+\abs{\nabla u}^r\,dx\\
	 & + (C_3)^r n^{1+\beta(r-1)}
	 \int_{\left\{\abs{u}+\abs{\nabla u}>\frac{1}{2n^{\beta}}\right\}}
	 \abs{u}^r+\abs{\nabla u}^r\,dx
	\end{aligned} 
	\label{pcospr03}
\intertext{for $r=p,q$ and $\beta\in (0,1]$ arbitrary. Moreover,}
  {}&\begin{aligned}	
  \frac{1}{n^r}\big|\Omega \setminus \hat{R}^{1/n}(u)\big|\underset{n\to\infty}{\To}0~~\text{for 
  $r=p,q$ and every fixed $u$,}
  \end{aligned}
  \label{pcospr04}
\intertext{and}
	{}&\begin{aligned}
	&\abs{\left\{x\in \Omega \,\left|\, \big|[u-\phi^{(3)}_n(u)](x)\big|+
	\big|\nabla [u-\phi^{(3)}_n(u)](x)\big|>\delta\right.\right\}}\\
	&\qquad \leq C_1 \abs{\left\{x\in \RR^N\,\left|\, C_0
	\min\{n^{-1},\cM(\abs{u}+\abs{\nabla u})(x)\}>\delta\right.\right\}}
	\end{aligned}\label{pcospr05}
\end{align}
for every $\delta\geq 0$, where $\cM$ is the maximal operator of Definiton~\ref{def:maxop}. 
Here, $C_0, C_1, C_3 \geq 1$ are constants which only depend on $M$, $N$ and
$\Omega$.
\end{prop}
\begin{proof}
Taking into account the definition of $\phi^{(3)}_n$ and $\hat{R}^{1/n}(u)$, the assertions are due to
\eqref{phcutoffomegaubound}, \eqref{phcutoffomegausupp}, \eqref{hcoOsmalllam}, \eqref{hcoOlim} and \eqref{phcutoffomegauhatMuineq},
respectively. The constants in these statements coincide with corresponding ones in the assertion
up to a factor which only depends on $M$.
\end{proof}
If $q>1$, the maps $\phi^{(3)}_n$ form a family of truncation operators on 
$W_0^{1,p}\cap W_0^{1,p}$ in the sense of Definition~\ref{def:cutoff}:
\begin{prop}\label{prop:cospread}
Let $1<q\leq p<\infty$ and let $\Omega\subset \RR^N$ be a Lipschitz domain. Then the family $\phi_n^{(3)}$ ($n\in \NN$) of maps on $X:=W_0^{1,p}(\Omega;\RR^M)\cap W_0^{1,p}(\Omega;\RR^M)$ satisfies~\eqref{phi1} and \eqref{phi3}.
Moreover,
\begin{align}
	\phi^{(3)}_n(u)\underset{n\to\infty}{\To} u~~\text{in $W_0^{1,p}\cap W_0^{1,q}$, for every fixed $u\in X$.}
	\label{pcovlim}
\end{align}
\end{prop} 
In the present setting, the concrete meaning of the property \eqref{uknequiint} (or \eqref{coequiintB}) 
in \eqref{phi3} is as follows:
\begin{prop}\label{prop:cospread2}
Let $\Omega\subset \RR^N$ a Lipschitz domain and let $1\leq q\leq p<\infty$.
If $(u_n)\subset X:=W_0^{1,p}(\Omega;\RR^M)\cap W_0^{1,q}(\Omega;\RR^M)$ is a bounded sequence which satisfies \eqref{coequiintB} for $\phi_n=\phi^{(3)}_n$,
then $\phi^{(3)}_n(u_n)$ does not spread out in $W^{1,q}(\Omega;\RR^M)$, in the 
sense of Definition~\ref{def:concentrate}. 
\end{prop}
%
\begin{proof}[Proof of Proposition~\ref{prop:cospread}]
To simplify notation, we write $\phi_n=\phi^{(3)}_n$ below. 
It is enough to discuss the case $p=q$; the general case can be recovered
by repeating each argument with $q$ replacing $p$.
Due to \eqref{pcospr01} and \eqref{pcospr02},
we have that 
\begin{equation}
\begin{aligned}
	&\int_\Omega \left(\abs{\nabla[u-\phi_n(u)]}^p+\abs{u-\phi_n(u)}^p\right)\,dx\\
	&\qquad\leq \int_{\{\abs{\nabla u}+\abs{u}\leq C_0\frac{1}{n}\}} 
	\left(\abs{\nabla u}^p+\abs{u}^p\right)\,dx
	+ 2(C_0)^p \frac{1}{n^p} \abs{\Omega\setminus \hat{R}^{1/n}(u)}
\end{aligned}\label{pcov1}
\end{equation}
Using \eqref{pcospr04}, we obtain \eqref{pcovlim}, and \eqref{phi1} is a consequence of \eqref{pcospr03} for $\beta=1$.
Last but not least, we verify \eqref{phi3}. We have to show that every bounded sequence $(u_n)\subset W_0^{1,p}(\Omega;\RR^M)$ 
has a subsequence $(u_{h(n)})$ such that
\begin{align} \label{pcov6}
		{}&\int_\Omega \left(\abs{\nabla(\phi_{n}-\phi_{j(n)})(u_{h(n)})}^p 
		+\abs{(\phi_{n}-\phi_{j(n)})(u_{h(n)})}^p \right)\,dx
		\underset{n\to \infty}{\to} 0.
\end{align}
for every sequence $(j(n))\subset \NN$ with $j(n)\to \infty$ as $n\to \infty$ and 
$j(n)<n$ for every $n\in\NN$. 
First note that
\begin{align*}
  \gamma(n):=Mn^p,~~\text{with $M\in \NN$ fixed such that $M\geq 2(C_3)^p\sup_{m\in\NN}\norm{u_m}_{W^{1,p}}^p$}, 
\end{align*}
satisfies
\begin{align}\label{pcov7}
	\gamma(n)
	\geq \sup_{m\in\NN}\abs{\Omega \setminus \big(\hat{R}^{1/j(n)}(u_m)\cap \hat{R}^{1/n}(u_m)\big)},
\end{align}
as a consequence of \eqref{pcospr03} and the fact that $n>j(n)$.
We now estimate the left hand side of \eqref{pcov6}. For every $h\in \NN$, 
we have that
\begin{align*}
		{}&\frac{1}{2^p}\int_\Omega \left(\abs{\nabla(\phi_{n}-\phi_{j(n)})(u_{h})}^p 
		+\abs{(\phi_{n}-\phi_{j(n)})(u_{h})}^p \right)\,dx \\
		& \leq \frac{1}{2^p}\int_{\Omega\setminus(\hat{R}^{1/j(n)}\cap \hat{R}^{1/n})}  
		\left(\abs{\nabla(\phi_{n}-\phi_{j(n)})(u_{h})}^p 
		+\abs{(\phi_{n}-\phi_{j(n)})(u_{h})}^p \right)\,dx \\
		& \qquad\qquad \begin{array}{l}
			\text{since $\phi_{n}(u_{h})=0$ on $\hat{R}^{1/n}=\hat{R}^{1/n}(u_{h})$	and}\\ 
			\text{$\phi_{j(n)}(u_{h})=0$ on 
			$\hat{R}^{1/j(n)}=\hat{R}^{1/j(n)}(u_{h})$ by~\eqref{pcospr02}}
		\end{array}
		\\
		& \leq \int_{\Omega\setminus (\hat{R}^{1/j(n)}\cap \hat{R}^{1/n})} \left(U_{n,h}\right)^p\,dx 
		+\int_{\Omega\setminus (\hat{R}^{1/j(n)}\cap \hat{R}^{1/n})} \left(U_{j(n),h}\right)^p\,dx\\
		& \qquad\qquad\begin{array}{l}
			\text{where}~U_{m,h}:=\abs{\nabla (u_{h}-\phi_{m}(u_{h}))}
			+\abs{u_{h}-\phi_{m}(u_{h})}\in L^p(\Omega)
		\end{array}
		\\
		& \leq 
		\sup_{\abs{E}\leq\gamma(n)}
		\int_{E} \left(U_{n,h}\right)^p\,dx 
		+\sup_{\abs{E}\leq\gamma(n)}
		\int_{E} \left(U_{j(n),h}\right)^p\,dx \\
		& \qquad\qquad\begin{array}{l}
			\text{due to~\eqref{pcov7}, where $E$ can be any measurable subset of $\Omega$}
		\end{array}
		\\
		& \leq 
		\sup_{\abs{E}\leq\gamma(n)}
		\int_{E} \left[C_0\min\left\{n^{-1},v_{h}\right\}\right]^p dx
		+\sup_{\abs{E}\leq\gamma(n)}
		\int_{E} \left[C_0\min\left\{j(n)^{-1},v_{h}\right\}\right]^p dx
		\\
		& \qquad\qquad
		\begin{array}{l}
			\text{due to~\eqref{pcospr05} and Cavalieri's principle, where}\\
			\text{$E$ can be any measurable subset of $\RR^N$ and}\\
			v_h(x):=\cM(\abs{\nabla u_{h}}+\abs{u_{h}})(C_1^{-\frac{1}{N}}x)
			~~\text{(cf.~Definition~\ref{def:maxop})}
		\end{array}
		\\
		& \leq 2 C_0^p\sup_{\abs{E}\leq \gamma(n)}
 	  \int_{E} \left[\min\left\{j(n)^{-1},v_{h}(x)\right\}\right]^p\,dx.
\end{align*}
Hence, Lemma~\ref{lem:Lpequiintloc} applied to the sequence $(v_n)\subset L^p(\RR^N)$ 
(which is bounded in $L^p$ for $p>1$, due to Lemma~\ref{lem:MaxOpcont}) 
yields a subsequence $k(n)$ of $n$ such that \eqref{pcov6} is satisfied for $h(n):=k(\gamma(n))$.
\end{proof}
\begin{proof}[Proof of Proposition~\ref{prop:cospread2}]
Again, we write $\phi_n$ instead of $\phi_n^{(3)}$ below.
We first verify that $\nabla\phi_n(u_n)$ does not spread out in $L^{q}(\Omega;\RR^M)$.
Let $\eps>0$. Due to \eqref{coequiintB}, there exists $j=j(\eps)\in \NN$ such that
\begin{align}\label{pcos2e1}
	\norm{\nabla(\phi_n-\phi_{j})(u_n)}_{L^q(\Omega;\RR^M)}^q<\frac{\eps}{2}{2^{-q}}\quad\text{for every $n>j$}.
\end{align}  
It is enough to show that for any $\delta\geq 0$,
\begin{align}\label{pcos2e2}
	\int_\Omega \min\{\delta,\abs{\nabla \phi_n(u_n)}^q\}\,dx<\eps~
	\text{for every $n\in \NN$, if $\delta<\delta_0:=\frac{\eps}{2}(2C_3jS)^{-q}$.}
\end{align}
Here, $C_3>0$ is the constant introduced in \eqref{pcospr03}, and
$S>0$ is a constant such that $\norm{\phi_m(u_n)}_{W^{1,q}}\leq S$, uniformly in $m,n\in\NN$.
For $n\leq j$, \eqref{pcos2e2} is a consequence of the estimate
\begin{align}\label{pcos2e3}
  \sup_{n\in\NN,m\leq j}\int_\Omega \min\{\delta,\abs{\nabla \phi_m(u_n)}^q\}\,dx
  \leq \sup_{n\in\NN,m\leq j} \delta \abs{\{\phi_m(u_n)\neq 0\}}
  \leq \delta (C_3jS)^q,
\end{align}
which is due to \eqref{pcospr02} and \eqref{pcospr03}.
On the other hand, if $n>j$, 
\begin{align*}
  & \int_\Omega \min\{\delta,\abs{\nabla \phi_n(u_n)}^q\}\,dx\\
  &	\leq 2^q
  \int_\Omega \min\{\delta,\abs{\nabla \phi_n(u_n)-\nabla \phi_j(u_n)}^q\}\,dx
  +2^q\int_\Omega \min\{\delta,\abs{\nabla \phi_j(u_n)}^q\}\,dx\\
  &	\leq 2^q
  \int_\Omega \abs{\nabla \phi_n(u_n)-\nabla \phi_j(u_n)}^q\,dx
  +(2C_3jS)^q \delta 
  < \frac{\eps}{2}+\frac{\eps}{2}=\eps
\end{align*}
provided that $\delta<\delta_0$, due to \eqref{pcos2e1} and \eqref{pcos2e3}. 
Analogously, we infer that $\phi_n(u_n)$ does not spread out in $L^{q}(\Omega;\RR^M)$.
\end{proof}
\subsection{A decomposition lemma}
Consider a sequence $(u_n)$ in $W_0^{1,p}(\Omega)\cap W_0^{1,q}(\Omega)$ and 
three subsequences $(k_i(n))_n$ of $n$, $i=1,2,3$.
A subsequence of $(u_n)$ then decomposes as follows:
\begin{align}
	u_{k(n)}=U_n^{0}+U_n^{1}+U_n^{2}+U_n^{3}+U_n^{4},~~\text{with $k(n):=k_2(k_1(k_3({n})))$},
	\label{ldecu}
\end{align}
where
\begin{equation}\label{ldec0}
\begin{aligned}
	&\begin{alignedat}{5}
	&U_n^{0}&&:=\big[~~~~~I && \circ ~~~\phi_{k_3(n)}^{(1)} && \circ ~~~\phi_{k_1(k_3(n))}^{(2)}  
		&& \big](u_{k(n)}),\\ 
	&U_n^{1}&&:=\big[~~~~~I && \circ ~~~\phi_{k_3(n)}^{(1)} && \circ(I-\phi_{k_1(k_3(n))}^{(2)})
		&& \big](u_{k(n)}),\\ 
	&U_n^{2}&&:=\big[~~~~~I && \circ (I-\phi_{k_3(n)}^{(1)}) && \circ(I-\phi_{k_1(k_3(n))}^{(2)})
		&& \big](u_{k(n)}),\\
	&U_n^{3}&&:=\big[~~~~\phi_{n}^{(3)} && \circ (I-\phi_{k_3(n)}^{(1)}) 
		&& \circ ~~~\phi_{k_1(k_3(n))}^{(2)} && \big](u_{k(n)}),\\ 
	&U_n^{4}&&:=\big[(I-\phi_{n}^{(3)}) && \circ (I-\phi_{k_3(n)}^{(1)}) 
		&& \circ ~~~\phi_{k_1(k_3(n))}^{(2)} && \big](u_{k(n)}).
	\end{alignedat}
\end{aligned}
\end{equation}
%
%
Combining properties of $\phi_n^{(1)}$, $\phi_n^{(2)}$ and $\phi_n^{(3)}$ derived above, we obtain 
\begin{lem}[decomposition lemma]\label{lem:dec}
Let $1<q\leq p<\infty$
and let $(u_n)$ be a bounded sequence in $X:=W_0^{1,p}(\Omega;\RR^M)\cap W_0^{1,q}(\Omega;\RR^M)$,
where $\Omega$ is a Lipschitz domain in $\RR^N$ in the sense of Definition~\ref{def:Lipschitzdom}.
Then 
each of the sequences $U_n^0,\ldots,U_n^4$ defined in \eqref{ldec0} 
is bounded in $W_0^{1,p}(\Omega)\cap W_0^{1,q}(\Omega)$.
Moreover, there exist three subsequences $k_1(n)$, $k_2(n)$ and $k_3(n)$ of $n$ such that
\eqref{ldec1} below holds in each of the cases (i)--(iv) listed thereafter.
For any such choice, the component sequences $U_n^0,\ldots,U_n^4$ 
carry the following properties:
\begin{align*}
	\begin{alignedat}{2}
			&&\mbox{\normalfont{(a)}}&
			\begin{array}[t]{l}
				U_n^{0}~\text{is equiintegrable in $W^{1,p}$, $W^{1,q}$ and $L^{p^*}$.}
			\end{array}\\
			&&\mbox{\normalfont{(b)}}&
			\begin{array}[t]{l}
		 		\text{$U_n^{1}$ is tight in $W^{1,p}$, $W^{1,q}$ and $L^{p^*}$, and}~
		 		\text{$\abs{\{U_n^{1}\neq 0\}}\underset{n\to\infty}{\To} 0$.}
			\end{array}\\
			&&\mbox{\normalfont{(c)}}&
			\begin{array}[t]{l}
			  \text{$U_n^{2}\underset{n\to\infty}{\To} 0$ in $W^{1,\infty}_{\loc}$, and}~
				\text{$\abs{\{U_n^{2}\neq 0\}}\underset{n\to\infty}{\To} 0$.}
			\end{array}\\
			&&\mbox{\normalfont{(d)}}&
			\begin{array}[t]{l}
				\text{$U_n^{3}$ does not spread out in $W^{1,q}$, 
				$U_n^{3}\underset{n\to\infty}{\To} 0$ in $W^{1,\infty}_{\loc}$,}\\
				\text{and $U_n^{3}$ does not concentrate in $W^{1,p}$ and $L^{p^*}$.}\\
			\end{array}\\
			&&\mbox{\normalfont{(e)}}&
			\begin{array}[t]{l}
			  \text{$U_n^{4}\underset{n\to\infty}{\To} 0$ in $W^{1,\infty}$.} 
			\end{array}
	\end{alignedat}
\end{align*}
Here, $p^*:=\frac{Np}{N-p}$ 
if $p<N$, whereas $p^*\geq p$ can be chosen arbitrarily (but fixed) above
if $p\geq N$. The subsequences $k_1(n)$, $k_2(n)$ and $k_3(n)$ have been chosen in such a 
way that 
\begin{align}
& \begin{aligned}
	{}&\left(\phi^{(i)}_{m}-\phi^{(i)}_{j(m)}\right)(v_{h(m)})\underset{m\to \infty}{\To} 0
	~\text{in X, for every sequence $(j(m))\subset \NN$}\\
	{}&\text{with $j(m)\underset{n\to\infty}{\To}\infty$ and $j(m)<m$ for every $m\in \NN$,}
\end{aligned}\label{ldec1}
\end{align}
holds in each of the cases
\begin{alignat*}{2}
	&\mbox{\normalfont{(i)}} ~&& \begin{array}[t]{l}
	i=2,
	~h(m)=k_2(m),~~
	v_{h(m)}= u_{h(m)},~~
	\end{array}\\
	&\mbox{\normalfont{(ii)}} ~&& \begin{array}[t]{l}
	i=1,
	~h(m)=k_1(m),~~
	v_{h(m)}= \phi_{h(m)}^{(2)}(u_{k_2(h(m))}),
	\end{array}\\
	&\mbox{\normalfont{(iii)}} ~&& \begin{array}[t]{l}
	i=1,
	~h(m)=k_1(m),~~
	v_{h(m)}= (I-\phi_{h(m)}^{(2)})(u_{k_2(h(m))}),
	\end{array}\\
	\text{and}~&\mbox{\normalfont{(iv)}} ~&& \begin{array}[t]{l}
	i=3,
	~h(m)=k_3(m),~~
	v_{h(m)}= \big[(I-\phi_{h(m)}^{(1)}) \circ \phi_{k_1(h(m))}^{(2)} \big](u_{k_2(k_1(h(m)))}).
	\end{array}
\end{alignat*}
\end{lem}
\begin{rem}\label{rem:ldec} ~
\begin{myenum}
\item The terms "equiintegrable", "does not concetrate", "does not spread out" and "tight" 
are meant in the sense of Definition~\ref{def:concentrate}.
\item In (c) and (d), respectively, it is possible that $U_n^{2},U_n^{3}\notin W^{1,\infty}_{\loc}(\Omega;\RR^M)$,
for some or even every $n$. If $V_n$ is a sequence in $X$,
by saying that $V_n\to 0$ \emph{in} $W^{1,\infty}_{\loc}$, 
we mean that
\begin{align*}
	&\text{for every $R>0$, there is some $n_0=n_0(R)$ sufficiently large}\\
	&\text{such that $V_n\in W^{1,\infty}(B_R(0)\cap \Omega;\RR^M)$	for every $n\geq n_0$, and}\\
  &\norm{V_n}_{W^{1,\infty}(B_R(0)\cap \Omega;\RR^M)}\underset{n\to\infty}{\To} 0~~
  \text{for every fixed $R$}.
\end{align*}
\item The sequence $U_n^2$ eventually vanishes on fixed every bounded subset of $\Omega$, i.e., for every $R>0$, there is a $n_0=n_0(R)$ sufficiently large such that $U_n^2(x)=0$ for a.e.~$x\in \Omega\cap B_R(0)$ if $n\geq n_0$. This property does not necessarily hold for $U_n^3$, because $\phi_n^{(3)}$ typically modifies the support of its argument. If desired, it can be regained 
artificially in case of $U_n^3$. For that purpose, pass to another subsequence $u_{\tilde{k}(n)}$ of $u_n$
(with $\tilde{k}(n):=k(n^{N+1})$) and use
the following modification of \eqref{ldec0}:
\begin{equation*}
\begin{alignedat}{2}
	&&&
	\tilde{U}_n^{j} := U_{n^{N+1}}^{j}~~\text{for $j=1,2,3$},\\
  &&&\tilde{U}_n^{3} := (I-\phi_{n}^{(1)}) (U_{n^{N+1}}^{3})~\text{and}~
	\tilde{U}_n^{4} := U_{n^{N+1}}^{4}+\phi_{n}^{(1)} (U_{n^{N+1}}^{3}). 
\end{alignedat}
\end{equation*}
Here, the main change 
is the term $\phi_{n}^{(1)}(U_{n^{N+1}}^{3})$ which was transferred
from $U_n^{3}$ to $U_n^{4}$. It has no influence on the relative compactness
of $u_n$: As a matter of fact,
$\phi_{n}^{(1)}(U_{n^{N+1}}^{3})\to 0$ strongly in $W^{1,\infty}$, 
$W^{1,p}$ and $W^{1,q}$.
\item In (d), if $U_n^3$ does not contain a "moving bulk of mass" in $L^q$, i.e.,
if $U_n^3$ is vanishing in $L^q(\Omega;\RR^M)$), then $U_n^3\to 0$ strongly in
$L^q(\Omega;\RR^M)$, due to Lemma~\ref{lem:cospread3}.
In this case, we automatically have that
$U_n^3\to 0$ in $L^{p^*}(\Omega;\RR^M)$, too,
since $U_n^3$ does not concentrate in $L^{p^*}$ (and $q\leq p^*$).
%
\item In the assertion of Lemma~\ref{lem:dec}, condition \eqref{ldec1}
governing the choice of subsequences
and the exact definition of the component sequences
$U_n^j$ are explicitly stated
to provide an interface for the application of 
our abstract results, Theorem~\ref{thm:op} and Theorem~\ref{thm:abstractprop}.
\item The cases (i)--(iv) for \eqref{ldec1}
consist of all possible combinations of $i\in\{1,2,3\}$, $m=m(n)\in\NN$, $(h(m))_m\subset \NN$ and 
$(v_{h(m)})_m\subset W_0^{1,p}(\Omega)$ such that
$\phi^{(i)}_{m}(v_{h(m)})$ matches an expression contained 
in \eqref{ldec0}. 
\end{myenum}
\end{rem}
\begin{proof}[Proof of Lemma~\ref{lem:dec}]
Recall that each of the three families $\phi_n^{(i)}$ ($i=1,2,3$) satisfies \eqref{phi1}
and \eqref{phi3}, as shown in Proposition~\ref{prop:codomain}, 
Proposition~\ref{prop:coheight} and Proposition~\ref{prop:cospread}, respectively.
In particular, they are equibounded, whence $U_n^0,\ldots,U_n^4$ are bounded sequences.
The subsequences $k_i(n)$ are chosen 
using \eqref{phi3}: 
First, choose $k_2(n)$ such that \eqref{ldec1} is valid in case (i),
then choose $k_1(n)$ such that \eqref{ldec1} is valid both in case (ii) and (iii),
and finally choose $k_3(n)$ such that \eqref{ldec1} is valid in case (iv).
As to the properties of $U_n^0,\ldots,U_n^4$ in (a)--(e), first note that
\begin{alignat*}{2}
	&V^{\text{(i)}}_n&&:=\phi_{n}^{(2)}\big(u_{k_2(n)}\big)~~
	\text{does not concentrate in $W^{1,p}$ and $L^{p^*}$,} \\
	&V^{\text{(ii)}}_n&&:=\big[\phi_{n}^{(1)}\circ \phi_{k_1(n)}^{(2)}\big]
	\big(u_{k_2(k_1(n))}\big)~~
	\text{is tight in $W^{1,p}$, $W^{1,q}$ and $L^{p^*}$,}\\
	&V^{\text{(iii)}}_n&&:=\big[\phi_{n}^{(1)}\circ \big(I-\phi_{k_1(n)}^{(2)}\big)\big]
	\big(u_{k_2(k_1(n))}\big)~~
	\text{is tight in $W^{1,p}$, $W^{1,q}$ and $L^{p^*}$, and}\\
	&V^{\text{(iv)}}_n&&:=\big[\phi_{n}^{(3)}\circ \big(I-\phi_{k_3(n)}^{(1)}\big)\circ
	\phi_{k_1(k_3(n))}^{(2)}\big]	\big(u_{k(n)}\big)~~
	\text{does not spread out in $W^{1,q}$,}
\end{alignat*}	
due to \eqref{ldec1} 
and Proposition~\ref{prop:coheight2}, Proposition~\ref{prop:codomain} and 
Proposition~\ref{prop:cospread2}, respectively. (As to $V^{\text{(iv)}}_n$, recall that
$k(n)=k_2(k_1(k_3(n)))$.)
In particular, $U_n^0=V^{\text{(ii)}}_{k_3(n)}$ and $U_n^1=V^{\text{(iii)}}_{k_3(n)}$ are tight
in $W^{1,p}$, $W^{1,q}$ and $L^{p^*}$, and
$U_n^3=V^{\text{(iv)}}_n$ does not spread out in $W^{1,p}$. 
Moreover, 
\begin{align*}
  \phi^{(1)}_n \big(V^{\text{(i)}}_{k_1(n)}\big)~~\text{and}~~
  (I-\phi^{(1)}_n)\big(V^{\text{(i)}}_{k_1(n)}\big)~~
  \text{do not concentrate in $W^{1,p}$ and $L^{p^*}$},
\end{align*}
since $\phi^{(1)}_{n}$ does not interfere with this property.
Neither does $\phi_n^{(3)}$ (cf.~\eqref{pcospr01}), whence $U_n^0$ and $U_n^2$ 
do not concentrate in $W^{1,p}$ and $L^{p^*}$ (and thus they do not concentrate 
in $W^{1,q}$, cf.~Remark~\ref{rem:spread}).
Furthermore, $U_n^4\to 0$ in $W^{1,\infty}$ due to
\eqref{pcospr01}. 
As a consequence of the properties of $\phi_n^{(2)}$ collected
in Proposition~\ref{prop:coheight0} and the definition of $\phi_n^{(1)}$,
$\abs{\{x\in\Omega\mid U_n^s(x)\neq 0\}}\to 0$ as $n\to\infty$, for $s=1,2$.
The definition of $\phi_n^{(1)}$ also entails that
\begin{align*}
	U_n^{2}(x)=0~~\text{for every $x\in B_{k_3(n)}(0)\cap \Omega$}.
\end{align*}
In particular, $\big\|U_n^{2}\big\|_{W^{1,\infty}(B_R(0)\cap \Omega)}\to 0$ 
as $n\to\infty$, for every fixed $R>0$. 
Since $\big\|(I-\phi^{(3)}_n)(U_n^{2})\big\|_{W^{1,\infty}(\Omega)}
\to 0$ by \eqref{pcospr01},
the latter also holds for $U_n^{3}$ in place of $U_n^{2}$;
recall that $U_n^{3}=U_n^{2}-(I-\phi^{(3)}_n)(U_n^{2})$.
\end{proof}
%
\section{Quasilinear systems in divergence form}\label{sec:app1}
Let $\Omega\subset \RR^N$ ($N\in \NN$) be a Lipschitz domain in the sense of Definition~\ref{def:Lipschitzdom} (possibly unbounded), 
$1<q\leq p<N$ and $M\in \NN$.
We consider quasilinear differential operators of second order of the form
\begin{equation}\label{1Fdef}
\begin{aligned}
	&F:X\to X',~\text{with $X:=W_0^{1,p}(\Omega;\RR^M)\cap W_0^{1,q}(\Omega;\RR^M)$, defined by}\\
	&F(u)[\varphi]:=\int_\Omega Q(x,u,\nabla u):\nabla \varphi
	+\left[g_1(x,u,\nabla u)+g_2(x,u,\nabla u)\right]\cdot\varphi\,dx
\end{aligned}
\end{equation}
for every $u\in X$ and every test function $\varphi\in X$. Here, $:$ and $\cdot$ denote the euclidean scalar products in $\RR^{M\times N}$ and $\RR^M$, respectively. $X$ is a Banach space with the canonical norm
$\norm{u}_X:=\norm{u}_{W^{1,p}}+\norm{u}_{W^{1,q}}$ and $X'$ denotes the dual of $X$.
We assume that
\begin{alignat*}{3}
	&Q:~&&\Omega \times \RR^M \times \RR^{M\times N}\to \RR^{M\times N}~~
	&&\text{is a Carathéodory function and}\label{Q0}\tag{Q:0} \\
	&g_j:~&&\Omega \times \RR^M \times \RR^{M\times N}\to \RR^{M}~~
	&&\text{is a Carathéodory function for $j=1,2$},\label{g0}\tag{g:0}
\end{alignat*}
i.e, the functions are measurable in their first variable $x$ and continuous in the other variables for a.e.~$x\in \Omega$.
To state the remaining assumptions,
the following abbreviations are useful.
For $\alpha,\varrho \in [0,q-1]$, $x\in \Omega$ and $s,t\in [0,\infty)$
we define
\begin{align}
{}&\begin{alignedat}{2}
	&H_\alpha(x,s,t)&~:=~&C\,\big(h_q(x)+s+t\big)^{q-1-\alpha}
	\,+\,C\,\big(h_p(x)+s^{\frac{p^*}{p}}+t\big)^{p-1-\alpha},\\
	&I_\alpha(x,s,t)&~:=~&C\,\big(h_q(x)+s+t\big)^{q-1-\alpha}
	\,+\,C\,\big(h_p(x)+s^{\frac{p^*}{p}}+t\big)^{p-\frac{p}{p^*}-\alpha},\\
	&J_\varrho(x,s,t)&~:=~&
	\begin{aligned}[t]
	h_{\infty}(x)\big(s+t\big)^{q-1}
	\,&+\,[h_q(x)]^\varrho \big(h_q(x)+s+t\big)^{q-1-\varrho}\\
	&+\,[h_{p^*}(x)]^{\varrho}\big(h_{p^*}(x)+s+t^{\frac{p}{p^*}}\big)^{p^*-1-\varrho}.
	\end{aligned}\label{HIJdef}
&\end{alignedat}
\end{align}
Here, $C>0$ is a constant, $p^*:=\frac{pN}{N-p}>p$ is the critical Sobolev exponent, 
\begin{equation}\label{hqdef}
\begin{aligned}
  &\text{$h_r$ denotes a fixed nonnegative function in $L^r(\Omega)$ if $r\in [1,\infty)$, and}\\
  & \text{$h_\infty\in L^\infty(\Omega)$ is nonnegative and satisfies}~
  \norm{h_\infty}_{L^\infty(\Omega\setminus B_R(0))}\underset{R\to\infty}{\To} 0. 
\end{aligned}
\end{equation} 
We assume that $Q$, $g_1$ and $g_2$ satisfy the growth conditions
\begin{align*}
	{}&\begin{aligned}
	\abs{Q(x,\mu,\xi)}\leq H_0(x,\abs{\mu},\abs{\xi}),
	\end{aligned}\label{Qgrowth}\tag{Q:1}
\intertext{and}
	{}&\begin{aligned}
	\abs{g_1(x,\mu,\xi)}&\leq I_0(x,\abs{\mu},\abs{\xi}),	\\
	\abs{g_2(x,\mu,\xi)}&\leq J_\varrho(x,\abs{\mu},\abs{\xi}),\quad\text{for a $\varrho\in (0,q-1]$},
	\end{aligned}\label{ggrowth}\tag{g:1}
\end{align*}
for a.e.~$x\in \Omega$ and every $(\mu,\xi)\in \RR^M\times \RR^{M\times N}$. Here, 
recall that $W_0^{1,p}$ is continuously (but not compactly) 
embedded in $L^{p^*}$ for arbitrary domains.
For $Q$ and $g_1$, we also assume a kind of Hölder continuity with respect to the last two variables
which is uniform in $x$:
\begin{align*}
	{}&\begin{aligned}
		\text{There is an $\alpha\in (0,q-1]$ such that}
	\end{aligned}\\
	{}&\begin{aligned}
		{}&\abs{Q(x,\mu_1,\xi_1)-Q(x,\mu_2,\xi_2)}\\
		{}&~\leq H_\alpha(x,\abs{\mu_1}+\abs{\mu_2},\abs{\xi_1}+\abs{\xi_2})
	\big(\abs{\mu_1-\mu_2}+\abs{\mu_1-\mu_2}^{\frac{p^*}{p}}+\abs{\xi_1-\xi_2}\big)^\alpha,
	\end{aligned}\label{QHcont}\tag{Q:2}\\
	{}&\begin{aligned}
		{}&\abs{g_1(x,\mu_1,\xi_1)-g_1(x,\mu_2,\xi_2)}\\
		{}&~\leq I_\alpha(x,\abs{\mu_1}+\abs{\mu_2},\abs{\xi_1}+\abs{\xi_2})
	\big(\abs{\mu_1-\mu_2}+\abs{\mu_1-\mu_2}^{\frac{p^*}{p}}+\abs{\xi_1-\xi_2}\big)^\alpha,
	\end{aligned}\label{g1Hcont}\tag{g:2}
\end{align*}
for a.e.~$x\in \Omega$, every $\mu_1,\mu_2\in \RR^M$ and every $\xi_1,\xi_2\in \RR^{M\times N}$.
\begin{rem}\label{rem:W1pAssumptions}~
\begin{myenum}
\item Our assumptions include the case of a lower order term with critical growth in $u$, a prototype of which would be $g_1(u)=\tilde{\gamma} \abs{u}^{p^*-2}u$ with some $\tilde{\gamma} \in \RR$. Examples of  admissible terms including the gradient are
\begin{align*}
g_1(x,\nabla u)=\gamma(x) \abs{\nabla u}^{p \frac{p^*-1}{p^*}}~\text{or}~
g_1(x, u,\nabla u)=\Gamma(x) \abs{\nabla u}^{p \frac{p^*-(1+\eps)}{p^*}}\abs{u}^{\eps -1}u,
\end{align*} 
where $\eps\in (0,p^*-1]$ is arbitrary but fixed, $\gamma\in L^\infty(\Omega;\RR^M)$ 
and $\Gamma\in L^\infty(\Omega;\RR^{M\times M})$. 
\item In the context of this section, there is no technical necessity to assume that $F$ maps $X$ into its dual. We chose that framework simply for notational convenience. As a matter of fact, with suitably modified 
assumptions \eqref{Q0}--\eqref{QHcont} and \eqref{g0}--\eqref{g1Hcont}, the results below remain true
for example if $F:X_1\to (X_2)'$, with $X_1=W_0^{1,p_1}\cap W_0^{1,q_1}$ and
$X_2=W_0^{1,p_2}\cap W_0^{1,q_2}$, where $1<q_1\leq p_1<N$ and $1<q_2\leq p_2<N$ are arbitrary.
In this more general setting, the examples above can be replaced by
\begin{align*}
g_1(x,u,\nabla u)=\gamma(x) \abs{\nabla u}^{p_1\frac{p_2^*-1}{p_2^*}}
~\text{or}~
g_1(x,u,\nabla u)=\Gamma(x) \abs{\nabla u}^{p_1 \frac{p_2^*-(1+\eps)}{p_2^*}}\abs{u}^{\eps-1}u
\end{align*}
with $\eps\in (0,p_2^*-1]$.
\item Terms with so-called natural growth with respect to the gradient are ruled out, such as
\begin{align} \label{g1explcritgrad}
g_1(u,\nabla u)=\tilde{\gamma} \abs{\nabla u}^{p}u,~~\text{for some $\tilde{\gamma}\in \RR$}.
\end{align}
Even if for some reason $F$ is studied only on a set is bounded in $L^\infty$, 
this example does not fit into one of the generalized settings explained in (ii) (unless $p_2>N$, which is a much simpler case). 
On the other hand, it is not entirely clear when such a term leads to a well posed equation, even on bounded domains. If the leading part is elliptic ($Q$ is monotone in a suitable sense), typically an additional one-sided growth condition (in the scalar case) or an angle condition 
(in the vector case) are assumed to obtain existence of a solutions to the Dirichlet problem. 
Both lead to a sign condition on $\tilde{\gamma}$ in the example \eqref{g1explcritgrad}. See \cite{ChaZha92a}, \cite{La00a} and \cite{BeBo02a} and the references therein for some results in that context.
\end{myenum}
\end{rem}
The Nemyskii operator associated to $Q$ maps into the vector space $L^{\frac{q}{q-1}}+L^{\frac{p}{p-1}}$, and
the range of the Nemytskii operator associated to $g_1$ and $g_2$ lies in
$L^{\frac{q}{q-1}}+L^{\frac{p^*}{p^*-1}}$. 
With a natural norm, these sums are Banach spaces. Let
\begin{align*}
	&L^r_s(\Omega;V):=
	\left\{\begin{array}{ll}	
	L^r(\Omega;V)+L^s(\Omega;V)&~\text{if $r<s$,}\\
	L^r(\Omega;V)\cap L^s(\Omega;V)&~\text{if $r\geq s$,}
	\end{array}\right.
\end{align*}
where $s,r\in [1,\infty]$ and $V$ is an euclidean vector space.
The corresponding norm is given by
\begin{align*}
	{}&\norm{u}_{L^r_s}:=
	\left\{ \begin{alignedat}{2}
	& \inf \big\{ \norm{v}_{L^r}+\norm{w}_{L^s}\,\big|\,
	v\in L^r,~w\in L^s~\text{and}~v+w=u \big\} &&~\text{if $r<s$,}\\
	&\norm{u}_{L^r}+\norm{u}_{L^s} &&~\text{if $r\geq s$.} 
	\end{alignedat} \right. 
\end{align*}
\renewcommand{\labelenumi}{(\roman{enumi})}
\begin{rem} \label{rem:sumnorm} ~
\begin{myenum}
\item 
Roughly speaking,  functions in $L^r_s$ decay like functions in $L^{s}$ as $\abs{x}\to \infty$ ($x\in \Omega$), while their restrictions to sets of finite measure always belong to $L^{r}$.
If $\Omega$ has finite measure, we have $L_{s}^{r}(\Omega)=L^{s}(\Omega)\cap L^{r}(\Omega)=L^{\max\{s,r\}}(\Omega)$ and the three associated norms are equivalent, for arbitrary $r,s\in [1,\infty]$.
\item For $r,s\in [1,\infty]$, $s':=\frac{s}{s-1}$ and $r':=\frac{r}{r-1}$, 
we have Hölder's inequality in the form
\begin{align}\label{LpqHineq}
	\norm{u\cdot w}_{L^1}\leq \norm{u}_{{L^{r'}_{s'}}} \norm{w}_{L^{r}_{s}}
	~~\text{for $u\in L^{r'}_{s'}(\Omega;\RR)$ and $w\in L^{r}_{s}(\Omega;\RR)$}.
\end{align}
\item If $s,r\in (1,\infty)$,
the set $L_{s}^{r}$ also can be described as follows:
\begin{align*}
	L_{s}^{r}(\Omega;V)=
	\left\{u:\Omega\to V~\text{measurable} ~\left|~ \int_{\Omega}A_{s}^{r}(\abs{u(x)})\,dx<\infty \right.\right\},
\end{align*}
the Orlicz class with respect to the "$N$-function"
\begin{align*}
	A_{s}^{r}(t):=\left\{ \begin{array}{ll} 
	s^{-1}t^{s} \quad & \text{if $t\in [0,1]$,}\\ 
	r^{-1}t^{r}+s^{-1}-r^{-1} \quad &\text{if $t\in [1,\infty)$.}
	\end{array}\right.
\end{align*}
Moreover, $L_{s}^{r}$ is isomorphic to the Orlicz space $L_{A_{s}^{r}}$
associated to $A_{s}^{r}$. For general information on Orlicz spaces, the reader is referred to
\cite{Ad75B}. 
%
%
\item For $r,s\in (1,\infty)$, $s':=\frac{s}{s-1}$ and $r':=\frac{r}{r-1}$, 
$L_{s'}^{r'}$ is isomorphic (algebraically and topologically) to the dual of $L_{s}^{r}$,
with respect to the dual pairing
\begin{align*}
	J:L_{s'}^{r'}(\Omega;V)\to \left[L_{s}^{r}(\Omega;V)\right]',\quad 
	J(u)[\varphi]:=\int_{\Omega} u(x)\cdot \varphi(x)\,dx.
\end{align*}
Here, $u\cdot \varphi$ denotes the scalar product of $u$ and $\varphi$ in the euclidean vector space $V$.
\end{myenum} \vspace*{1ex}
\end{rem}

\begin{prop}\label{1propFcont}
Let $1<q\leq p<N$, $q':=\frac{q}{q-1}$, $p':=\frac{p}{p-1}$ and ${p^*}':=\frac{p^*}{p^*-1}$, where $p^*=\frac{pN}{N-p}$, and assume that  \eqref{Q0}--\eqref{QHcont} and \eqref{g0}--\eqref{g1Hcont} are satisfied.
Then the Nemytskii operators
\begin{alignat*}{2}
	&\tilde{Q}:\tilde{X}
	\to L_{q'}^{p'}(\Omega;\RR^{M\times N}),\quad 
	&&\tilde{Q}(u,U)(x):=Q(x,u(x),U(x)),\\
	&\tilde{g}_1: \tilde{X}
	\to L_{q'}^{{p^*}'}(\Omega;\RR^M),\quad
	&&\tilde{g}_1(u,U)(x):=g_1(x,u(x),U(x))\quad\text{and}\\
	&\tilde{g}_2: \tilde{X}
	\to L_{q'}^{{p^*}'}(\Omega;\RR^M),\quad
	&&\tilde{g}_2(u,U)(x):=g_2(x,u(x),U(x))
\end{alignat*}
are well defined and continuous on
\begin{align*}
	\tilde{X}:=\big[L^q(\Omega;\RR^M)\cap L^{p^*}(\Omega;\RR^M)\big]\times 
	\big[L^q(\Omega;\RR^{M\times N})\cap L^p(\Omega;\RR^{M \times N})\big], &\\
	\text{where}~\norm{(u,U)}_{\tilde{X}}:=\norm{u}_{L^q}+
	\norm{u}_{L^{p^*}}+\norm{U}_{L^q}+\norm{U}_{L^p}.&
\end{align*}
All three operators map bounded sets onto bounded sets.
Furthermore, $\tilde{Q}$ and $\tilde{g}_1$ are uniformly continuous on bounded subsets of $\tilde{X}$,
and $\tilde{g}_2$ satisfies
\begin{equation}\label{g2equiint}
\begin{alignedat}{3}
	&\sup_{(u,U)\in \tilde{W}}~\norm{\tilde{g}_2(u,U)}_{L_{q'}^{{p^*}'}(\Omega\setminus B_R(0);\RR^M)}
	&&\underset{R\to \infty}{~\To~} 0&&,\quad\text{and}\\
	&\sup_{(u,U)\in \tilde{W}}~\sup_{E\subset \Omega,\abs{E}\leq \delta}~ 
	\norm{\tilde{g}_2(u,U)}_{L_{q'}^{{p^*}'}(E;\RR^M)} 
	&&\underset{\delta \searrow 0}{~\To~} 0&&,
\end{alignedat}
\end{equation}
for every bounded subset $\tilde{W}\subset \tilde{X}$. 
\end{prop}
Before moving on to the proofs, we state the main results of this section.
As a matter of fact, Lemma~\ref{lem:dec}, Theorem~\ref{thm:op} and Theorem~\ref{thm:abstractprop} 
combined yield
\begin{thm}\label{lem:W1p-prop}
Let $1<q\leq p<N$ and let $\Omega$ be a Lipschitz domain in $\RR^N$ 
in the sense of Definition~\ref{def:Lipschitzdom}. 
Let $(u_n)$ be a bounded sequence in 
$X=W_0^{1,p}(\Omega;\RR^M)\cap W_0^{1,q}(\Omega;\RR^M)$
and let 
\begin{align*}
	u_{k(n)}=U_n^0+U_n^1+U_n^2+U_n^3+U_n^4
\end{align*}
denote a subsequence of $u_n$ chosen via Lemma~\ref{lem:dec}
(with $k(n)=k_2(k_1(k_3(n)))$), where $U_n^0,\ldots,U_n^4\in X$ are its bounded component
sequences with the properties (a)--(e) listed therein.
Furthermore,  let $F:X\to X'$ be the operator defined in \eqref{1Fdef},
and assume that \eqref{Q0}--\eqref{QHcont} 
and \eqref{g0}--\eqref{g1Hcont} are satisfied. 
Then we have that
\begin{equation}\label{Fconv}
\begin{aligned}
	{}\big[F(u_{k(n)})-F(U_n^{0})\big]~+~
	\sum_{i=1}^4\, \big[F(0)-F(U_n^{i})\big]
	~\underset{n\to\infty}{\To} 0
	~~\text{in $X'$.}
\end{aligned}
\end{equation}
Moreover, if $F(u_{n})$ converges in $X'$ then each of the five summands 
in \eqref{Fconv} converges to zero in $X'$.
\end{thm}
Theorem~\ref{lem:W1p-prop} can be used to characterize properness of $F$ 
on bounded subsets of $X$ 
as follows:
\begin{thm}\label{thm:W1p-prop}
Let $\Omega$ be a Lipschitz domain in $\RR^N$, let
$1<q\leq p<N$ and define $p^*=\frac{pN}{N-p}$. Furthermore, let $F:X\to X'$ denote the operator 
defined in \eqref{1Fdef}
and assume that \eqref{Q0}--\eqref{QHcont} 
and \eqref{g0}--\eqref{g1Hcont} are satisfied. 
Then
\begin{align*}
	\text{(i)}~&\begin{aligned}[t]
		\text{$F$ is proper on closed bounded subsets of $X$}
	\end{aligned}
\intertext{if and only if}	
  \text{(ii)}~&\begin{aligned}[t]
		&\text{every bounded sequence $(u_n)\subset X$ which satisfies}\\
		&\text{at least one of the five alternatives below, i.e.,}\\
		&\begin{alignedat}{2}
			&&\mbox{\normalfont{(a)}}&
			\begin{array}[t]{l}
				\text{$F(u_n)$~converges in $X'$, and}\\ 
				\text{$u_n$ is equiintegrable in $W^{1,p}$, $W^{1,q}$ and $L^{p^*}$,}
			\end{array}\\
			&&\mbox{\normalfont{(b)}}&
			\begin{array}[t]{l}
		 		\text{$F(u_n)\to F(0)$ in $X'$, 
		 		$\abs{\{u_n\neq 0\}}\to 0$, and}~\\ 
		 		\text{$u_n$ is tight in $W^{1,p}$, $W^{1,q}$ and $L^{p^*}$,}
			\end{array}\\
			&&\mbox{\normalfont{(c)}}&
			\begin{array}[t]{l}
			  \text{$F(u_n)\to F(0)$ in $X'$, $u_n\to 0$ in $W^{1,\infty}_{\loc}$, and
			  $\abs{\{u_n\neq 0\}}\to 0$,}
 			\end{array}\\
			&&\mbox{\normalfont{(d)}}&
			\begin{array}[t]{l}
				\text{$F(u_n)\to F(0)$ in $X'$, 
				$u_n$ does not spread out in $W^{1,q}$, }~\\
				\text{$u_n\to 0$ in $W^{1,\infty}_{\loc}$, and }
				\text{$u_n$ does not concentrate in $W^{1,p}$ and $L^{p^*}$,}
			\end{array}\\
			\text{or}~~
			&&\mbox{\normalfont{(e)}}&
			\begin{array}[t]{l}
			  \text{$F(u_n)\to F(0)$ in $X'$, 
			  and $u_n\to 0$ in $W^{1,\infty}$,} 
			\end{array}
	\end{alignedat}\\
	&\text{has a convergent subsequence.}
	\end{aligned}
\end{align*}
Here, the terms "equiintegrable", "does not concetrate", "does not spread out" and "tight" are meant in the sense of Definition~\ref{def:concentrate}.
\end{thm}
\begin{proof}
By Theorem~\ref{lem:W1p-prop}, (ii) implies (i) and the converse implication is trivial.
\end{proof} 
%
%
\begin{proof}[Proof of Proposition~\ref{1propFcont}]
By \eqref{Qgrowth}, 
$\tilde{Q}$ really maps into $L_{q'}^{p'}(\Omega;\RR^{M\times N})$.
Likewise, \eqref{ggrowth} ensures that $\tilde{g}_1$ and $\tilde{g}_2$ are well defined.
We only show the uniform continuity of $\tilde{Q}$ on bounded subsets of $\tilde{X}$. The proof for $\tilde{g}_1$ is analogous, the continuity of $\tilde{g}_2$ can be deduced from \eqref{g0} and \eqref{ggrowth} using standard arguments, and
\eqref{g2equiint} essentially is a consequence of the growth condition on $g_2$ and Hölder's inequality (the assumption $\varrho>0$ in \eqref{ggrowth} is crucial here, as is the decay of $h_{\infty}$ as $\abs{x}\to\infty$ in case of the first limit, cf.~\eqref{HIJdef} and \eqref{hqdef}).\\
Due to \eqref{QHcont},
\begin{align*}
	\abs{Q(x,u,U)-Q(x,v,V)}\leq S_1(x)^{q-1-\alpha}D(x)^\alpha+S_2(x)^{p-1-\alpha}D(x)^\alpha
\end{align*}
for every $(u,U),(v,V)\in \tilde{X}$, with the abbreviations
\begin{alignat*}{2}
	&S_1(x)&&:=\tilde{C}\big(\abs{U(x)}+\abs{u(x)}+\abs{V(x)}+\abs{v(x)}+h_q(x)\big),\\
	&S_2(x)&&:=\tilde{C}\big(\abs{U(x)}+\abs{u(x)}^\frac{p^*}{p}
	+\abs{V(x)}+\abs{v(x)}^\frac{p^*}{p}+h_p(x)\big)~\text{and}\\
	&D(x)&&:=\abs{U(x)-V(x)}+\abs{u(x)-v(x)}+\abs{u(x)-v(x)}^{\frac{p^*}{p}}.
\end{alignat*}
Here, 
$\tilde{C}=\tilde{C}(C,q,p,\alpha)\geq 0$ is a suitable constant.
Consequently,
\begin{align*}
	& \norm{Q(x,u,U)-Q(x,v,V)}_{L^{p'}_{q'}(\Omega;\RR^{M\times N})} \\
	&\qquad \begin{aligned} 
	\leq \,&
	\norm{(S_1(\cdot))^{q-1-\alpha} (D(\cdot))^\alpha}^{L^{q'}}
	+\norm{(S_2(\cdot))^{p-1-\alpha} (D(\cdot))^\alpha}^{L^{p'}}\\
	\leq \, &\norm{S_1}_{L^q}^{q-1-\alpha}
	\norm{D}_{L^q}^{\alpha}
	+\norm{S_2}_{L^p}^{p-1-\alpha}
	\norm{D}_{L^p}^{\alpha},
	\end{aligned}
\end{align*}
by Hölder's inequality. This implies that $Q$ is globally Hölder continuous 
on every bounded subset of $\tilde{X}$ since
\begin{align*}
	\norm{D}_{L^p}+\norm{D}_{L^q}
	\leq 3\norm{(u,U)-(v,V)}_{\tilde{X}},
\end{align*}
where we used that 
$\norm{u-v}_{L^r}\leq \norm{u-v}_{L^{p^*}}+\norm{u-v}_{L^q}$ for every $r\in [q,p^*]$.
\end{proof}
\subsubsection*{Proof of Theorem~\ref{lem:W1p-prop}}
The proof of Theorem~\ref{lem:W1p-prop} is based on the 
subsequent application of Theorem~\ref{thm:op} and Theorem~\ref{thm:abstractprop},
respectively, for the three types 
of truncation operators discussed in Section~\ref{secCOOpExamples}. 
In particular, this requires a suitable abstract setting
which is laid out below. 

As before, let $\phi_n^{(1)}$, $\phi_n^{(2)}$, and $\phi_n^{(3)}$ 
denote the truncation operators on 
\begin{align*}
	X=W_0^{1,p}(\Omega;\RR^M)\cap W_0^{1,q}(\Omega;\RR^M)
\end{align*}
introduced in Section~\ref{secCOOpExamples}.
We define
\begin{align*}
	\psi_n^{(j)}:X'\to Y,~~\psi_n^{(j)}(f)[\varphi]:=f[\phi^{(j)}_n(\varphi)],~~j=1,2,3,
\end{align*}
for every $f\in X'$ and every $\varphi\in X$ with $\norm{\varphi}_X\leq 1$,
where 
\begin{align*}
	&Y:=\mysetl{f:X\supset \overline{B_1(0)} \to \RR}{\norm{f}_Y<\infty},\\
	&\text{with}~\norm{f}_Y:=\sup\mysetr{\abs{f(\varphi)}}{\varphi\in \overline{B_1(0)}\subset X},
\end{align*}
is a normed vector space. Note that $X'$ 
is isometrically embedded in $Y$. 
Hence the function $F$ defined in \eqref{1Fdef}
can also be considered as a map from $X$ into $Y$. Naturally, the closure of the 
range of $F$ in $Y$ lies in $X'$. Accordingly, we define
\begin{align*}
	D:=X~~\text{and}~~R:=X'.
\end{align*}
Last but not least, we decompose 
\begin{align*}
  &\begin{aligned}
 	  &\text{$F=F_1+F_2$ with $F_{1,2}:X\to Y$, where for every $\varphi\in X$,}
 	\end{aligned}\\
 	&\begin{aligned}
 		F_1(u)[\varphi]:=&\int_\Omega Q(x,u,\nabla u):\nabla \varphi	
 		+g_1(x,u,\nabla u)\cdot\varphi\,dx\quad\text{and}\\
 		F_2(u)[\varphi]:=&\int_\Omega g_2(x,u,\nabla u)\cdot\varphi\,dx.
 	\end{aligned}
\end{align*}
\begin{rem}\label{rem:Y}~
\begin{myenum}
\item Note that $\psi_n^{(2)}$ and $\psi_n^{(3)}$ do not map into $X'$ due to the lack of linearity of 
$\phi_n^{(2)}$ and $\phi_n^{(3)}$. Our use of $Y$ compensates for that problem. 
\item Actually, $Y$ is a Banach space, but we do not exploit that fact.
\end{myenum}
\end{rem}
In a series of propositions, we now check the assumptions of
Theorem~\ref{thm:op} and Theorem~\ref{thm:abstractprop}, 
i.e., \eqref{copsi1}, \eqref{copsi2} and \eqref{F0}--\eqref{F2},
for each pair $\phi_n^{j}$, $\psi_n^{j}$ ($j=1,2,3$).
Throughout, we assume that
	$\Omega \subset \RR^N$ is a Lipschitz domain and $1<q\leq p<N$.
\begin{prop}\label{prop:F-psi1psi2}
For every $j=1,2,3$,
the family $\psi_n^{(j)}:X\to Y$ satisfies 
\eqref{copsi1} and \eqref{copsi2}. As to the latter, we even have that
\begin{align}\label{psi12limit}
	\psi_n^{(j)}(f)\underset{n\to \infty}{\To}f~\text{in $Y$, for every $f\in X'$}.
\end{align}
\end{prop}
\begin{proof}
Each $\psi_n^{(j)}$ is linear. 
Since all three families $\phi_n^{(j)}$ are equibounded, 
we infer \eqref{copsi1}.
For the proof of \eqref{psi12limit} we argue indirectly. 
Assume that there is
a sequence $(\varphi_n)\subset X$ with $\norm{\varphi_n}_X\leq 1$ for every $n$
such that
\begin{align*}
	(f-\psi_n^{(j)}[f])(\varphi_n)=f[(I-\phi_n^{(j)})(\varphi_n)]~~\text{does not converge to zero.}
\end{align*}
However, this is impossible, because $(I-\phi^{(j)}_n)(\varphi_n)\rightharpoonup 0$
weakly in $X$, for each $j$: Recall that
$\phi^{(1)}_n(\varphi_n)-\varphi_n=0$ on $B_n(0)$ by the definition
of $\phi^{(1)}_n$, 
that $\big|\{\phi^{(2)}_n(\varphi_n)-\varphi_n\neq 0\}\big|\leq (C_3)^pn^{-p}\to 0$ 
due to Proposition~\ref{prop:coheight0},
and that 
$\phi^{(3)}_n(\varphi_n)-\varphi_n\to 0$ in $W^{1,\infty}$ due to
Proposition~\ref{prop:cospread0}.
\end{proof}
\begin{prop}\label{prop:F-12F0F2}
Assume that \eqref{Q0}--\eqref{QHcont} 
and \eqref{g0}--\eqref{g1Hcont} hold.
Then for $j=1,2$, 
$F=F_1+F_2:X\to Y$ satisfies \eqref{F0} and \eqref{F2}
with $\phi_n=\phi_n^{(j)}$ and $\psi_n=\psi_n^{(j)}$.
\end{prop}
\begin{proof}
Due to Proposition~\ref{1propFcont}, we have \eqref{F0} 
as well as the uniform continuity of $F_1$ 
on bounded subsets of $X$ as required in \eqref{F2}. 
It remains to show that
\begin{align*}
	\sup_{w\in W}~ \sup_{\varphi\in X,\,\norm{\varphi}_X\leq 1}~
	F_2(w)[(I-\phi_n^{(j)})(\varphi)]~\underset{n\to\infty}{\To}0
\end{align*}
for every bounded set $W\subset X$.
For a proof, first recall that
$(I-\phi_n^{(j)})[\varphi]$ is bounded
in $X$, uniformly in $\varphi$ with $\norm{\varphi}_X\leq 1$,
due to \eqref{phi1}.
Combining this with the first line of \eqref{g2equiint} yields the assertion in case $j=1$.
In case $j=2$, one employs the second line of \eqref{g2equiint} instead,
as well as \eqref{pcoh02}, \eqref{pcoh04} and Hölder's inequality. 
\end{proof}
As it turns out in the proof of Theorem~\ref{lem:W1p-prop}, 
in case of the third truncation method ($\phi_n=\phi_n^{(3)}$ and $\psi_n=\psi_n^{(3)}$),
we need \eqref{F0}--\eqref{F2} just for $F=F_1$ (instead of $F=F_1+F_2$).
Essentially, this is due to the "subcritical" behavior of $F_2$ with respect 
to sequences converging to zero in $W^{1,p}_{\loc}$:
\begin{prop}\label{prop:F-3noF2}
Assume that \eqref{Q0}--\eqref{QHcont} 
and \eqref{g0}--\eqref{g1Hcont} hold,
and let $T_n$ be a bounded sequence in $X$
such that $T_n\to 0$ in $W^{1,p}_{\loc}$ and $L^{p^*}_{\loc}$.
Then $[F(T_n)-F(0)]-[F_1(T_n)-F_1(0)]\to 0$ in $X'$.
\end{prop}
\begin{proof}
It suffices to show that 
$\norm{g_2(\cdot,T_n,\nabla T_n)-g_2(\cdot,0,0)}_{L^{{p^*}'}_{q'}(\Omega)}\to 0$ 
as $n\to\infty$,
which is a consequence of the first line in \eqref{g2equiint} and
the continuity of the Nemytskii operator $\tilde{g}_2$ at $0$
(cf.~Proposition~\ref{1propFcont}).
\end{proof}
\begin{prop}\label{prop:F-3F0F2}
Assume that \eqref{Q0}--\eqref{QHcont} 
and \eqref{g0}--\eqref{g1Hcont} hold.
Then $F=F_1:X\to Y$ satisfies \eqref{F0} and \eqref{F2}
with $\phi_n=\phi_n^{(3)}$ and $\psi_n=\psi_n^{(3)}$.
\end{prop}
\begin{proof}
Due to Proposition~\ref{1propFcont}, we have \eqref{F0} 
as well as the uniform continuity of $F_1$ 
on bounded subsets of $X$ as required in \eqref{F2}. 
\end{proof}
This leaves us with the proof of \eqref{F1}, i.e.,
the compatibility of $\phi_n^{(j)}$ and $\psi_n^{(j)}$ with respect to $F$ in the sense of
\eqref{compat1} and \eqref{compat2}.
\begin{prop}\label{prop:F-1F1}
Assume that \eqref{Q0}--\eqref{QHcont} 
and \eqref{g0}--\eqref{g1Hcont} hold.
Then \eqref{F1} is satisfied for $F=F_1+F_2:X\to Y$ with
$\phi_n=\phi_n^{(1)}$ and $\psi_n=\psi_n^{(1)}$.
\end{prop}
\begin{proof}
By definition of $\phi_n^{(1)}$ and $\psi_n=\psi_n^{(1)}$,
the members of the sequences considered in \eqref{compat1} and \eqref{compat2} 
are zero for every $n\geq m+1$.
\end{proof}
\begin{prop}\label{prop:F-2F1}
Assume that \eqref{Q0}--\eqref{QHcont} 
and \eqref{g0}--\eqref{g1Hcont} hold.
Then \eqref{F1} is satisfied for $F=F_1+F_2:X\to Y$ with
$\phi_n=\phi_n^{(2)}$ and $\psi_n=\psi_n^{(2)}$.
\end{prop}
\begin{proof}
Let $W\subset X$ be a bounded set, $v,w\in W$ and $\varphi\in X$ with $\norm{\varphi}_X\leq 1$.
We first show \eqref{compat1}.
Observe that due to \eqref{pcoh04}, the measure of 
\begin{align*}
	E_n=E_n(w):=\mysetr{x\in\Omega}{(I-\phi_n^{(2)})(w)(x)\neq 0 }
\end{align*}
converges to zero as $n\to \infty$, uniformly in $w\in W$. 
With $w^{(n)}:=(I-\psi_n^{(2)})(w)$, we have that
\begin{align*} 
	&\abs{\left(\psi_m^{(2)}F_1\big(v+w^{(n)}\big)
	-\psi_m^{(2)}F_1(v)\right)[\varphi]}\\
	&\leq ~\begin{aligned}[t] \int_{\Omega} & \Big(
	\big|Q\big(x,v+w^{(n)},\nabla v+\nabla w^{(n)}\big)
	-Q\big(x,v,\nabla v\big)\big|
	\,\big|\nabla \phi_m^{(2)}(\varphi)\big|\\
	& ~~ +\, \big|g_1\big(x,v+w^{(n)},\nabla v+\nabla w^{(n)}\big)
	-g_1\big(x,v,\nabla v\big)\big|\,\big|\phi_m^{(2)}\varphi
	\big| \Big) \,dx
	\end{aligned}\\
	&\leq ~\begin{aligned}[t]
	\tilde{C}_1 m 
	\int_{ E_n } & \Big( 
	H_0\big(x,\big|v+w^{(n)}\big|,\big|\nabla v+\nabla w^{(n)}\big|\big)
	+H_0\big(x,\big|v\big|,\big|\nabla v\big|\big)\\
	&~~+I_0\big(x,\big|v+w^{(n)}\big|,\big|\nabla v+\nabla w^{(n)}\big|\big)
	+I_0\big(x,\big|v\big|,\big|\nabla v\big|\big)
	\Big) \,dx,
	\end{aligned}
\end{align*}
for a constant $\tilde{C}_1\geq 1$ independent on $n$, $m$, $v$, $w$ and $\varphi$,
due to \eqref{pcoh01} and the growth conditions \eqref{Qgrowth} and \eqref{ggrowth}, respectively.
Hölder's inequality, the continuous embedding of $W_0^{1,p}$ into $L^{p^*}$ 
and the equiboundedness of the family $\phi_n^{(2)}$ in $W^{1,p}$ and $W^{1,q}$ entail that
\begin{align*}
	&\norm{\psi_m^{(2)}F_1\big(v+(I-\psi_n^{(2)})(w)\big)-\psi_m^{(2)}F_1(v)}_Y\\
	& \qquad \leq ~\begin{aligned}[t]\tilde{C}_2 m\Big[&
	\left(\norm{w}_{W^{1,p}(\Omega)}^{p-1}+\norm{v}_{W^{1,p}(\Omega)}^{p-1}+1\right)
	\abs{E_n}^{\frac{1}{p}}\\
	&+\left(\norm{w}_{W^{1,q}(\Omega)}^{q-1}+\norm{v}_{W^{1,q}(\Omega)}^{q-1}+1\right)
	\abs{E_n}^{\frac{1}{q}}\\
	&+\left(\norm{w}_{W^{1,p}(\Omega)}^{p^*-1}+\norm{v}_{W^{1,p}(\Omega)}^{p^*-1}+1\right)
	\abs{E_n}^{\frac{1}{p^*}}
	\Big]
	\end{aligned}
\end{align*}
for a suitable constant $\tilde{C}_2\geq \tilde{C}_1$, which shows \eqref{compat1} for $F=F_1$. 
The proof for the general case $F=F_1+F_2$ is completely analogous.
For the proof of \eqref{compat2}, consider the set
\begin{align*}
	\tilde{E}_n=\tilde{E}_n(\varphi):=\mysetr{x\in\Omega}{(I-\phi_n^{(2)})(\varphi)(x) \neq 0}
\end{align*}
and its indicator function $\chi_{\tilde{E}_n}:\Omega\to \{0,1\}$,
where $\varphi\in X$ is a test function.
Just as for $E_n$ above, we have that 
\begin{align}\label{phi2compat20}
	|\tilde{E}_n(\varphi)|~\underset{n\to\infty}{\To} 0,
	~~\text{uniformly in $\varphi$ with $\norm{\varphi}_X\leq 1$}.
\end{align}
In the following, let
\begin{align*}
	u_n(x):=\chi_{\tilde{E}_n}(x)\cdot\phi_m^{(2)}(w)(x)
	~~\text{and}~~
	U_n(x):=\chi_{\tilde{E}_n}(x)\cdot\nabla\phi_m^{(2)}(w)(x)~~\text{for $x\in \Omega$}.
\end{align*}
Since $|\phi_m^{(2)}(w)|+|\nabla\phi_m^{(2)}(w)|\leq C_0 m$ a.e.~on $\Omega$, \eqref{phi2compat20} implies
that
\begin{align}\label{phi2compat21}
	\norm{u_n}_{L^{p^*}(\Omega;\RR^M)\cap L^q(\Omega;\RR^M)}
	\underset{n\to\infty}{\To} 0
	~~\text{and}~~
	\norm{U_n}_{L^p(\Omega;\RR^{M\times N})\cap L^q(\Omega;\RR^{M\times N})}
	\underset{n\to\infty}{\To} 0,
\end{align}
uniformly in $w\in W$ and $\varphi\in \overline{B_1(0)}\subset X$.
Moreover,
\begin{align*}
	&\abs{\left((I-\psi_n^{(2)})F_1\big(v+\phi_m^{(2)}(w)\big)-(I-\psi_n^{(2)})F_1(v)\right)[\varphi]}\\
	&\leq ~ \int_{ \Omega } 
	\begin{aligned}[t]
		\Big(\big|Q(x,v+u_n,\nabla v+U_n)-Q(x,v,\nabla v)\big|
		\,\big|\nabla (I-\phi_n^{(2)})(\varphi)\big|&\\
		+\,\big|g_1(x,v+u_n,\nabla v+U_n)-g_1(x,v,\nabla v)\big|
		\,\big|(I-\phi_n^{(2)})(\varphi)\big|&\Big)\,dx
	\end{aligned}\\
	&\leq ~ \tilde{C}_3 
	\begin{aligned}[t]
		\Big(&\norm{Q(\cdot,v+u_n,\nabla v+U_n)-Q(\cdot,v,\nabla v)}_{L_{q'}^{p'}}	 \\
		&+\,\norm{g_1(\cdot,v+u_n,\nabla v+U_n)-g_1(\cdot,v,\nabla v)}_{L_{q'}^{{p^*}'}} \Big) 
		\norm{\varphi}_{W^{1,q}\cap W^{1,p}}, 
	\end{aligned}
\end{align*}
with a suitable constant $\tilde{C}_3$,
due to Hölders inequality \eqref{LpqHineq} (as always $r':=\frac{r}{r-1}$ for $r\in (0,\infty)$) 
and the equiboundedness of the $\phi_n^{(2)}$.
Since the Nemytskii operators associated to $Q$ and $g_1$ 
are uniformly continuous on bounded sets as shown in 
Proposition \ref{1propFcont}, \eqref{phi2compat21} entails \eqref{compat2} for $F=F_1$.
This also proves the general case $F=F_1+F_2$ since $F_2$ can be neglected in \eqref{compat2} 
due to \eqref{F2}.
\end{proof}
\begin{prop}\label{prop:F-3F1}
Assume that \eqref{Q0}--\eqref{QHcont} 
and \eqref{g0}--\eqref{g1Hcont} hold.
Then \eqref{F1} is satisfied for $F=F_1:X\to Y$ with
$\phi_n=\phi_n^{(3)}$ and $\psi_n=\psi_n^{(3)}$.
\end{prop}
\begin{proof}
Let $W\subset X$ be a bounded set, $v,w\in W$ and $\varphi\in X$ with $\norm{\varphi}_X\leq 1$.
We first show \eqref{compat1}.
Observe that due to \eqref{pcospr03},
\begin{align}\label{pF-3F1-1}
	\abs{E(\varphi)}\leq (C_3)^q m^q,~~\text{where}~~
	E(\varphi):=\big\{x\in\Omega\mid \phi_m^{(3)}(\varphi)(x)\neq 0\big\},
\end{align}
for every $\varphi$ with $\norm{\varphi}_X\leq 1$.
Below, we use the abbreviations
\begin{align*}
	u_n(x)=u_n(x,\varphi,w,m)&:=\chi_{E(\varphi)}(x)\cdot(I-\phi_n^{(3)})(w)(x)
	~~\text{and}~~\\
	U_n(x)=U_n(x,\varphi,w,m)&:=\chi_{E(\varphi)}(x)\cdot \nabla[(I-\phi_n^{(3)})(w)](x),
\end{align*}
for $x\in \Omega$, where $\chi_{E(\varphi)}:\Omega\to \{0,1\}$ is the indicator function of $E(\varphi)$.
Since by \eqref{pcospr01}, $(I-\phi_n^{(3)})(w)\to 0$ in $W^{1,\infty}$,
uniformly in $w\in W$, \eqref{pF-3F1-1} implies
that
\begin{align}\label{pF-3F1-2}
	\norm{u_n}_{L^{p^*}(\Omega;\RR^M)\cap L^q(\Omega;\RR^M)}
	\underset{n\to\infty}{\To} 0
	~~\text{and}~~
	\norm{U_n}_{L^p(\Omega;\RR^{M\times N})\cap L^q(\Omega;\RR^{M\times N})}
	\underset{n\to\infty}{\To} 0,
\end{align}
uniformly in $w\in W$ and $\varphi\in \overline{B_1(0)}\subset X$.
Moreover,
\begin{align*}
	&\abs{\psi_m^{(3)}\big(F_1\big(v+(I-\phi_n^{(3)})(w)\big)\big)[\varphi]
	-\psi_m^{(3)}\big(F_1(v)\big)[\varphi]}\\
	&\leq 
	~\begin{aligned}[t] \int_{\Omega} &
	\Big(\big|Q\big(x,v+u_n,\nabla v+U_n\big)-Q\big(x,v,\nabla v\big)\big|
	\,\big|\nabla \phi_m^{(3)}(\varphi)\big|\\
	& ~~ +\, \big|g_1\big(x,v+u_n,\nabla v+U_n\big)
	-g_1\big(x,v,\nabla v\big)\big|\,\big|\phi_m^{(3)}(\varphi)
	\big|\Big)\,dx
	\end{aligned}\\
	&\leq ~ \tilde{C}_1 
	\begin{aligned}[t]
		\Big(&\norm{Q(\cdot,v+u_n,\nabla v+U_n)-Q(\cdot,v,\nabla v)}_{L_{q'}^{p'}} \\
		&+\,\norm{g_1(\cdot,v+u_n,\nabla v+U_n)-g_1(\cdot,v,\nabla v)}_{L_{q'}^{{p^*}'}} \Big)
		\norm{\varphi}_{W^{1,q}\cap W^{1,p}}, 
	\end{aligned}
\end{align*}
with a suitable constant $\tilde{C}_1$,
due to Hölders inequality \eqref{LpqHineq} (as always $r':=\frac{r}{r-1}$ for $r\in (0,\infty)$) 
and the boundedness of $\phi_m^{(3)}$.
Since the Nemytskii operators associated to $Q$ and $g_1$ 
are uniformly continuous on bounded sets as shown in 
Proposition \ref{1propFcont}, \eqref{pF-3F1-2} entails \eqref{compat2} for $F=F_1$.

For the proof of \eqref{compat2}, consider the set
\begin{align*}
	\tilde{E}(w):=\big\{ x\in \Omega\mid \phi_m^{(3)}(w)(x)\neq 0 \big\}
\end{align*}
and its indicator function $\chi_{\tilde{E}(w)}:\Omega\to\{0,1\}$.
To simpliy notation, we define
\begin{align*}
	w^{[m]}:=\phi_m^{(3)}(w),~~\text{whence}~~w^{[m]}\in W^{[m]}:=\phi_m^{(3)}(W) \subset X.
\end{align*}
For $F=F_1$, the expression on the left hand side of \eqref{compat2} can be estimated es follows:
\begin{align*}
	&\abs{\left((I-\psi_n^{(3)})F_1\big(v+w^{[m]}\big)-(I-\psi_n^{(3)})F_1(v)\right)[\varphi]}\\
	&\leq ~ \int_{ \tilde{E}(w) }
	\begin{aligned}[t]
		\Big(&\big|Q\big(x,v+w^{[m]},\nabla v+\nabla w^{[m]}\big)-Q\big(x,v,\nabla v\big)\big|
		\,\big|\nabla (I-\phi_n^{(3)})(\varphi)\big|\\
		&~+\,\big|g_1\big(x,v+w^{[m]},\nabla v+\nabla w^{[m]}\big)-g_1\big(x,v,\nabla v\big)\big|
		\,\big|(I-\phi_n^{(3)})(\varphi)\big|\Big)\,dx
	\end{aligned}\\
	&\leq ~\begin{aligned}[t]
	C_0 \frac{1}{n} 
	\int_{ \tilde{E}(w) } & \Big( 
	H_0\big(x,\big|v+w^{[m]}\big|,\big|\nabla v+\nabla w^{[m]}\big|\big)
	+H_0\big(x,\big|v\big|,\big|\nabla v\big|\big)\\
	&~~+I_0\big(x,\big|v+w^{[m]}\big|,\big|\nabla v+\nabla w^{[m]}\big|\big)
	+I_0\big(x,\big|v\big|,\big|\nabla v\big|\big)
	\Big)\,dx,
	\end{aligned}
\end{align*}
due to \eqref{pcospr01} and the growth conditions \eqref{Qgrowth} and \eqref{ggrowth}, respectively.
Since $|\tilde{E}(w)|\leq (C_3)^q m^q \norm{w}^q_X$ due to \eqref{pcospr03}
and both $W$ and $W^{[m]}$ are bounded in $X$ (the latter because $\phi^{(3)}_m$ is bounded),
this shows \eqref{compat1} for $F=F_1$.
\end{proof}
Having collected all preliminaries of 
Theorem~\ref{thm:op} and Theorem~\ref{thm:abstractprop}, 
we conclude the section with the proof of its main result.
\begin{proof}[Proof of Theorem~\ref{lem:W1p-prop}]
Assume that $(u_n)\subset X$ is a bounded sequence 
such that $F(u_n)$ converges in $X'$ (or, equivalently, in $Y$).
If $F(u_n)$ does not converge, the corresponding assertion can be obtained using
Theorem~\ref{thm:op} instead of Theorem~\ref{thm:abstractprop} below; the precise argument 
for this case is analogous to the one carried out in the proof of Theorem~\ref{lem:energyop}.
Recall that $k_1(n)$, $k_2(n)$ and $k_3(n)$ denote the subsequences 
obtained in Lemma~\ref{lem:dec}, that $k(n)=k_2(k_1(k_3(n)))$ and 
that $U_n^0\ldots,U_n^4$ denote the summands of $u_{k(n)}$ defined in
\eqref{ldec0}. 
By \eqref{ldec1} (in case (i) in Lemma~\ref{lem:dec}), 
Theorem~\ref{thm:abstractprop} is applicable to the sequence
$v_n:=u_{k_2(n)}$, with $\phi_n=\phi_n^{(2)}$ and $\psi_n=\psi_n^{(2)}$.
Thus
\begin{align} \label{Fconvphi3}
\begin{alignedat}{2}
	&F\big(u_{k_2(n)}\big)-F\big(\phi_n^{(2)}(u_{k_2(n)})\big)&&\underset{n\to\infty}{\To} 0
	~~\text{and}~~\\
	&F\big(0\big)-F\big((I-\phi_n^{(2)})(u_{k_2(n)})\big)&&\underset{n\to\infty}{\To} 0
\end{alignedat}
\end{align}
in $Y$ (and hence in $X'$). 
Repeating the argument with
the families $\phi_n=\phi_n^{(1)}$ and $\psi_n=\psi_n^{(1)}$, 
\begin{align*}
  &\text{for $v_n:=\phi^{(2)}_{k_1(n)} \big(u_{k_2(k_1(n))}\big)$ (cf.~case (ii) in Lemma~\ref{lem:dec}), and}\\
  &\text{for $v_n:=(I-\phi^{(2)}_{k_1(n)}) \big(u_{k_2(k_1(n))}\big)$ (cf.~case (iii)), respectively,}
\end{align*}
the limits in \eqref{Fconvphi3} decompose further and we infer that
\begin{align*}
	&F(u_{k(n)})-F(U_n^0)\underset{n\to\infty}{\To}0,~~
  F(0)-F(U_n^j)\underset{n\to\infty}{\To}0~~\text{for $j=1,2$, and}\\
  &F(0)-F(V_n)\underset{n\to\infty}{\To}0,~~\text{where}~~
  V_n:=\big[(I-\phi^{(1)}_n)\circ \phi^{(2)}_{k_1(n)}\big] (u_{k_2(k_1(n))}).
\end{align*}
Due to Proposition~\ref{prop:F-3noF2}, the previous line is equivalent to 
\begin{align} \label{Fconvphi4}
	F_1(0)-F_1(V_n)\underset{n\to\infty}{\To}0,
\end{align}
since $V_n\to 0$ in $W^{1,\infty}_{\loc}$ (in the sense of
Remark~\ref{rem:ldec} (ii), to be precise).
Again by
Theorem~\ref{thm:abstractprop}, this time applied to the sequence $v_n:=V_{k_3(n)}$
with $\phi_n=\phi^{(3)}_n$, $\psi_n=\psi^{(3)}_n$ and $F=F_1$, 
\eqref{Fconvphi4} turns into
\begin{align*} 
	F_1(0)-F_1(U_n^3)\underset{n\to\infty}{\To}0
	~~\text{and}~~
	F_1(0)-F_1(U_n^4)\underset{n\to\infty}{\To}0.
\end{align*}
Invoking Proposition~\ref{prop:F-3noF2} once more, we obtain
the remaining two limits of the assertion. Here, note that
$U_n^4\to 0$ in $W^{1,\infty}$ by \eqref{pcospr01}, and  
thus $U_n^3=V_{k_3(n)}-U_n^4\to 0$ in $W^{1,\infty}_{\loc}$.
\end{proof}
%
\section{Variational problems}\label{sec:app1b}
In Theorem~\ref{lem:W1p-prop}, 
we observed that the nonlinear operator $F$
behaves asymptotically additive with respect to
the decomposition of a sequence $u_n$ 
obtained in Lemma~\ref{lem:dec}. If the system \eqref{1Fdef} in the previous section 
has variational structure, i.e., if
$F(u)[\varphi]$ is the first derivative of a functional $E$ at $u$ in direction $\varphi$ for every such test function, it is natural to ask whether the energy 
exhibits the same behavior. 
Theorem~\ref{lem:energyop} below
answers this question in the affirmative.
In the following, we consider a functional
\begin{equation}\label{1Edef}
\begin{aligned}
	&E(u):=\int_\Omega W(x,u,\nabla u)\,dx,\quad E:X\to \RR,\\
	&\text{with $X:=W_0^{1,p}(\Omega;\RR^M)\cap W_0^{1,q}(\Omega;\RR^M)$,}
\end{aligned}	
\end{equation}
where $1<q\leq p<\infty$ are fixed and
\begin{align*}
	{}&\begin{aligned}
		&W=W_1+W_2~~\text{with two Carathéodory functions}\\
		&W_1,W_2:~\Omega \times \RR^M \times \RR^{M\times N}\to \RR.
	\end{aligned}\label{W0}\tag{W:0} 
\end{align*}
Below, we use the abbreviations
\begin{align}
{}&\begin{alignedat}{2}
	&\tilde{H}_\alpha(x,s,t)&~:=~&C\,\big(h_q(x)+s+t\big)^{q-\alpha}
	\,+\,C\,\big(h_p(x)+s^{\frac{p^*}{p}}+t\big)^{p-\alpha},\\
	&\tilde{J}_\varrho(x,s,t)&~:=~&\abs{h_{\infty}(x)}\big(s+t\big)^q
	\,+\,\abs{h_q(x)}^\varrho\big(h_q(x)+s+t\big)^{q-\varrho}\\
	&&&+\abs{h_{p^*}(x)}^{\varrho}(h_{p^*}(x)+s+t^{\frac{p}{p^*}})^{p^*-\varrho}.
&\end{alignedat} \label{EHalphadef}
\end{align}
for arbitrary $s,t\geq 0$, $\alpha \in [0,1]$ and $\varrho\in (0,1]$. Here, 
$C>0$ is a constant, $p^*:=\frac{pN}{N-p}$ is the critical Sobolev exponent and 
$h_q$, $h_p$, $h_{p^*}$ and $h_\infty$ are fixed nonnegative functions as defined in \eqref{hqdef} 
in Section~\ref{sec:app1}.
In addition to \eqref{W0}, we assume the growth conditions
\begin{align*}
	{}&\begin{aligned}
		\abs{W_1(x,\mu,\xi)}&\leq \,\tilde{H}_0(x,\abs{\mu},\abs{\xi})~~\text{and}\\
		\abs{W_2(x,\mu,\xi)}&\leq \,\tilde{J}_\varrho(x,\abs{\mu},\abs{\xi})~~\text{for a $\varrho\in (0,1]$}
	\end{aligned}\label{Wgrowth}\tag{W:1}
\end{align*}
and Hölder continuity of $W_1$ in the last two variables:
\begin{align*}
	{}&\begin{aligned}
		\text{There is an $\alpha\in (0,1]$ such that}
	\end{aligned}\\
	{}&\begin{aligned}
		{}&\abs{W_1(x,\mu_1,\xi_1)-W_1(x,\mu_2,\xi_2)}\\
		{}&~\leq \tilde{H}_{\alpha}(x,\abs{\mu_1}+\abs{\mu_2},\abs{\xi_1}+\abs{\xi_2})
	\big(\abs{\mu_1-\mu_2}+\abs{\mu_1-\mu_2}^{\frac{p^*}{p}}+\abs{\xi_1-\xi_2}\big)^\alpha.
	\end{aligned}\label{WHcont}\tag{W:2}
\end{align*}
Here, $x\in\Omega$, $\mu,\mu_1,\mu_2\in \RR^M$ and $\xi,\xi_1,\xi_2\in \RR^{M\times N}$ are arbitrary. As a consequence of \eqref{W0}--\eqref{WHcont} we have
\begin{prop}\label{1propWcont}
Let $\Omega\subset \RR^N$ be open, $1<q\leq p<N$ and $p^*=\frac{pN}{N-p}$, and assume that \eqref{W0}--\eqref{WHcont} are satisfied.
Then the Nemytskii operators
\begin{alignat*}{2}
	&\tilde{W}_j:\tilde{X}\to L^1(\Omega),\quad &&\tilde{W}_j(u,U)(x):=W_j(x,u(x),U(x))~~(j=1,2)
\end{alignat*}
are well defined and continuous on
\begin{align*}
	\tilde{X}:=\big[L^q(\Omega;\RR^M)\cap L^{p^*}(\Omega;\RR^M)\big]\times 
	\big[L^q(\Omega;\RR^{M\times N})\cap L^p(\Omega;\RR^{M \times N})\big], &
\end{align*}
Both operators map bounded sets onto bounded sets.
Furthermore, $\tilde{W}_1$ is uniformly continuous on bounded subsets of $\tilde{X}$,
and $\tilde{W}_2$ satisfies
\begin{equation}\label{EW2equiint}
\begin{alignedat}{3}
	&\sup_{(u,U)\in \tilde{B}}~
	\int_{\Omega\setminus B_R(0)} \big|W_2(x,u,U)\big|\,dx
	&&\underset{R\to \infty}{~\To~} 0&&,~~\text{and}\\
	&\sup_{(u,U)\in \tilde{B}}~\sup_{E\subset \Omega,\abs{E}\leq \delta}~ 
	\int_E \big|W_2(x,u,U)\big|\,dx
	&&\underset{\delta \searrow 0}{~\To~} 0&&,
\end{alignedat}
\end{equation}
for every bounded subset $\tilde{B}\subset \tilde{X}$. 
\end{prop}
\begin{proof}
Apart from obvious modifications, this is the same as the proof of Proposition~\ref{1propFcont}. 
We omit the details.
\end{proof}
Lemma~\ref{lem:dec} combined with Theorem~\ref{thm:op} and Theorem~\ref{thm:abstractprop}, respectively, 
leads to
\begin{thm}\label{lem:energyop}
Let $1< q\leq p<N$ and let $\Omega$ be a Lipschitz domain in $\RR^N$ in the sense of 
Definition~\ref{def:Lipschitzdom}. 
Let $(u_n)$ be a bounded sequence in 
$X=W_0^{1,p}(\Omega;\RR^M)\cap W_0^{1,q}(\Omega;\RR^M)$
and let 
\begin{align*}
	u_{k(n)}=U_n^0+U_n^1+U_n^2+U_n^3+U_n^4
\end{align*}
denote a subsequence of $u_n$ chosen via Lemma~\ref{lem:dec}
(with $k(n)=k_2(k_1(k_3(n)))$), where $U_n^0,\ldots,U_n^4\in X$ are its bounded component
sequences with the properties (a)--(e) listed therein.
Furthermore, let $E:X\to \RR$ be the functional defined in \eqref{1Edef},
and assume that \eqref{W0}--\eqref{WHcont} are satisfied.
Then we have that
\begin{equation}\label{calWconv}
\begin{aligned}
	&\big[\cW(u_{k(n)})-\cW(U_n^{0})\big]
	~+~
	\sum_{i=1}^4\, \big[\cW(0)-\cW(U_n^{i})\big]
	\underset{n\to\infty}{\To} 0~~\text{in $L^1(\Omega)$,}
\end{aligned}
\end{equation}
where $\cW(v):=W(\cdot,v,\nabla v)\in L^1(\Omega)$ for arbitrary $v\in X$.
In particular,
\begin{align*}
  &\big[E(u_{k(n)})-E(U_n^{0})\big]
	~+~
	\sum_{i=1}^4\, \big[E(0)-E(U^{i})\big]
	\underset{n\to\infty}{\To} 0~~\text{in $\RR$}.
\end{align*}
If $\cW(u_n)$ converges in $L^1(\Omega)$, then each of the five
summands in \eqref{calWconv} converges to zero.
\end{thm}
\begin{rem} \label{rem:W1p-PS}
If both Theorem~\ref{lem:W1p-prop} (with a suitable $F$, for instance the Fréchet derivative $DE$ of $E$) and Theorem~\ref{lem:energyop} 
(with a suitable $E$) are applicable to a bounded sequence $(u_n)\subset X$, it is clear 
that the same subsequence $u_{k(n)}$ and the same decomposition thereof can be used in both results simultaneously. The choice of $u_{k(n)}$ and its component sequences $U_n^0,\ldots,U_n^4$,
being due to Lemma~\ref{lem:dec}, only depends on the sequence $u_n$ and the truncations operators
involved, but not on $F$ or $E$. For the same reason,
one fixed decomposition of $u_{k(n)}$ can be used even if multiple (possibly infinitely many) operators $F$ and functionals $E$ are involved. 
\end{rem}
\subsubsection*{Proof of Theorem~\ref{lem:energyop}}
We follow the lines of the proof of Theorem~\ref{lem:W1p-prop} with 
different maps $\psi_n$. 
Again, we first lay out a suitable abstract setting.
We consider the map
\begin{align*}
  \cW:D\to R,~\text{with}~
  D:=X=(W_0^{1,p}\cap W_0^{1,q})(\Omega;\RR^M)~\text{and}~
  R:=Y:=L^1(\Omega).
\end{align*}
For each family of truncation operators $\phi_n^{(j)}$ ($j=1,2,3$) 
defined in Section~\ref{secCOOpExamples},
the corresponding maps $\psi_n^{(j)}:L^1(\Omega)\to L^1(\Omega)$ are defined as follows: 
For every $v\in L^1(\Omega)$, every $n\in \NN$ and every $x\in \Omega$ let
\begin{alignat}{3}
	\label{Epsi1def}
		&\psi_n^{(1)}(v)(x)&&:=\nu(\abs{x}-n)\cdot v(x),
\intertext{%
where $\nu:\RR\to [0,1]$ is the smooth nonincreasing function introduced in Definition~\ref{def:trunc1}.
(Hence $\psi_n^{(1)}$ and $\phi_n^{(1)}$ are the same map 
operating on different spaces.) Moreover, let
}	
	  \label{Epsi2def}
  	&\psi_n^{(2)}(v)(x)&&:=\eta_n(v(x)),\\
  	\label{Epsi3def}
	  &\psi_n^{(3)}(v)(x)&&:=v(x)-\eta_{1/n}(v(x)),
\end{alignat}
where $\eta_\lam:\RR\to\RR$ is the truncation of a scalar at height $\lam$ defined in Lemma~\ref{lem:Lpequiintloc}. 
For later use, note that $\eta_\lam:\RR\to\RR$ is globally Lipschitz continuous:
\begin{equation}
\begin{aligned}\label{etanLipschitz}
	\abs{\eta_\lam(t_1)-\eta_\lam(t_2)}
	\leq \abs{t_1-t_2},~~
	\text{for every $t_1,t_2\in \RR$ and every $\lam\in(0,\infty)$}.&
\end{aligned}
\end{equation}
Last but not least, we write $\cW=\cW_1+\cW_2$, where
\begin{align*}
 \cW_1(u):=W_1(\cdot,u,\nabla u)\in L^1(\Omega)~~\text{and}~~
 \cW_2(u):=W_2(\cdot,u,\nabla u)\in L^1(\Omega).
\end{align*}
The propositions below provide the
preliminaries of Theorem~\ref{thm:op} and Theorem~\ref{thm:abstractprop} in the present setting,
i.e., \eqref{copsi1}, \eqref{copsi2} and \eqref{F0}--\eqref{F2}. 
Throughout, we assume that
	$\Omega \subset \RR^N$ is a Lipschitz domain and $1< q\leq p<\infty$.
\begin{prop}\label{prop:E-psi1psi2}
For every $j=1,2,3$,
the family $\psi_n^{(j)}:L^1(\Omega)\to L^1(\Omega)$ satisfies 
\eqref{copsi1} and \eqref{copsi2}. As to the latter, we even have that
\begin{align}\label{Epsi12limit}
	\psi_n^{(j)}(f)\underset{n\to \infty}{\To}f~\text{in $L^1(\Omega)$, for every fixed $f\in L^1(\Omega)$}.
\end{align}
\end{prop}
\begin{proof}
The maps $\psi_n^{(1)}$ are linear and equibounded by definition.
For $j=2,3$, each $\psi_n^{(j)}$ is globally Lipschitz continuous with constant $1$
as a consequence of \eqref{etanLipschitz}.
Moreover, for $j=1,2,3$, we have that
$|\psi_n^{(j)}(f)(x)|\leq \abs{f(x)}$ and that 
$\psi_n^{(j)}(f)(x)\to f(x)$ for a.e.~$x\in \Omega$, 
due to the definitions of $\psi_n^{(j)}$.
This implies \eqref{Epsi12limit} by dominated convergence.
\end{proof}
\begin{prop}\label{prop:E-12F0F2}
Assume that \eqref{W0}--\eqref{WHcont} hold.
Then for $j=1,2$, 
$\tilde{F}:X\to L^1(\Omega)$, $\tilde{F}:=\cW=\cW_1+\cW_2$, satisfies \eqref{F0} and \eqref{F2}
with $\phi_n=\phi_n^{(j)}$ and $\psi_n=\psi_n^{(j)}$.
\end{prop}
\begin{proof}
Due to Proposition~\ref{1propWcont}, we have \eqref{F0} and the uniform continuity of $\cW_1$ on bounded subsets of $X$. Now fix a bounded set $S\subset X$. The first line in \eqref{EW2equiint} implies that
\begin{align*}
	&\big\|[(I-\psi_n^{(1)})\circ \cW](w)-[(I-\psi_n^{(1)})\circ \cW_1](w)\big\|_{L^1(\Omega)}\\
	&\qquad \leq \int_{\Omega \setminus B_{n}(0)} \abs{W_2(x,w,\nabla w)}\,dx
	~\underset{n\to\infty}{\To} 0,~~
	\text{uniformly in $w\in S$,}
\end{align*}
which concludes the proof of \eqref{F2} for $j=1$. For $j=2$, we consider the set
\begin{align*}
	\hat{T}_n=\hat{T}_n(w):=\left\{~x\in\Omega~:~ |\cW(w)(x)|>n 
	~\text{or}~|\cW_1(w)(x)|>n~ \right\},
\end{align*}
where $w\in S$.
Observe that 
\begin{align}\label{hatTnto0}
	|\hat{T}_n|\leq \frac{1}{n}\big(\norm{\cW(w)}_{L^1}+\norm{\cW_1(w)}_{L^1}\big)
	~\underset{n\to\infty}{\To} 0,~~\text{uniformly in $w\in S$},
\end{align}
since the images $\cW_1(S)$ and $\cW(S)$ are bounded in $L^1(\Omega)$.
Due to \eqref{etanLipschitz}, we have the estimate
\begin{align*}
	&\big\|[(I-\psi_n^{(2)})\circ \cW](w)-[(I-\psi_n^{(2)})\circ \cW_1](w)\big\|_{L^1(\Omega)}\\
	&= \int_{\hat{T}_n} \big|(I-\eta_n)\big(\cW(w)(x)\big)-
	(I-\eta_n)\big(\cW_1(w)(x)\big)\big|\,dx\\
	&\leq 2 \int_{\hat{T}_n} \abs{W_2(x,w,\nabla w)}\,dx.
\end{align*}
As a consequence of the second line in \eqref{EW2equiint},
the last integral above converges to zero, uniformly in $w\in S$. 
\end{proof}
Proposition~\ref{prop:F-3noF2}
and Proposition~\ref{prop:F-3F0F2} in the previous section 
are replaced as follows:
\begin{prop}\label{prop:E-3noF2}
Assume \eqref{W0}--\eqref{WHcont},
and let $T_n$ be a bounded sequence in $X$
such that $T_n\to 0$ in $W^{1,p}_{\loc}$ and $L^{p^*}_{\loc}$.
Then $[\cW(T_n)-\cW(0)]-[\cW_1(T_n)-\cW_1(0)]\to 0$ in $X'$.
\end{prop}
\begin{proof}
It suffices to show that 
$\norm{W_2(\cdot,T_n,\nabla T_n)-W_2(\cdot,0,0)}_{L^1(\Omega)}\to 0$ 
as $n\to\infty$,
which is a consequence of the first line in \eqref{EW2equiint} and
the continuity of the Nemytskii operator $\tilde{W}_2$ at $0$
(cf.~Proposition~\ref{1propWcont}).
\end{proof}
\begin{prop}\label{prop:E-3F0F2}
Assume that \eqref{W0}--\eqref{WHcont} hold.
Then $\tilde{F}:=\cW_1:X\to \RR$ satisfies \eqref{F0} and \eqref{F2}
with $\phi_n=\phi_n^{(3)}$ and $\psi_n=\psi_n^{(3)}$.
\end{prop}
\begin{proof}
Due to Proposition~\ref{1propWcont}, we have \eqref{F0} 
as well as the uniform continuity of $\tilde{F}_1:=\cW_1=\tilde{F}$ 
on bounded subsets of $X$ as required in \eqref{F2}. 
\end{proof}
This leaves us with the proof of \eqref{F1}, i.e.,
the compatibility of $\phi_n^{(j)}$ and $\psi_n^{(j)}$ with respect to $\cW$ 
(or $\cW_1$, in case $j=3$) in the sense of \eqref{compat1} and \eqref{compat2}.
\begin{prop}\label{prop:E-1F1}
Assume that \eqref{W0}--\eqref{WHcont} hold.
Then \eqref{F1} is satisfied for $\tilde{F}:X\to L^1(\Omega)$, $\tilde{F}:=\cW=\cW_1+\cW_2$, with
$\phi_n=\phi_n^{(1)}$ and $\psi_n=\psi_n^{(1)}$.
\end{prop}
\begin{proof}
By definition of $\phi_n^{(1)}$ and $\psi_n=\psi_n^{(1)}$,
the members of the sequences considered in \eqref{compat1} and \eqref{compat2} 
are zero for every $n\geq m+1$.
\end{proof}
\begin{prop}\label{prop:E-2F1}
Assume that \eqref{W0}--\eqref{WHcont} hold.
Then \eqref{F1} is satisfied for $\tilde{F}:X\to L^1(\Omega)$, $\tilde{F}:=\cW=\cW_1+\cW_2$, with
$\phi_n=\phi_n^{(2)}$ and $\psi_n=\psi_n^{(2)}$.
\end{prop}
\begin{proof}
Let $S\subset X$ be a bounded set, let $v,w\in S$ and fix $m\in\NN$.
First observe that the measure of $T_n=T_n(w):=\{ x\in\Omega\mid (I-\phi_n^{(2)})(w)(x)\neq 0 \}$ 
converges to zero as $n\to \infty$ uniformly in $w\in S$, due to
\eqref{pcoh04}. 
Hence
\begin{align*}
	&\big\|(\psi_m^{(2)}\circ\cW)\big(v+(I-\phi_n^{(2)})(w)\big)-(\psi_m^{(2)}\circ\cW)(v)\big\|_{L^1(\Omega)}\\
	&\qquad= ~
	\int_{T_n}
	\big|\eta_m(\cW\big(v+(I-\phi_n^{(2)})(w)\big)(x))
	-\eta_m(\cW(v)\big|\,dx
	~\leq~ 
	2 m \abs{T_n} ~\underset{n\to\infty}{\To}~ 0,
\end{align*}
uniformly in $v,w\in S$, which shows \eqref{compat1} for $\tilde{F}=\cW$. 
Next, we prove \eqref{compat2} for $\tilde{F}=\cW_1$. Consider the set
\begin{align*}
	\tilde{T}_n=\tilde{T}_n(v,w):=\left\{~x\in\Omega~\left|~ 
	|\cW_1\big(v+\phi_m^{(2)}(w)\big)(x)|>n 
	~\text{or}~|\cW_1(v)(x)|>n \right.\right\} 
\end{align*}
and its indicator function
$\chi_{\tilde{T}_n}$.
Since $\cW_1$ maps bounded subsets of $X$ onto bounded sets in $L^1(\Omega)$, 
we have that 
\begin{align}\label{Ephi2compat20}
	|\tilde{T}_n|~\underset{n\to\infty}{\To} 0,~~\text{uniformly in $v,w\in S$}.
\end{align}
To shorten notation in the following, we define
\begin{align*}
	u_n(x):=\chi_{\tilde{T}_n}(x)\cdot\phi_m^{(2)}(w)(x)
	~~\text{and}~~
	U_n(x):=\chi_{\tilde{T}_n}(x)\cdot\nabla\phi_m^{(2)}(w)(x)~~\text{for $x\in \Omega$}.
\end{align*}
Since $|\phi_m^{(2)}(w)|+|\nabla\phi_m^{(2)}(w)|\leq C_0 m$ a.e.~on $\Omega$, 
\eqref{Ephi2compat20} implies that
\begin{align}\label{Ephi2compat21}
	\norm{u_n}_{L^{p^*}(\Omega;\RR^M)\cap L^q(\Omega;\RR^M)}
	\underset{n\to\infty}{\To} 0
	~~\text{and}~~
	\norm{U_n}_{L^p(\Omega;\RR^{M\times N})\cap L^q(\Omega;\RR^{M\times N})}
	\underset{n\to\infty}{\To} 0,
\end{align}
uniformly in $v,w\in S$.
Moreover,
\begin{align*}
	&\big\|(I-\psi_n^{(2)})[\cW_1\big(v+\phi_m^{(2)}(w)\big)]-(I-\psi_n^{(2)})[\cW_1
	(v)]\big\|_{L^1(\Omega)}\\
	&= ~ \int_{ \tilde{T}_n }
	\begin{aligned}[t]
		\big|&(I-\eta_n)[\cW_1\big(v+\phi_m^{(2)}(w)\big)(x)]
		-(I-\eta_n)[\cW_1(v)(x)]\big|	\,dx
	\end{aligned}\\
	&\leq ~ 2\int_{ \Omega } 
	\begin{aligned}[t]
		\big|W_1\big(x,v+u_n,\nabla v+U_n)-W_1(x,v,\nabla v)\big|
		\,&\,dx,
	\end{aligned}
\end{align*}
due to \eqref{etanLipschitz}. By the uniform continuity of the Nemytskii operator associated to $W_1$
on bounded sets as shown in Proposition \ref{1propWcont}, 
\eqref{Ephi2compat21} entails \eqref{compat2} for $\tilde{F}=\cW_1$.
Using \eqref{F2}, we infer \eqref{compat2} for $\tilde{F}=\cW$.
\end{proof}
\begin{prop}\label{prop:E-3F1} 
Assume that \eqref{W0}--\eqref{WHcont} hold.
Then \eqref{F1} is satisfied for $\tilde{F}:=\cW_1:X\to L^1(\Omega)$, with
$\phi_n=\phi_n^{(3)}$ and $\psi_n=\psi_n^{(3)}$.
\end{prop}
\begin{proof}
Let $S\subset X$ be a bounded set, let $v,w\in S$ and fix $m\in\NN$.
First observe that the measure of $T(w):=\{ x\in\Omega\mid \phi_m^{(3)}(w)(x)\neq 0 \}$ 
is bounded uniformly in $w\in S$ (for fixed $m$), due to
\eqref{pcospr04}. 
Hence
\begin{align*}
	&\big\|(I-\psi_n^{(3)})\big[\cW_1\big(v+\phi_m^{(3)}(w)\big)\big]-(I-\psi_n^{(3)})
	\big[\cW_1(v)\big]\big\|_{L^1(\Omega)}\\
	&\quad= ~
	\int_{T(w)} \big|\eta_{1/n}\big[\cW_1\big(v+\phi_m^{(3)}(w)\big)(x)\big]
	-\eta_{1/n}\big[\cW_1(v)(x)\big]\big|\,dx	\\ 
	&\quad \leq~ 
	2 \frac{1}{n} \abs{T(w)}~\underset{n\to\infty}{\To}~0,~~
	\text{uniformly in $w\in S$},
\end{align*}
which shows \eqref{compat2} for $\tilde{F}=\cW_1$.
Next, we prove \eqref{compat1},
consider the set
\begin{align*}
	\tilde{T}_n:=\mysetr{x\in\Omega}{
	\big|\cW_1\big(v+(I-\phi_n^{(3)})(w)\big)(x)\big|>\frac{1}{m}~~\text{or}~~
	\big|\cW_1(v)(x)\big|>\frac{1}{m} },
\end{align*}
($\tilde{T}_n=\tilde{T}_n(v,w)$) and its indicator function
$\chi_{\tilde{T}_n}:\Omega\to \{0,1\}$.
We have that
\begin{align}\label{pE-3F1-1}
	\big|\tilde{T}_n\big|~\leq~
	m\left(\big\|\cW_1\big(v+(I-\phi_n^{(3)})(w)\big)\big\|_{L^1}
	+\big\|\cW_1(v)\big\|_{L^1}\right)
	~\leq~ 2mC,
\end{align}
where $C>0$ is a constant independent of $n\in \NN$ and $v,w\in S$. Here, recall that
$S$ is bounded, the family $\phi_n^{(3)}$ is equibounded, and
$\cW_1$ maps bounded sets onto bounded set, cf.~Proposition~\ref{1propWcont}.
In the following, let
\begin{align*}
	u_n(x):=\chi_{\tilde{T}_n}(x)\cdot \big(I-\phi_n^{(3)}\big)(w)(x)
	~~\text{and}~~
	U_n(x):=\chi_{\tilde{T}_n}(x)\cdot \nabla \big[\big(I-\phi_n^{(3)}\big)(w)\big](x).
\end{align*}
Combining \eqref{pE-3F1-1} and \eqref{pcospr01},
we infer that
\begin{align}\label{pE-3F1-2} 
	\norm{u_n}_{[L^{p^*}\cap L^q](\Omega;\RR^M)}
	\underset{n\to\infty}{\To} 0
	~~\text{and}~~
	\norm{U_n}_{[L^p\cap L^q](\Omega;\RR^{M\times N})}	
	\underset{n\to\infty}{\To} 0,
\end{align}
uniformly in $v,w\in S$.
Moreover,
\begin{align*}
	&\Big\|\psi_m^{(3)}\big[\cW_1\big(v+(I-\phi_n^{(3)})(w)\big)\big]-\psi_m^{(3)}
	\big[\cW_1(v)\big]\Big\|_{L^1(\Omega)}\\
	&= ~ \int_{ \Omega }
	\begin{aligned}[t]
		\Big|&(I-\eta_{1/m})\big[W_1(x,v+u_n,\nabla v+U_n)\big]
		-(I-\eta_{1/m})\big[W_1(x,v,\nabla v)\big]\Big|\,dx
	\end{aligned}\\
	&\leq ~ 2\int_{ \Omega }
	\begin{aligned}[t]
		\Big|&W_1(x,v+u_n,\nabla v+U_n)-W_1\big(x,v,\nabla v\big)\Big|\,dx
	\end{aligned}
\end{align*}
due to \eqref{etanLipschitz}. By the uniform continuity of the Nemytskii operator associated to $W_1$
on bounded sets as shown in Proposition \ref{1propWcont}, 
\eqref{pE-3F1-2} entails \eqref{compat1} for $\tilde{F}=\cW_1$.
\end{proof}
Having collected all preliminaries of 
Theorem~\ref{thm:op} and Theorem~\ref{thm:abstractprop}, 
we conclude the section with the proof of its main result.
\begin{proof}[Proof of Theorem~\ref{lem:energyop}]
Let $(u_n)\subset X$ be a bounded sequence.
We only show \eqref{calWconv}, with the help of Theorem~\ref{thm:op}. If
$\cW(u_n)$ converges in $L^1(\Omega)$, the convergence of each summand
of the sequence in \eqref{calWconv} can be obtained using 
Theorem~\ref{thm:abstractprop} instead,
closely following the lines of the proof of Theorem~\ref{lem:W1p-prop}.
Recall that $k_1(n)$, $k_2(n)$ and $k_3(n)$ denote the subsequences 
obtained in Lemma~\ref{lem:dec}, that $k(n)=k_2(k_1(k_3(n)))$ and 
that $U_n^0\ldots,U_n^4$ denote the summands of $u_{k(n)}$ defined in
\eqref{ldec0}. 
By \eqref{ldec1}, 
Theorem~\ref{thm:op} is applicable to the sequence $v_n$
with the families $\phi_n=\phi_n^{(i)}$ and $\psi_n=\psi_n^{(i)}$, where
\begin{alignat*}{3}
	&i=2~~\text{and}~~&&v_n:=u_{k_2(n)}~~&&\text{(cf.~case (i) in Lemma~\ref{lem:dec})},\\
  &i=1~~\text{and}~~&&v_n:=\phi^{(2)}_{k_1(n)} \big(u_{k_2(k_1(n))}\big)
  ~~&&\text{(cf.~case (ii)), and}\\
  &i=1~~\text{and}~~&&v_n:=(I-\phi^{(2)}_{k_1(n)}) \big(u_{k_2(k_1(n))}\big) 
  ~~&&\text{(cf.~case (iii)),}
\end{alignat*}
respectively. Thus we infer that
\begin{equation}
\begin{aligned}\label{lEop1}
	\left[\cW(u_{k(n)})-\cW(U_n^0)\right]
	&+\left[\cW(0)-\cW(U_n^1)\right]\\
	&+\left[\cW(0)-\cW(U_n^2)\right]
	+\left[\cW(0)-\cW(V_{k_3(n)})\right]
		\underset{n\to\infty}{\To}0,\\
\end{aligned}
\end{equation}
where $V^3_m :=\big[(I-\phi^{(1)}_m)\circ \phi^{(2)}_{k_1(m)}\big] (u_{k_2(k_1(m))})$ for $m\in\NN$.
Moreover, 
\begin{align} \label{lEop2}
	\left[\cW_1\big(V^3_{k_3(n)}\big)-\cW_1\big(U_n^3\big)\right]
	+\left[\cW_1(0)-\cW_1\big(U_n^4\big)\right]\underset{n\to\infty}{\To}0,
\end{align}
due to Theorem~\ref{thm:op}, this time applied to the sequence $v_n:=V_{k_3(n)}$
with $\phi_n=\phi^{(3)}_n$, $\psi_n=\psi^{(3)}_n$ and $\tilde{F}=\cW_1$
(cf.~case (iv) in Lemma~\ref{lem:dec}). 
Finally, as a consequence of Proposition~\ref{prop:F-3noF2}, we 
may replace $\cW_1$ with $\cW$ in \ref{lEop2}, 
since $V^3_{k_3(n)}\to 0$ in $W^{1,\infty}_{\loc}$ (in the sense of
Remark~\ref{rem:ldec} (ii), to be precise), $U_n^4\to 0$ in $W^{1,\infty}$ by \eqref{pcospr01}, and  
thus $U_n^3=V_{k_3(n)}-U_n^4\to 0$ in $W^{1,\infty}_{\loc}$. Hence
\eqref{lEop2} turns into
\begin{align} \label{lEop3}
	\left[\cW\big(V^3_{k_3(n)}\big)-\cW\big(U_n^3\big)\right]
	+\left[\cW(0)-\cW\big(U_n^4\big)\right]\underset{n\to\infty}{\To}0.
\end{align}
The assertion is a consequence of \eqref{lEop1} and \eqref{lEop3} combined.
\end{proof}
%
\section{Quasilinear systems in $\mbf{W^{2,p}}$ with $\mbf{p>N}$}\label{sec:app2}
In this section, we recover a characterization of properness obtained in 
\cite{RaStu01a} for $\Omega=\RR^N$ and in \cite{GeStu05a} for exterior domains with boundary of class $C^2$), namely that properness is equivalent to
"properness at 0" for elliptic operators.
Moreover, we generalize it to domains with boundary of class $C^2$, including the case of unbounded boundary. 
This is possible since our approach allows us to avoid the limit problems used in 
\cite{RaStu01a} and \cite{GeStu05a} which are hard to define in general
if $\Omega$ is unbounded but not an exterior domain. For the same reason, we do not need well defined 
limits, asymptotic periodicity or similar properties of the coefficient functions as 
$\abs{x}\to \infty$ ($x\in \Omega$).
As in \cite{RaStu01a} and \cite{GeStu05a} we consider a map
\begin{equation}\label{2Fdef}
\begin{aligned}
	&F:X \to L^p(\Omega;\RR^M),~
	\text{with $X:=W^{2,p}(\Omega;\RR^M)\cap W_0^{1,p}(\Omega;\RR^M)$},\\
	&F(u):=-\sum_{\alpha,\beta=1}^N A_{\alpha\beta}(\cdot,u,\nabla u)\,\partial^2_{\alpha\beta}u+b(\cdot,u,\nabla u).
\end{aligned}
\end{equation}
Here, $\Omega\subset \RR^N$, $p>N$,
\begin{alignat*}{3}
	&\begin{aligned}
	A_{\alpha\beta}:~\Omega \times \RR^M \times \RR^{M\times N}\to \RR^{M\times M}\quad
	\text{is a Carathéodory function,}
	\end{aligned}
	\label{2A0}\tag{A:0}
\intertext{and for every $\mu\in \RR^M$, $\xi\in \RR^{M\times N}$ and for a.e.~$x\in\Omega$,}
	&\begin{aligned}
	&b(x,\mu,\xi)=b_0(x,\mu,\xi)\cdot \mu+B(x,\mu,\xi):\xi+b(x,0,0),
	~~\text{where}\\
	&b_0:~\Omega \times \RR^M \times \RR^{M\times N}\to \RR^{M\times M}~\text{and}~
	B:~\Omega \times \RR^M \times \RR^{M\times N}\to \RR^{M\times(M\times N)}\\
	&\text{are Carathéodory functions and $b(\cdot,0,0)\in L^p(\Omega;\RR^M)$.}
	\end{aligned}\label{2b0}\tag{b:0}
\end{alignat*}
In the first line of \eqref{2b0}, $\cdot$ and $:$ denote the euclidean scalar products in $\RR^M$ and 
$\RR^{M\times N}$, respectively.
Our growth conditions are
\begin{align*}
  &\abs{A_{\alpha\beta}(x,\mu,\xi)} \leq f(\abs{\mu}+\abs{\xi}),
  \label{2Agrowth}\tag{A:1}\\
	&\abs{b_0(x,\mu,\xi)}\leq f(\abs{\mu}+\abs{\xi})
	\quad\text{and}\quad
	\abs{B(x,\mu,\xi)} \leq f(\abs{\mu}+\abs{\xi}),
	\label{2bgrowth}\tag{b:1}
\end{align*}
for every $\alpha,\beta\in \{1,\ldots,N\}$, $\mu\in \RR^M$ and $\xi\in \RR^{M\times N}$ and a.e.~$x\in \Omega$ 
where \begin{equation}\label{fhpdef}
	f:[0,\infty)\to [1,\infty)\quad\text{is an increasing function}.
\end{equation}
Moreover, we assume that the $A_{\alpha\beta}$, $b_0$ and $B$
are 
"equicontinuous bundle maps" (this is the terminology used 
in \cite{RaStu01a} and \cite{GeStu05a}, apart from the fact that we do not assume 
continuity in $x$ which does not come in until later),  
i.e., the family
\begin{alignat*}{3}
	&\begin{aligned}
		(A_{\alpha_\beta}(x,\cdot,\cdot))_{x\in \Omega}~
		\text{is equicontinuous at every point in $\RR^M\times \RR^{M\times N}$,}
	\end{aligned}
	\label{2Aequi}\tag{A:2}
\intertext{and the families} 	 
	&\begin{aligned}
		&(b_0(x,\cdot,\cdot))_{x\in \Omega}~\text{and}~(B(x,\cdot,\cdot))_{x\in \Omega}~
		\text{are equicontinuous}\\
		&\text{at every point in $\RR^M\times \RR^{M\times N}$.}
	\end{aligned}	
	\label{2bequi}\tag{b:2}
\end{alignat*}
We do not need stronger growth conditions for large values of $\mu$ or $\xi$, because $X$ is continuously embedded in $W^{1,\infty}(\Omega;\RR^M)$, even in the Hölder space $C^{1,\alpha}(\overline{\Omega})$ with $\alpha=1-N/p$, at least as long as the boundary of $\Omega$ is sufficiently smooth, a Lipschitz domain for instance. Here, recall that on unbounded domains, the Hölder norm is defined as
\begin{align*}
	\norm{u}_{C^{0,\alpha}(\overline{\Omega})}:=
	\sup \mysetl{\frac{\abs{u(x)-u(y)}}{\abs{x-y}^\alpha}}
	{x,y\in \overline{\Omega},~x\neq y,~\abs{x-y}\leq 1}+\norm{u}_{C^{0}(\overline{\Omega})}
\end{align*}
As a consequence of our assumptions, we have
uniform continuity of the corresponding Nemytskii operators on bounded sets:
\begin{prop}\label{prop:2Fcont}
Let 
$p\in (N,\infty)$ 
and assume that \eqref{2A0}--\eqref{2Aequi} and \eqref{2b0}--\eqref{2bequi} are satisfied.
Then the Nemytskii operators
\begin{alignat*}{2}
	&\tilde{A}_{\alpha\beta}:\tilde{X}
	\to L^{\infty}(\Omega;\RR^{M\times M}),\qquad 
	&&\tilde{A}_{\alpha\beta}(u,U)(x):=A_{\alpha\beta}(x,u(x),U(x)),\\
	&\tilde{b}: \tilde{X}
	\to L^p(\Omega;\RR^{M}),\qquad 
	&&\tilde{b}(u,U)(x):=b(x,u(x),U(x))
\end{alignat*}
are well defined and map bounded sets onto bounded sets. Moreover, they are uniformly 
continuous on bounded subsets of $\tilde{X}$. Here,
\begin{align*}
	&\tilde{X}:=\big[L^{p}(\Omega;\RR^M)\cap L^{\infty}(\Omega;\RR^M)\big]\times 
	\big[L^p(\Omega;\RR^{M \times N})\cap L^\infty(\Omega;\RR^{M \times N}) \big], 
\end{align*}
with norm $\norm{(u,U)}_{\tilde{X}}:=\norm{u}_{L^p}+\norm{u}_{L^\infty}+\norm{U}_{L^p}+\norm{U}_{L^\infty}$.
\end{prop}
Proofs are collected at the end of this section.
In the present framework, the abstract result, Theorem~\ref{thm:abstractprop}, 
leads to 
\begin{thm}\label{thm:2Fprop}
Let $p>N$ and let $\Omega\subset \RR^N$ be a domain which is sufficiently smooth such that 
$W^{2,p}(\Omega)$ is continuously embedded in $W^{1,\infty}(\Omega)$. Moreover, 
let $F$ be the operator defined in \eqref{2Fdef} and assume that \eqref{2A0}--\eqref{2Aequi} and \eqref{2b0}--\eqref{2bequi} are satisfied. Then 
\begin{align*}
	\text{(i)}~&\text{$F$ is proper on closed bounded subsets of $X$}
\intertext{if and only if}
 	\text{(ii)}~&\text{every bounded sequence $(u_n)\subset X$ such that} \\
	&\begin{alignedat}{3}
		&&&&\text{(a)}&
		\begin{array}[t]{l}
			\text{$F(u_n)$ converges in $L^p(\Omega;\RR^M)$ and}\\
			\text{$u_n$ is tight in $W^{2,p}(\Omega;\RR^M)$}
		\end{array}\\
		&\text{or}~~&&&\text{(b)}&
		\begin{array}[t]{l}
			\text{$F(u_n)\to F(0)$ in $L^p(\Omega;\RR^M)$ and for every $R>0$}\\
			\text{there exists an $n_0\in \NN$ such that $u_n(x)=0$}\\
			\text{for every $x\in B_R(0)\cap \Omega$ and every $n\geq n_0$}
		\end{array}
	\end{alignedat}\\
	&\text{has a subsequence which converges in $X$.}
\end{align*}
Here, "tight" is meant in the sense of
Definition~\ref{def:concentrate}.
\end{thm}
As a matter of fact, we are always able to find a convergent subsequence
of every bounded sequence of type (a) in Theorem~\ref{thm:2Fprop}
if we assume that the leading part of $F$ is elliptic
in the following sense:
\begin{align*}
	&\begin{aligned}
	&\det\Big(\sum_{\alpha,\beta=1}^N\eta_\alpha\eta_\beta A_{\alpha\beta}(x,\mu,\xi)\Big)
	\geq \gamma(x,\mu,\xi) \abs{\eta}^{2M},	\\
	&\text{where}~~\gamma:\overline{\Omega}\times\RR^M\times\RR^{M\times N}\to (0,\infty)~~\text{is continuous},
	\end{aligned}
	\label{2Aelliptic}\tag{A:3}
\end{align*}
for every $\eta=(\eta_\alpha)\in\RR^N$, $\mu\in \RR^M$, $\xi\in \RR^{M\times N}$ and a.e.~$x\in \Omega$.
Here, note that the assumption on $\gamma$ implies that 
$\gamma$ is bounded from below by a positive constant
on every compact subset of $\overline{\Omega}\times\RR^M\times\RR^{M\times N}$.
We also need that
\begin{align*}
	A_{\alpha\beta}:\overline{\Omega}\times\RR^M\times \RR^{M\times N}\to \RR^{M\times M}~\text{is continuous,}
	\label{2Acont}\tag{A:4}
\end{align*}
as well as some additional regularity of the boundary of $\Omega$, cf.~\eqref{C2Omega} below.
We exploit \eqref{2Aelliptic} and \eqref{2Acont} in form of the following a priori estimate
for linear elliptic systems with continuous coefficients:
\begin{lem}[cf.~Theorem 17 in \cite{Ko62a}]\label{lem:2Aapriori}
Assume \eqref{2Aelliptic} and \eqref{2Acont}, and let $\Omega_k\subset \Omega$ be 
a {bounded} subdomain with boundary of class $C^2$.
Then for every fixed 
$v\in C^{1}(\overline{\Omega}_k;\RR^M)$ and every $w\in W^{2,p}(\Omega_k;\RR^M)\cap W_0^{1,p}(\Omega_k;\RR^M)$
we have that
\begin{align*}
	\norm{w}_{W^{2,p}(\Omega_k;\RR^M)}\leq 
	C \big(\big\|\tilde{A}(v)w\big\|_{L^p(\Omega_k;\RR^M)}+\norm{w}_{L^p(\Omega_k;\RR^M)}\big),&\\
	\text{where}~\tilde{A}(v)w:=\sum_{\alpha,\beta=1}^N A_{\alpha\beta}(\cdot,v,\nabla v)\partial^2_{\alpha\beta}w&.
\end{align*}
Here, $C\geq 0$ is a constant which depends on $v$, $\Omega_k$ and $A_{\alpha\beta}$ 
but not on $w$. 
\end{lem}
In view of Lemma~\ref{lem:2Aapriori}, we are lead to study domains $\Omega$ of the following type:
\begin{equation}\label{C2Omega}
\begin{aligned}
	&\text{$\Omega=\bigcup{}_{k\in\NN} \Omega_k$,~~where
	$(\Omega_k)_{k\in\NN}$ is a sequence of subdomains}\\
	&\text{with boundary of class $C^2$	such that every bounded subset $B\subset \Omega$}\\
	&\text{is contained in $\Omega_{k_0}$ for a suitable $k_0=k_0(B)\in \NN$, and}\\
	&\text{$\Omega_k\subset B_{k}(0)\cap \Omega_{k+1}$~~and~~
	$\partial \Omega_k\cap \overline{B_{k-\frac{1}{2}}(0)}\subset \partial\Omega$,
	~~for every $k$.}
\end{aligned}
\end{equation}
This leads to a characterization of properness for elliptic operators:
\begin{cor}\label{cor:2Fprop}
Let $p>N$ and let $\Omega\subset\RR^N$ be a domain
which is sufficiently smooth such that 
$W^{2,p}(\Omega)$ is continuously embedded in $W^{1,\infty}(\Omega)$
and which satisfies \eqref{C2Omega}.
Moreover, let $F$ be the operator defined in \eqref{2Fdef} and assume that \eqref{2A0}--\eqref{2Acont} and \eqref{2b0}--\eqref{2bequi} are valid. Then the following
three statements are equivalent:
\begin{align*}
	(i)~&\begin{aligned}[t]
	&\text{$F$ is proper on closed bounded subsets of $X$.}
	\end{aligned}\\
	(ii)~&\begin{aligned}[t]
	&\text{Every bounded sequence $(u_n)\subset X$ such that}\\
	&\begin{alignedat}{2}
		\qquad &&& \begin{array}[t]{l}
		\text{$F(u_n)\to F(0)$ in $L^p(\Omega)$ and}\\
		\text{for every $R>0$ there exists an $n_0\in \NN$ such that}\\
		\text{$u_n=0$ on $B_R(0)\cap \Omega$ for every $n\geq n_0$}
		\end{array}
	\end{alignedat}\\
	&\text{has a convergent subsequence.}
	\end{aligned}\\
	(iii)~& \begin{aligned}[t]
	&\text{Every bounded sequence $(u_n)\subset X$ such that}\\
	&\begin{alignedat}{2}
		\qquad &&&\begin{array}[t]{l}
		\text{$F(u_n)\to F(0)$ in $L^p(\Omega)$ and 
		$u_n\rightharpoonup	0$ weakly in $X$}
		\end{array}
	\end{alignedat}\\
	&\text{has a convergent subsequence.}
	\end{aligned}
\end{align*}
\end{cor}
\begin{rem} \label{rem:cor2Fprop}
If $\Omega=\RR^N$ and the coefficient functions of $F$ are asymptotically periodic (as $\abs{x}\to\infty$), 
the equivalence of (i) and (iii) was shown in Theorem 6.5 in \cite{RaStu01a} (see also Theorem 5.7 in
\cite{GeStu05a}).
Also note that the characterization given in (ii) is easier to check than (iii), since it provides a stronger assumption on the sequence: every bounded sequence of functions
which eventually leaves every bounded subset of $\Omega$ as in (ii) converges to zero weakly in $X$,
but the converse does not hold in general. 
\end{rem}
%
\subsection*{Proofs of Proposition~\ref{prop:2Fcont}, Theorem~\ref{thm:2Fprop} and Corollary~\ref{cor:2Fprop}}
\begin{proof}[Proof of Proposition~\ref{prop:2Fcont}]
In view of \eqref{2A0}, \eqref{2Agrowth}, \eqref{2b0} and \eqref{2bgrowth} it is clear that $\tilde{A}_{\alpha\beta}$ and $\tilde{b}$ are well defined. 
Now consider a bounded set $W\subset \tilde{X}$. 
First, we claim that $\tilde{A}_{\alpha\beta}$ is uniformly continuous on $W$.
There is a compact set $K\subset\RR^M\times \RR^{M\times M}$ such that
for every $(u,U)\in W$, $(u(x),U(x))\in K$ for a.e.~$x\in \Omega$.
Moreover, \eqref{2A0} and \eqref{2Aequi} imply that 
the family $(b(x,\cdot,\cdot))_{x\in \Omega}$ is uniformly equicontinuous on every compact subset of
$K\subset\RR^M\times \RR^{M\times M}$, i.e., for every $\eps>0$ there exists a $\delta>0$ such that
\begin{equation}\label{Auniequicont} 
\begin{aligned}
	&\abs{A_{\alpha\beta}(x,\mu_1,\xi_1)-A_{\alpha\beta}(x,\mu_2,\xi_2)}<\frac{\eps}{2}\\
	&\text{whenever $(\mu_1,\xi_1),(\mu_2,\xi_2)\in K$ and $\abs{(\mu_1,\xi_1)-(\mu_2,\xi_2)}<\delta$,}
\end{aligned}
\end{equation}
for a.e.~$x\in \Omega$, cf.~Lemma 2.5(i) in \cite{GeStu05a}. 
Since the range of all functions in $W$ is contained in $K$,
\eqref{Auniequicont} entails the assertion.
Analogously, the Nemytskii operators associated to $b_0$ and $B$, i.e.,
\begin{align*}
	&\tilde{b}_0:\tilde{X}
	\to L^{\infty}(\Omega;\RR^M)\quad \text{and}\quad
	\tilde{B}:\tilde{X}	\to L^{\infty}(\Omega;\RR^{M\times N})
\end{align*}
are uniformly continuous on $W$. By \eqref{2b0}, this entails the uniform continuity of $\tilde{b}$ on $W$.
\end{proof}
\begin{proof}[Proof of Theorem~\ref{thm:2Fprop}]
First observe that "only if" is trivial. We obtain the converse implication 
with the aid of Theorem~\ref{thm:abstractprop}. For that purpose define $R:=Y:=L^p(\Omega;\RR^M)$ and let 
$\phi^{(1)}_n:W^{2,p}(\Omega;\RR^M)\to W^{2,p}(\Omega;\RR^M)$
be the operators introduced in Definition~\ref{def:trunc1} which cut off the outer regions of $\Omega$. 
The $\phi^{(1)}_n$ form a family of truncation operators on $X$ 
due to Proposition~\ref{prop:codomain}. Here, note that $\phi^{(1)}_n$ maps the closed subspace $X=W_0^{1,p}(\Omega;\RR^M)\cap W^{2,p}(\Omega;\RR^M)\subset W^{2,p}(\Omega;\RR^M)$ onto itself. Moreover, let $\psi_n:L^p(\Omega;\RR^M)\to L^p(\Omega;\RR^M)$ denote the same operators, now considered as an endomorphism of $L^p(\Omega;\RR^M)$. The family $\psi_n$ satisfies \eqref{copsi1}, and in place of \eqref{copsi2}, we even have the stronger property 
\begin{align*}
	\psi_n(h)\to h\quad\text{in $L^p(\Omega;\RR^M)$, for every fixed $h\in L^p(\Omega;\RR^M)$}.
\end{align*}
Due to Proposition~\ref{prop:2Fcont}, both
\eqref{F0} and \eqref{F2} hold, the latter with $F_1:=F$ and $F_2:=0$.
Last but not least, \eqref{F1} is also valid since the expressions on the left hand side 
of \eqref{compat1} and \eqref{compat2}, respectively, both are zero for every $n\geq m+1$. 
Hence Theorem~\ref{thm:abstractprop} is applicable with 
$\phi_n=\phi_n^{(1)}$ and $\psi_n=\psi_n^{(1)}$. 
To show that $F$ is proper, fix a bounded sequence $(u_n)\subset X$
such that $F(u_n)\to G$ in $L^p(\Omega;\RR^M)$ for a suitable limit $G$. Due to \eqref{phi3},
there is a subsequence $u_{k(n)}$ of $u_n$ satisfying \eqref{uknequiint},
whence Theorem~\ref{thm:abstractprop} yields that 
\begin{alignat*}{2}
&F(U_n^0)\to G,~~&&\text{where}~~U_n^0:=\phi^{(1)}_n(u_{k(n)}),~~\text{and}\\
&F(U_n^1)\to F(0),~~&&\text{where}~~U_n^1:=(I-\phi^{(1)}_n)(u_{k(n)}).
\end{alignat*}
Together with Proposition~\ref{prop:codomain}, this implies that $U_n^0$ and $U_n^1$ are bounded sequences in $X$ with the additional properties stated in (a) and (b), respectively. Thus by assumption,
both $U_n^0$ and $U_n^1$ converge in $X$, at least for a suitable subsequence (not relabeled), whence
the corresponding subsequence of $u_{k(n)}=U_n^0+U_n^1$ converges as well.
\end{proof} 
For the proof of Corollary~\ref{cor:2Fprop}, we need the following auxiliary result based on Lemma~\ref{lem:2Aapriori}:
\begin{prop}\label{prop:2Alocconv}
Let $p>N$ and assume that the domain $\Omega\subset\RR^N$ satisfies \eqref{C2Omega}. Moreover, assume that \eqref{2Aelliptic}, \eqref{2Acont}, \eqref{2b0} and \eqref{2bgrowth} are valid, and let 
$(u_n)\subset W^{2,p}(\Omega;\RR^M)\cap W^{1,p}_0(\Omega;\RR^M)$ be a bounded sequence such that $F(u_n)$ converges in $L^p(\Omega_{k};\RR^M)$ for every $k\in \NN$. 
Then $u_n$ has a subsequence (independent of $k$) which converges in 
$W^{2,p}(\Omega_{k-1};\RR^M)$ for every $k\geq 2$.
\end{prop}
\begin{proof}
It is enough to show that for every (fixed) $k\geq 2$, any subsequence $u_n^{(k-1)}$ of $u_n$ 
has another subsequence $u^{(k)}_n$ which converges in $W^{2,p}(\Omega_{k-1};\RR^M)$, where $u_n^{(1)}:=u_n$. 
A suitable diagonal sequence $u^{(n)}_{m(n)}$ then has the asserted property. 
We omit the superscript in the following. Fix $k\geq 2$ and let $u$ denote the weak limit of $u_n$ in $W^{2,p}(\Omega;\RR^M)\cap W^{1,p}_0(\Omega;\RR^M)$ (if necessary, pass to a suitable subsequence first). 
We choose a function $\eta\in C_0^\infty(\RR^N)$ such that $\eta=1$ on $\overline{B}_{k-1}(0)$ and $\eta=0$ on $\RR^N\setminus \overline{B}_{k-\frac{3}{4}}(0)$. 
Using \eqref{C2Omega}, we infer that $\eta\cdot u_n\in W^{1,p}_0(\Omega_k;\RR^M)$ for every $n$. Moreover,
$\eta u_n\rightharpoonup \eta u$ weakly in $W^{2,p}(\Omega_k;\RR^M)\cap W^{1,p}_0(\Omega_k;\RR^M)$.
By compact embedding we may assume that $u_n\to u$ in $C^1(\overline{\Omega}_k;\RR^M)\cap L^p(\Omega_k;\RR^M)$
and $\eta(u_n-u)\to 0$ in $L^p(\Omega_k;\RR^M)$ with respect to the strong topologies. 
Due to Lemma~\ref{lem:2Aapriori} (with $w:=\eta(u_n-u)$ and $v:=u$), 
it now suffices to show that $\tilde{A}(u)[\eta(u_n-u)]\to 0$ strongly in $L^p(\Omega_k;\RR^M)$. 
Since $u_n\to u$ in $C^1(\overline{\Omega}_k;\RR^M)$, we have that
$A_{\alpha\beta}(\cdot,u_n,\nabla u_n)\to A_{\alpha\beta}(\cdot,u,\nabla u)$ in 
$L^\infty(\Omega_k;\RR^{M\times M})$, whence
$(\tilde{A}(u_n)-\tilde{A}(u))u_n\to 0$ in $L^p(\Omega_k;\RR^M)$. Moreover,
$\tilde{b}(u_n,\nabla u_n)=b(\cdot,u_n,\nabla u_n)$ has a strong limit in $L^p(\Omega_k;\RR^M)$ 
by continuity of the Nemytskii operator $\tilde{b}$.
Since $\tilde{A}(u)u_n=F(u_n)-(\tilde{A}(u_n)-\tilde{A}(u))
u_n-\tilde{b}(u_n,\nabla u_n)$, we infer that $\tilde{A}(u)u_n$ converges strongly in $L^p(\Omega_k;\RR^M)$, 
which in turn implies that the same holds for $\eta\tilde{A}(u)[u_n-u]$. A direct calculation shows that 
$\eta\tilde{A}(u)[u_n-u]-\tilde{A}(u)[\eta(u_n-u)]$ only contains derivatives up to first order of $u_n$, 
which entails that this difference converges strongly in $L^p(\Omega_k;\RR^M)$, too.
Altogether, we infer that $\tilde{A}(u)[\eta(u_n-u)]$ has a strong limit in $L^p(\Omega_k;\RR^M)$. 
This limit is zero since $\eta(u_n-u)\rightharpoonup 0$ weakly in $W^{2,p}(\Omega_k;\RR^M)$.
\end{proof}
\begin{proof}[Proof of Corollary~\ref{cor:2Fprop}]
Obviously, (i) implies (iii), and (iii) implies (ii) as mentioned above. 
It remains to show that (ii) entails (i). By assumption, it is enough to verify statement (ii) of Theorem~\ref{thm:2Fprop} in case (a). Consider a bounded sequence such that $F(u_n)$ converges in $L^p(\Omega)$ and such that $u_n$ is tight in $W^{2,p}(\Omega;\RR^M)$. 
We have to show that $u_n$ has a subsequence which converges in $W^{2,p}(\Omega;\RR^M)$. 
Due to Proposition~\ref{prop:2Alocconv}, we already know that
a suitable subsequence (not relabeled) converges in $W^{2,p}(\Omega_k;\RR^M)$ for every $k\in \NN$.
Let $\eps>0$ and let $u$ denote the weak limit 
of $u_n$ in $W^{2,p}(\Omega;\RR^M)\cap W_0^{2,p}(\Omega;\RR^M)$ 
and thus also the strong limit in $W^{2,p}(\Omega_k;\RR^M)$ for every $k\in \NN$. 
Since $u_n-u$ is tight in $W^{2,p}(\Omega;\RR^M)$, we have that
\begin{align*}
	\norm{u_n-u}^p_{W^{2,p}(\Omega;\RR^M)}&
	= \norm{u_n-u}^p_{W^{2,p}(\Omega\setminus \Omega_k;\RR^M)}
	+ \norm{u_n-u}^p_{W^{2,p}(\Omega_k;\RR^M)}\\
	& < \frac{\eps}{2} + \norm{u_n-u}^p_{W^{2,p}(\Omega_k;\RR^M)}
\end{align*}
uniformly in $n\in \NN$, for every $k\geq k_0$ if $k_0=k_0(\eps)$ is sufficiently large. 
With $k:=k_0$ fixed, we also get $\norm{u_n-u}^p_{W^{2,p}(\Omega_k;\RR^M)}<\frac{\eps}{2}$
for every $n\geq n_0=n_0(\eps,k_0)$, due to the convergence in $W^{2,p}(\Omega_k;\RR^M)$. Hence 
$\norm{u_n-u}^p_{W^{2,p}(\Omega;\RR^M)}
< \frac{\eps}{2} + \frac{\eps}{2}=\eps$
whenever $n\geq n_0$.
\end{proof}

\appendix
\section{Truncation of gradients on unbounded domains}
%
The results presented here form the basis for the definition of
the truncation operators $\phi_n^{(2)}$ and $\phi_n^{(3)}$ in Section~\ref{secCOOpExamples}.
So-called maximal operators play an important role:
\begin{defn}[Maximal operator]\label{def:maxop}
For $v\in L^1_{\loc}(\RR^N)$ and $x\in \RR^N$ let 
\begin{align*}
 (\cM(v))(x):=\sup_{r\in (0,\infty)}\frac{1}{\abs{B_r(x)}}\int_{B_r(x)}\abs{v}\,dy.
 \end{align*}
\end{defn}
We recall a basic property of $\cal M$:
\begin{lem}[e.g., \cite{Stei70B}]\label{lem:MaxOpcont}
For $p\in(1, \infty]$, $\cM$ maps $L^p(\RR^N)$ into $L^p(\RR^N)$, and
\begin{align} \label{Mopbounded}
	\norm{\cM(v)}_{L^p(\RR^N)}\leq C \norm{v}_{L^p(\RR^N)}
\end{align}
for every $v\in L^p(\RR^N)$, where $C=C(N,p)$ is a constant.
\end{lem}
\begin{rem}\label{rem:truncgradp1}
Lemma~\ref{lem:MaxOpcont} does not hold for $p=1$. More precisely, 
$\cM(v)$ is integrable on $\RR^N$ if and only if $v=0$ a.e.. 
Moreover, \eqref{Mopbounded} does not hold for $p=1$ with a uniform constant $C$ even if 
the norm on the left hand side is taken over some bounded domain instead of $\RR^N$.
\end{rem}
In case $\Omega=\RR^N$, a scalar function can be truncated as follows:
\begin{lem}[The case $\Omega=\RR^N$, e.g., \cite{EvGa92B}]
\label{lem:hcutoffrN}
Let $p\in [1,\infty)$ and $u\in W^{1,p}(\RR^N)$. For every $\lam>0$, there exists a function $\bar{u}\in W^{1,p}(\RR^N)\cap W^{1,\infty}(\RR^N)$ such that
\begin{align}
	{}&
	\abs{\bar{u}(x)}+\abs{\nabla \bar{u}(x)}\leq \bar{C} \lam\quad\text{for a.e.}~x\in\RR^N,
	\label{phcutoffrNbound}\\
	{}&
	u(x)=\bar{u}(x)\quad \text{for a.e.}~x\in R^\lam,~~\text{and}
	\label{phcutoffrNsupp}\\
	{}&	
	\abs{\RR^N\setminus R^\lam}\leq \bar{C} \frac{1}{\lam} 
	\int_{\left\{\abs{u}+\abs{\nabla u}>\frac{\lam}{2}\right\}}\abs{u}+\abs{\nabla u}\,dx.
	\label{phcutoffrNR}
\end{align}
Here,
\begin{align*}
	R^\lam:=\left\{x\in \RR^N~\left|~ \cM(\abs{u}+\abs{\nabla u})(x)\leq \lam\right\}\right.,
\end{align*}
and $\bar{C}=\bar{C}(N)\geq 1$ is a constant independent of $u$, $\lam$ and $p$. 
Furthermore,
\begin{align} \label{phcutoffrNbaruMineq}
	\abs{\bar{u}(x)}+\abs{\nabla \bar{u}(x)}
	\leq \bar{C}\min\{\lam,\cM(\abs{u}+\abs{\nabla u})(x)\}~\text{for a.e.~$x\in \RR^N$}.
\end{align}
Both $\overline{u}$ and $R^\lam$ depend on $u$ and $\lam$, but not on $p$.
\end{lem}
\begin{proof}
See \cite{EvGa92B}, Theorem 3 in Section 6.6.3. The estimate for $\abs{\RR^N\setminus R^\lam}$ stated there is slightly weaker than~\eqref{phcutoffrNR}, but~\eqref{phcutoffrNR} is shown along the way in the proof. 
The last assertion~\eqref{phcutoffrNbaruMineq} is a consequence of \eqref{phcutoffrNbound}, the definition of $\cM$, the definition of $R^\lam$ and \eqref{phcutoffrNsupp}.
\end{proof}
We need an analogous result for $u\in W_0^{1,p}(\Omega)$ with $\Omega\neq \RR^N$. The main difficulty arises from the fact that the support of $\bar{u}$ is possibly larger than that of $u$. In particular, the "bad set" $\RR^N\setminus R^\lam$ is not entirely contained in $\Omega$ if the gradient of $u$ is large very close to the boundary. 
For bounded $\Omega$ 
this problem can be overcome as shown in \cite{DoHuMue00a} (Lemma 4.1); see also 
the appendix of \cite{FrieJaMue02a}. 
The proof therein works for some unbounded domains as well, but not without
imposing restrictions on the global shape of $\Omega$. 
The crucial assumption in \cite{DoHuMue00a} is that for every $x\in \Omega$, $\abs{B_{2r}(x)\setminus \Omega}\geq c r^N$, where $r:=\dist{x}{\RR^N\setminus \Omega}$ and $c>0$ is a constant 
independent of $x$.
For instance,
exterior domains (i.e., domains with bounded complement in $\RR^N$) are not admissible. 
By contrast, our approach below admits arbitrary domains with the following boundary regularity. 
\begin{defn}[Lipschitz domain]\label{def:Lipschitzdom}
We call $\Omega\subset \RR^N$ a \emph{Lipschitz domain}, if it is open and connected and if there exist constants $\varrho>0$, $L\geq 1$ and $K_0\in\NN$, a countable (or finite) set $Y\subset \partial\Omega$, 
a family of radii $r(y)>0$ and a family of bijective maps $\theta_y: \RR^N \supset B_{2r(y)}(y)\to B_{2r(y)}(0)\subset \RR^N$ ($y\in Y$) such that 
\begin{align*}
  & \text{both $\theta_y$ and $\theta_y^{-1}$ are Lipschitz continuous with Lipschitz constant $L$,}\\
	& \theta_y(\Omega \cap B_{2r(y)}(y))=\{x=(x_1,\ldots,x_N)\in B_{2r}(0)\mid x_1>0\},\\
	& \overline{\Omega}\cap B_{\varrho}(z)\subset \bigcup{}_{y\in Y} B_{r(y)}(y)~~\text{for each $z\in \partial \Omega$, and}\\
	& \text{for every $x\in \RR^N$, the set $\{y\in Y \mid x\in B_{r(y)}(y)\}$ has at most $K_0$ elements.}
\end{align*}
\end{defn}
Note that $\Omega$ may be unbounded, but $\varrho$ and $L$ are required to be uniform. 
For such domains, Lemma~\ref{lem:hcutoffrN} generalizes to
\begin{thm}[Truncation of gradients of scalar functions]\label{thm:hcutoffOmega}
Let $\Omega\subset \RR^N$ be a Lipschitz domain, $p\in [1,\infty)$ 
and $u\in W_0^{1,p}(\Omega)$. 
For every $\lam >0$, there exists a function $\phi_\lam(u)\in W_0^{1,p}(\Omega)\cap W^{1,\infty}(\Omega)$ and a measurable set $\hat{R}^\lam=\hat{R}^\lam(u)\subset \Omega$ with the 
following properties:
\begin{align}
	{}&
	\abs{\nabla \phi_\lam(u)(x)}+\abs{\phi_\lam(u)(x)}\leq C_0 \lam\quad\text{for a.e.}~x\in\Omega,
	\label{phcutoffomegaubound}\\
	{}&
	u(x)=\phi_\lam(u)(x)\quad \text{for a.e.}~x\in \hat{R}^\lam,
	\label{phcutoffomegausupp}\\
	{}&	
	\abs{\Omega \setminus \hat{R}^\lam}\leq C_1 \abs{\RR^N\setminus R^\lam}
	\leq C_2 \frac{1}{\lam} \int_{\left\{\abs{u}+\abs{\nabla u}>\frac{\lam}{2}\right\}}
	\abs{u}+\abs{\nabla u}\,dx
	\label{phcutoffomegaRhat0}
\end{align}
and 
\begin{equation}
	{}\begin{aligned}
	&\abs{\left\{x\in \Omega\mid\abs{\phi_\lam(u)(x)}+\abs{\nabla \phi_\lam(u)(x)}>\delta\right\}}\\
	&~ \leq C_1 \abs{\left\{x\in \RR^N\mid C_0
	\min\{\lam,\cM(\abs{u}+\abs{\nabla u})(x)\}>\delta\right\}}
	~ \text{for every $\delta\geq 0$}.
	\end{aligned}\label{phcutoffomegauhatMuineq}
\end{equation}
Here, $C_i=C_i(N,\varrho,L,K_0)\geq 1$ ($i=0,1,2$) are constants independent of $u$, $\lam$ and $p$
(where $\varrho,L,K_0$ are the constants of the Lipschitz domain $\Omega$), and
\begin{align*}
	R^\lam=R^\lam(u):=\left\{x\in \RR^N~\left|~ 
	\cM(\abs{u}+\abs{\nabla u})(x)\leq \lam\right\}\right.
\end{align*}
(for the application of $\cM$ extend $u$ with zero outside of $\Omega$). 
Both $\phi_\lam(u)$ and $\hat{R}^\lam(u)$ are independent of $p$. In particular, if $u$
also is an element of $W_0^{1,q}(\Omega)$ for some $q\in [1,\infty)$, then so is $\phi_\lam(u)$.
\end{thm}
Several consequences of \eqref{phcutoffomegaRhat0} are also useful:
\begin{cor}\label{cor:hcoOmega}
In the situation of Theorem~\ref{thm:hcutoffOmega}, we have in addition that
\begin{equation}
\begin{aligned}
	\abs{\Omega \setminus \hat{R}^\lam(u)}
	&~\leq~	(C_3)^r \frac{1}{\lam^r} \int_{\left\{\abs{u}+\abs{\nabla u}>\frac{\lam}{2}\right\}}
	\abs{u}^r+\abs{\nabla u}^r\,dx
\end{aligned}\label{phcutoffomegaRhat}
\end{equation}
for every $\lam>0$ and $r\in [1,\infty)$, with 
$C_3:=4 C_2$ (not necessarily optimal). Moreover,
\begin{equation}
\begin{aligned}
  \abs{\Omega \setminus \hat{R}^\lam(u)}
  ~\leq~	& (C_3)^r \frac{1}{\lam^{r}} 
  \int_{\left\{\frac{1}{2}\lam^{\beta}\geq \abs{u}+\abs{\nabla u}> \frac{\lam}{2}\right\}}
  \abs{u}^r+\abs{\nabla u}^r\,dx\\
	& + (C_3)^r \frac{1}{\lam^{1+\beta(r-1)}}
	\int_{\left\{\abs{u}+\abs{\nabla u}>\frac{1}{2}\lam^{\beta}\right\}}
	\abs{u}^r+\abs{\nabla u}^r\,dx
\end{aligned}\label{hcoOsmalllam}
\end{equation} 
for every $\lam\in (0,1]$, $r\in [1,\infty)$ and $\beta\in (0,1]$. 
In particular,
\begin{equation}\label{hcoOlim}
	\lam^p\abs{\Omega \setminus \hat{R}^\lam(u)}\to 0~~\text{both as $\lam\to\infty$ and as $\lam\to 0$}
\end{equation}
for arbitrary but fixed $u$.
\end{cor}
%
%
The proof of Corollary~\ref{cor:hcoOmega} is carried out first. In particular, we observe that \eqref{phcutoffomegaRhat0} implies \eqref{phcutoffomegaRhat}, 
a property also employed in the proof of Theorem~\ref{thm:hcutoffOmega}.
\begin{proof}[Proof of Corollary~\ref{cor:hcoOmega}]
Since $(C_3)^r=(4 C_2)^r\geq 2^{2r-1} C_2$ (for $C_2\geq 1$),
inequality \eqref{phcutoffomegaRhat} is a direct consequence of \eqref{phcutoffomegaRhat0} and 
the elementary estimate
\begin{align*}
	\abs{u}+\abs{\nabla u}\leq \frac{2^{r-1}}{\lam^{r-1}} (\abs{u}+\abs{\nabla u})^{r}\leq 
	\frac{2^{2r-1}}{\lam^{r-1}} (\abs{u}^r+\abs{\nabla u}^r) \quad
	\text{on}~\left\{\abs{u}+\abs{\nabla u}>\frac{\lam}{2}\right\}.
\end{align*}
By similar reasoning, \eqref{phcutoffomegaRhat0} also entails \eqref{hcoOsmalllam}.
Finally note that $1+\beta(p-1)<p$ if $\beta<1$ and that
\begin{align*}
 \int_{\left\{\abs{u}+\abs{\nabla u}>\frac{\lam}{2}\right\}}\abs{u}^p+\abs{\nabla u}^p\,dx
 \underset{\lam\to\infty}{\To}0~~\text{and}~~
 \int_{\left\{\frac{1}{2}\lam^\beta\geq \abs{u}+\abs{\nabla u}\right\}}\abs{u}^p+\abs{\nabla u}^p\,dx
 \underset{\lam\to 0}{\To}0
\end{align*}
for fixed $u$, whence \eqref{hcoOlim} follows from \eqref{phcutoffomegaRhat} and \eqref{hcoOsmalllam}, respectively.
\end{proof}
\begin{proof}[Proof of Theorem~\ref{thm:hcutoffOmega}]
Below, we abbreviate $\hat{u}=\phi_\lam(u)$, the function to be constructed.
In view of Lemma~\ref{lem:hcutoffrN}, we may assume that $\Omega\neq \RR^N$, and since $\Omega$ is a Lipschitz domain, we even have that $\abs{\RR^N\setminus \Omega}>0$.
In the following, $u$ is considered as an element of $W^{1,p}(\RR^N)$ by
extending it with zero outside of $\Omega$. 
In particular, this ensures that the definition of $R^\lam$ coincides with the one in Lemma~\ref{lem:hcutoffrN}.
Now let $\bar{u}\in W^{1,p}(\RR^N)$ denote the function obtained in Lemma~\ref{lem:hcutoffrN}. Recall that $\bar{u}$ satisfies $\abs{\bar{u}}+\abs{\nabla \bar{u}}\leq \bar{C}\lam$ a.e.~on $\RR^N$, for a suitable constant $\bar{C}\geq 1$ (cf.~\eqref{phcutoffrNbound}). For $x\in \RR^N$, we define 
\begin{align*}
	&\hat{u}(x):=
	\left\{\begin{array}{r@{\quad}l@{~~}r@{\,}l@{\,}l}
	h_\lam(x) & \text{if}&h_\lam(x)<&\bar{u}(x),&\\
	\bar{u}(x) & \text{if}&-h_\lam(x)\leq& \bar{u}(x) \leq& h_\lam(x),\\
	-h_\lam(x) & \text{if}&-h_\lam(x)>&\bar{u}(x),&
	\end{array}\right.
\intertext{where}	
	&h_\lam(x):=\bar{C}\lam L^2 \max\left\{1,\varrho^{-1}\right\} \dist{x}{\RR^N \setminus \Omega}.
\end{align*} 
Here, $\varrho>0$ and $L\geq 1$ are the constants of the Lipschitz domain $\Omega$ introduced in Definition~\ref{def:Lipschitzdom}. Note that $\hat{u}$ is independent of $p$ just as $\bar{u}$.
Accordingly, we define 
\begin{align*}
	\hat{R}^\lam:=R^\lam \cap \{x\in \Omega\mid -h_\lam(x)\leq \bar{u}(x) \leq h_\lam(x)\}.
\end{align*}
As an immediate consequence, (\ref{phcutoffrNsupp}) entails (\ref{phcutoffomegausupp}).
Also note that the function $h_\lam:\RR^N\to \RR$ is Lipschitz continuous with Lipschitz constant $\bar{C}L^2\max\left\{1,\varrho^{-1}\right\}\lam$. In particular, $h_\lam$ is weakly differentiable
with gradient $\nabla h_\lam\in L^{\infty}(\RR^N)$, and thus $\hat{u}\in W^{1,\infty}(\RR^N)$. 
Furthermore,
\begin{equation} 
\begin{aligned}\label{propcuomegabound}
	\abs{\hat{u}} \leq \bar{C}\lam ~~\text{and}~~
	\abs{\nabla \hat{u}}+\abs{\hat{u}}\leq C_0 \lam \quad\text{a.e.~in $\RR^N$}, &\\
	\text{where}~C_0:=\bar{C}+\bar{C}L^2\max\left\{1,\varrho^{-1}\right\}.&
\end{aligned}
\end{equation}
In particular, we obtain \eqref{phcutoffomegaubound}.
Assuming that \eqref{phcutoffomegaRhat0} is valid, so is \eqref{phcutoffomegaRhat},
as shown in the proof of Corollary~\ref{cor:hcoOmega}. 
In this case, $\phi_\lam(u)=\hat{u}$ is an element of $W^{1,p}(\RR^N)$:
both $\abs{\hat{u}}^p$ and $\abs{\nabla \hat{u}}^p$ are integrable on $\RR^N$, 
due to \eqref{phcutoffomegaRhat} and \eqref{phcutoffomegaubound}. 
Since $\hat{u}$ is continuous on $\RR^N$ (by the embedding $W^{1,\infty}(\RR^N)\subset C_B(\RR^N)$) and $\hat{u}=0$ on $\RR^N\setminus \Omega$, we infer that $\hat{u}\in W_0^{1,p}(\Omega)$.

It remains to show that~\eqref{phcutoffomegaRhat0} and~\eqref{phcutoffomegauhatMuineq} are satisfied.
Since 
\begin{align*}
	\Omega\setminus \hat{R}^\lam=(\Omega\setminus R^\lam)\cup (\Omega\cap \{\abs{\bar{u}}> h_\lam\}),
\end{align*}
\eqref{phcutoffomegaRhat0} follows provided that
\begin{align}\label{phcutoffomega1a}
	\abs{\Omega\cap \{\abs{\bar{u}}> h_\lam\}}~\leq~ C_1 \abs{(\RR^N\setminus \Omega)\cap (\RR^N\setminus R^\lam)},
\end{align} 
for a constant $C_1=C_1(N,\Omega)\geq 1$. 
Before we verify \eqref{phcutoffomega1a}, let us observe that \eqref{phcutoffomega1a} also implies \eqref{phcutoffomegauhatMuineq}:
It is enough to prove~\eqref{phcutoffomegauhatMuineq}
for every $\delta\in [0,C_0\lam)$
because $\{\abs{\hat{u}}+\abs{\nabla \hat{u}}>\delta\}$
is of measure zero if $\delta\geq C_0\lam$. 
Due to~\eqref{phcutoffrNbaruMineq},
\begin{align*}
	\abs{\hat{u}(x)}+\abs{\nabla \hat{u}(x)}
	=\abs{\bar{u}(x)}+\abs{\nabla \bar{u}(x)}
	\leq C_0\cM(\abs{u}+\abs{\nabla u})(x) 
\end{align*}
for a.e.~$x\in \Omega$ such that $\abs{\bar{u}(x)}\leq h_\lam(x)$, 
whence
\begin{align}
&\begin{aligned}
	{}&\abs{\Omega\cap \{\abs{\bar{u}}\leq h_\lam\} \cap \left\{
	\abs{\hat{u}}+\abs{\nabla \hat{u}}>\delta\right\}}\\
	{}&\qquad\qquad\qquad \leq~ \abs{\Omega \cap \left\{C_0
	\cM(\abs{u}+\abs{\nabla u})>\delta\right\}}
\end{aligned}\label{phcutoffomega1b1}
\intertext{for every $\delta>0$. On the other hand, 
\eqref{phcutoffomega1a} entails that}
&\begin{aligned}
	{}&\abs{\Omega\cap \{\abs{\bar{u}}>h_\lam\} \cap \left\{
	\abs{\hat{u}}+\abs{\nabla \hat{u}}>\delta\right\}}\\
	{}&\qquad\qquad\qquad \leq~ C_1\abs{(\RR^N\setminus \Omega) \cap \left\{C_0
	\cM(\abs{u}+\abs{\nabla u})>\delta\right\}}
\end{aligned}\label{phcutoffomega1b2}
\end{align}
for every $\delta\in [0,C_0\lam)$ since in that case,
\begin{align*}
	\RR^N\setminus R^\lam =\left\{ \cM(\abs{u}+\abs{\nabla u})>\lam \right\}
	\subset 
	\left\{ C_0\cM(\abs{u}+\abs{\nabla u})>\delta \right\}.
\end{align*}
Summing \eqref{phcutoffomega1b1} and \eqref{phcutoffomega1b2} yields \eqref{phcutoffomegauhatMuineq}
for every $\delta\in [0,C_0\lam)$. 

The proof of~\eqref{phcutoffomega1a} is carried out in three steps:

{\bf(i) A local estimate for an affine piece of the boundary}\\
We consider the following local situation: Assume that for a $r>0$,
\begin{align*}
	B_{2r}(0)\cap \Omega~=~\{x=(x_1,\ldots,x_N)\in B_{2r}(0)\mid x_1>0\},
\end{align*} 
and let $\rho_1$ denote the reflection with respect to the hyperplane $\{x_1=0\}$, i.e.,
$x=(x_1,x_2,\ldots,x_N)\mapsto \rho_1 x:=(-x_1,x_2,\ldots,x_N)$.
We claim that in this case
\begin{align}\label{phcutoffomega2a0}
	\rho_1\left(B_r(0)\cap \Omega\cap \{\abs{\bar{u}}> L^{-2}h_\lam\}\right)
	~\subset~ (B_r(0) \setminus \Omega)\cap (B_r(0) \setminus R^\lam), 
\end{align}
at least up to a set of measure zero which we ignore.
To see this, consider an arbitrary $x\in B_r(0)\cap \Omega$ such that $\abs{\bar{u}(x)}> L^{-2}h_\lam(x)$.
In particular, $\abs{\bar{u}(x)}> 2\bar{C}\lam\dist{x}{\partial \Omega}$ due to the definition of $h_\lam$.
On the other hand, 
\begin{align*}
	\abs{\bar{u}(x)-\bar{u}(\rho_1 x)}~\leq~ \bar{C}\lam\abs{x-\rho_1 x}
	~=~2\bar{C}\lam\dist{x}{\partial \Omega},
\end{align*}
whence
$\bar{u}(\rho_1 x)\neq 0=u(\rho_1 x)$. Thus
the reflection $\rho_1 x$ of $x$ is not an element of $R^\lam$ 
for almost every $x\in B_r(0)\cap \Omega\cap \{\abs{\bar{u}}> L^{-2}h_\lam\}$, 
which shows~\eqref{phcutoffomega2a0}. 
In particular, we have the estimate
\begin{align}\label{phcutoffomega2a}
	\abs{B_r(0)\cap \Omega\cap \{\abs{\bar{u}}> L^{-2}h_\lam\}}
	~\leq~ \abs{(B_r(0) \setminus \Omega)\cap (B_r(0) \setminus R^\lam)},
\end{align} 
since the Lebesgue measure is invariant under $\rho_1$.
 
{\bf (ii) A local estimate for Lipschitz boundary}\\
The estimate analogous to~\eqref{phcutoffomega2a} for a general piece of Lipschitz boundary is obtained by using the local maps $\theta_y$ ($y\in Y\subset \partial \Omega$) of Definition~\ref{def:Lipschitzdom} to reduce the problem to the special case discussed in (i): For every $y\in Y$,
\begin{align}\label{phcutoffomega3a}
	\abs{B_r(y)\cap \Omega\cap \{\abs{\bar{u}}> h_\lam\}}\leq L^{2N} 
	\abs{(B_r(y) \setminus \Omega)\cap (B_r(y) \setminus R^\lam)},
\end{align} 
where $r=r(y)>0$ and $L\geq 1$ are the constants of the Lipschitz domain $\Omega$.

{\bf (iii) Proof of the global estimate~\eqref{phcutoffomega1a} by a covering argument}\\
Since $\abs{\bar{u}}\leq \bar{C}\lam$ a.e.~in $\RR^N$, the definition of $h_\lam$ and the properties of the Lipschitz domain imply that
\begin{align*}
	\Omega\cap \{\abs{\bar{u}}> h_\lam\}~& \subset~\Omega\cap \{\bar{C}\lam> h_\lam\}~\\
	&\subset~\Omega\cap \bigcup_{z\in \partial \Omega} B_{\varrho}(z)~\subset~
	\Omega\cap \bigcup_{y\in Y}  B_{r(y)}(y),
\end{align*}
for every $\lam>0$. 
Applying~\eqref{phcutoffomega3a} at every $y\in Y$ and summing over $y$ yields
\eqref{phcutoffomega1a}, with $C_1:=K_0 L^{2N}$. 
\end{proof}

\bibliographystyle{plain}
\bibliography{07_bib}

\end{document}